\documentclass[11pt,letterpaper]{amsart}
\usepackage[svgnames]{xcolor}
\usepackage{amsmath, fullpage, amssymb, amsthm, amsfonts, color, mathrsfs, moreverb, mathtools, float}
\usepackage{graphicx}
\usepackage{natbib}
\usepackage{xspace}
\usepackage{mathalfa}
\usepackage[mathscr]{euscript}
\usepackage{pstricks}
\usepackage{enumerate}
\usepackage[colorlinks,unicode,
            linkcolor=FireBrick,
            anchorcolor=red,
            citecolor=RoyalBlue
            ]{hyperref}
\usepackage{mathtools}

\allowdisplaybreaks[1]

\DeclareMathOperator{\R}{\mathbb{R}} % Real numbers
\renewcommand{\P}{\mathbb{P}} % Probability P
\newcommand{\E}{\mathbb{E}} % Expectation E 
\newcommand{\1}{\mathbf{1}}
\newcommand{\onealpha}{\1_{\alpha}}
\newcommand{\onemalpha}{\1_{-\alpha}}

\newcommand{\Ker}{\mathrm{Ker}}
\newcommand{\Img}{\mathrm{Im}}

\newcommand{\nedge}{\mathscr{E}}
\newcommand{\nvertex}{\mathscr{V}}
\newcommand{\nvright}{\mathscr{R}}

\newcommand{\Gg}{\mathcal{G}(n,m,d_1,d_2)}
\newtheorem{theorem}{Theorem}
\newtheorem{lemma}[theorem]{Lemma}
\newtheorem{proposition}[theorem]{Proposition}

\newtheorem{corollary}[theorem]{Corollary}

\newtheorem{definition}{Definition}

\newtheorem*{remark}{Remark}
\newcommand\kframe{random regular frame graph\xspace}

\newcommand{\ER}{Erd\H{o}s-R\'enyi\xspace}

\begin{document}
\title{Spectral gap in random bipartite biregular graphs and applications}
\author{Gerandy Brito}
\address{Georgia Institute of Technology, College of Computing}
\email{gbrito3@gatech.edu}
\author{Ioana Dumitriu}
\address{University of California San Diego, Mathematics}
\email{idumitriu@ucsd.edu}

\author{Kameron Decker Harris}
\address{University of Washington, Paul G.\ Allen School of Computer Science \& Engineering 
and Biology}
\curraddr{Western Washington University, Computer Science}
\email{kameron.harris@wwu.edu}

\date{\today}

%%%%%%%%%%%%%%%%
%%% Abstract %%%
%%%%%%%%%%%%%%%%

\begin{abstract}
We prove an analogue of Alon's spectral gap conjecture 
for random bipartite, biregular graphs.
We use the Ihara-Bass formula to connect the non-backtracking
spectrum to that of the adjacency matrix, 
employing the moment method to show there exists a spectral gap
for the non-backtracking matrix.
A byproduct of our main theorem is that random rectangular 
zero-one matrices with fixed row and column sums
are full-rank with high probability.
Finally, we illustrate applications to
community detection, coding theory, and
deterministic matrix completion.
\end{abstract}

\maketitle

%%%%%%%%%%%%%%%%
%%% Document %%%
%%%%%%%%%%%%%%%%

%%%%%%%%%%%%%%%%%%%%%%%%%%%%%%%%%%%%%%%%%%%%

\section{Introduction} \label{sec:intro}
%%%%%%%%%%%%%%%%%%%%%%%%%%%%%%%%%%%%%%%%%%%%

Random regular graphs,
where each vertex has the same degree $d$,
are among the most well-known examples of {\it expanders}:
graphs with high connectivity and which exhibit rapid mixing.
Expanders are of particular interest in computer science, 
from sampling and complexity theory to design of error-correcting codes.
For an extensive review of their applications, see \cite*{hoory2006}.
What makes random regular graphs particularly interesting expanders 
is the fact that they exhibit all three existing
types of expansion properties: edge, vertex, and spectral. 

The study of regular random graphs took off with the work of
\citet*{bender1974}, \citet*{bender1978},
\citet*{bollobas1980},
and slightly later \citet*{mckay1984} and \citet*{wormald1981}. 
Most often, their expanding properties are described 
in terms of the existence of the \emph{spectral gap}, which we define below. 

Let $A$ be the adjacency matrix of a simple graph,
where $A_{ij} = 1$ if $i$ and $j$ are connected and zero otherwise.
Denote
$\sigma(A) = \{\lambda_1 \geq \lambda_2 \geq \ldots \}$
as its spectrum. 
For a random $d$-regular graph, $\lambda_1 = \max_{i} |\lambda_i| = d$, 
but the second largest eigenvalue
$\eta = \max (|\lambda_2|, |\lambda_{n}| )$ 
is asymptoticly almost surely 
of much smaller order, leading to a spectral gap.
Note that we will always use $\eta$ to be the second largest 
eigenvalue of the adjacency matrix $A$.
For a list of important symbols see Appendix~\ref{sec:list_of_symbols}.

Spectral expansion properties of a graph are, strictly speaking,
defined with respect to the smallest nonzero eigenvalue of the normalized Laplacian,
$\mathcal{L} = I - D^{-1/2} A D^{-1/2}$, 
where $I$ is the identity and $D$ is the diagonal matrix of vertex degrees.
In the case of a $d$-regular graph, $\sigma(\mathcal{L})$ is a scaled
and shifted version of $\sigma(A)$.
Thus, a spectral gap for $A$ translates directly into one for $\mathcal{L}$.

The study of the second largest eigenvalue in regular graphs 
had a first breakthrough in the Alon-Boppana bound 
\citet*{alon1986},
which states that the second largest eigenvalue satisfies
\[
 \eta \geq 2 \sqrt{d-1} - \frac{c_d}{\log n} .
\]
Graphs for which the Alon-Boppana bound is attained are called {\it Ramanujan}.
\citet*{friedman2003a} proved the conjecture of \citet*{alon1986} that
almost all $d$-regular graphs have $\eta \leq 2 \sqrt{d-1} + \epsilon$
for any $\epsilon > 0$
with high probability as the number of vertices goes to infinity. This result was simultaneous simplified and deepened in \citet*{friedman2014}. 
More recently, \citet*{bordenave2015} gave a different proof that
$\eta \leq 2 \sqrt{d-1}+\epsilon_n$ 
for a sequence $\epsilon_n \rightarrow 0$ as $n$, 
the number of vertices, tends to infinity; the new proof is based on the non-backtracking operator and the Ihara-Bass identity.  

\subsection{Bipartite biregular model}

In this paper we prove the analog of Friedman and Bordenave's result 
for bipartite, biregular random graphs. 
These are graphs for which the vertex set partitions into two independent
sets $V_1$ and $V_2$, such that all edges occur between the sets.
In addition, all vertices in set $V_i$ have the same degree $d_i$.
See Figure~\ref{fig:bipartite_schematic} for a schematic of such a graph.
Along the way, we also bound the smallest positive eigenvalue and
the rank of the adjacency matrix.

Let $\mathcal{G}(n,m,d_1,d_2)$ be the uniform distribution of
simple, bipartite, biregular random graphs. 
Any $G \sim \mathcal{G}(n,m,d_1,d_2)$ 
is sampled uniformly from the set of simple bipartite graphs 
with vertex set $V=V_1 \bigcup V_2$, 
with $|V_1|=n$, $|V_2|=m$ 
and where every vertex in $V_i$ has degree $d_i$. 
Note that we must have $n d_1 = m d_2 = |E|$.
Without any loss of generality, we will assume $n \leq m$
and thus $d_1 \geq d_2$
when necessary.
Sometimes we will write that $G$ is a $(d_1, d_2)$-regular graph, 
when we want to explicitly state the degrees.
Let $X$ be the $n \times m$ matrix with entries $X_{ij}=1$ 
if and only if there is an edge between vertices 
$i \in V_1$ and $j \in V_2$. 
Using the block form of the adjacency matrix
\begin{equation}
\label{eq:matrix_blocks}
A=\left(\begin{array}{cc} 0 & X \\
X^* & 0 \end{array}\right ),
\end{equation}
It is well known that $\mathcal{G}(n,m,d_1,d_2)$ is connected with high probability, as long as $d_i\geq 3$. From \eqref{eq:matrix_blocks}, it can be verified that all eigenvalues of $A$
occur in pairs $\lambda$ and $-\lambda$,
where $|\lambda|$ is a singular value of $X$,
along with at least $|n - m|$ zero eigenvalues.
For these reasons, the second largest eigenvalue is
$\eta = \lambda_2(A) = -\lambda_{n+m-1} (A)$.
Furthermore,  
the leading or Perron 
eigenvalue of $A$ is always $\sqrt{d_1 d_2}$, matched to the left by $-\sqrt{d_1d_2}$,
which reduces to the result for $d$-regular when $d_1 = d_2$.
% and assuming an eigenvector of the form 
% $v=(v_1,v_2)$ with $v_1\in \R^n$, $v_2\in \R^m$,
% with eigenvalue $\lambda \in \sigma(A)$,
% we see that 
% \begin{eqnarray*}
% X   v_2 &= \lambda v_1 \\ 
% X^* v_1 &= \lambda v_2
% \end{eqnarray*}
% which imply another system of equations
% \begin{eqnarray*}
% (X^* X) v_2 &= \lambda X^* v_1 = \lambda^2 v_2 \\ 
% (X X^*) v_1 &= \lambda X   v_2 = \lambda^2  v_1.
% \end{eqnarray*}
% We observe that the eigenvalues of $X X^*$ and $X^* X$ are just $\lambda^2$, 
% where $\lambda \in \sigma(A)$,
% so the singular values of $X$ and $X^T$ are just the eigenvalues $| \lambda |$ 
% and possibly zero.
% Furthermore, taking $v_1, v_2$ proportional to the all-ones vector, 
% the leading eigenvalue of $A$ is clearly $\sqrt{d_1 d_2}$,
% since every row of $X X^*$ and $X^* X$ has $d_1 d_2$ entries.

We will focus on the spectrum of the adjacency matrix.
Similar to the case of the $d$-regular graph, 
in the bipartite, biregular graph, 
the spectrum of the normalized Laplacian is a scaled and shifted version
of the adjacency matrix:
Because of the structure of the graph,
$D^{-1/2} A D^{-1/2} = \frac{1}{\sqrt{d_1 d_2}} A$.
Therefore, a spectral gap for $A$ again implies that one exists for $\mathcal{L}$.

\begin{figure}[t!]
\centering
\includegraphics[width=0.5\linewidth]{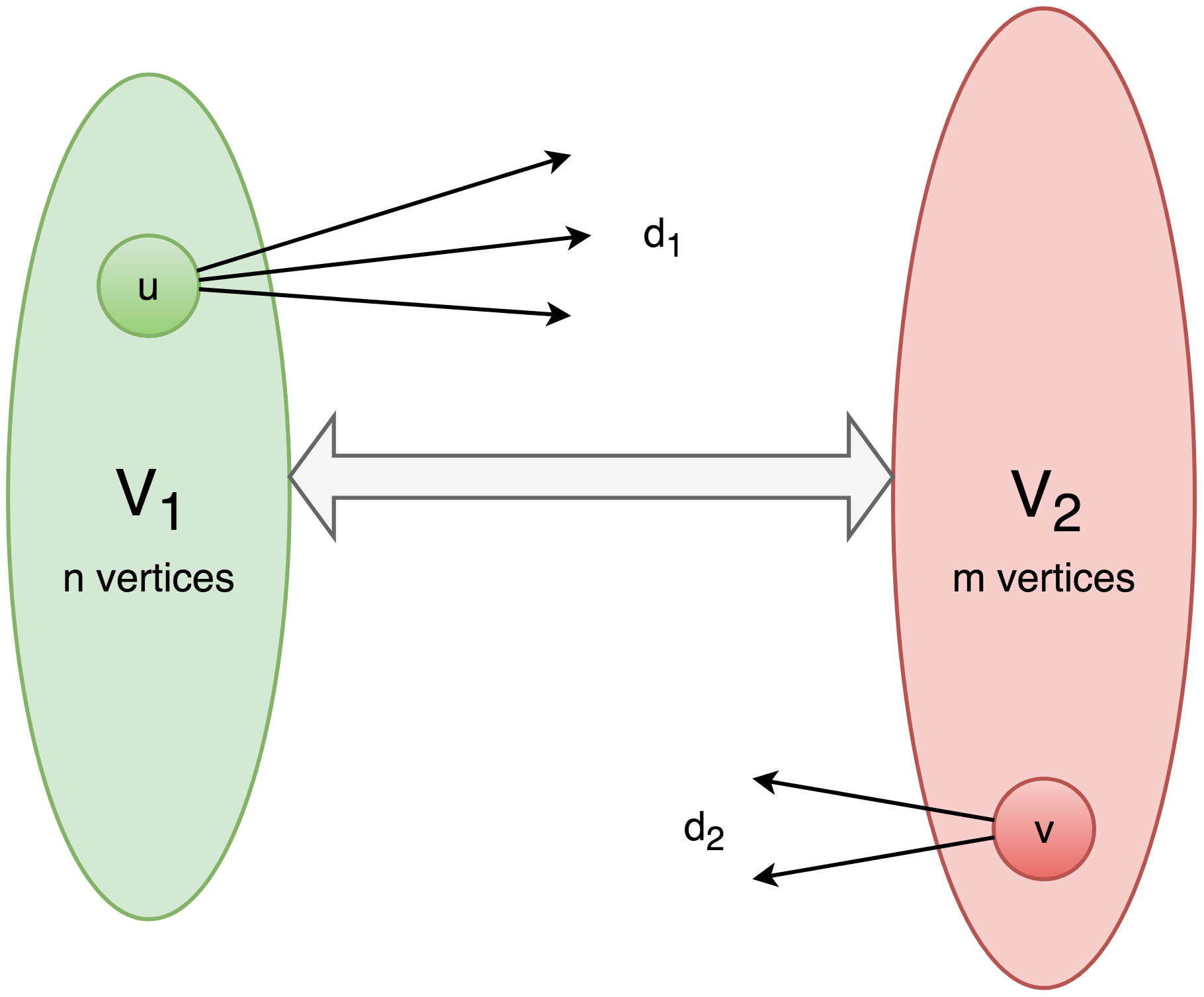}
\caption{The structure of a bipartite, biregular graph. 
	There are $n = |V_1|$ left vertices, $m=|V_2|$ right vertices,
    each of degree $d_1$ and $d_2$, with the constraint that
    $n d_1 = m d_2$.
	The distribution $\mathcal{G}(n,m,d_1,d_2)$ is taken uniformly over all such graphs.
    }
\label{fig:bipartite_schematic}
\end{figure}

Previous work on bipartite, biregular graphs
includes the work of \citet*{feng1996} and \citet*{li1996},
who proved the analog of the Alon-Boppana bound. 
For every $\epsilon >0$,
\begin{eqnarray}
\label{lower_bound}
\eta \geq \sqrt{d_1-1} + \sqrt{d_2 - 1} - \epsilon
\end{eqnarray}
as the number of vertices goes to infinity. 
This bound also follows immediately from the fact that 
the second largest eigenvalue cannot be asymptotically smaller 
than the right limit of the asymptotic 
support for the eigenvalue distribution,
which is $\sqrt{d_1-1} + \sqrt{d_2 - 1}$ and was first computed by
\citet*{godsil1988}.
They found the spectral measure 
$\mu(\lambda)$ has a point mass at $\lambda=0$ of size 
$\frac{1}{2} |d_1 - d_2| / (d_1 + d_2)$
and a continuous part given by the density

\begin{equation}
    \mathrm{d}\mu(\lambda) = 
    \frac{d_1 d_2 \sqrt{(- \lambda^2 + d_1 d_2 - (z - 1)^2)
                (\lambda^2 - d_1 d_2 + (z + 1)^2)}}
            {\pi (d_1 + d_2) (d_1 d_2 - \lambda^2) |\lambda|}~, 
    \label{eq:bipartite_esd}
\end{equation}
supported on 
$$|\sqrt{d_1-1} - \sqrt{d_2 - 1}| \leq |\lambda| \leq \sqrt{d_1-1} + \sqrt{d_2 - 1}~,~~\mbox{
where} ~~z=\sqrt{(d_1 - 1)(d_2 - 1)}~.$$
% As we will show here, this also follows from the Ihara-Bass formula
% \citep*{bass1992,kotani2000}, see section \ref{sec: connecting} \textcolor{red}{I don't have a proof of this from the Ihara-Bass formula}.

Graphs where $\eta$ attains the
Alon-Boppana bound, Eqn.~\eqref{lower_bound}, are also called Ramanujan.
Complete graphs are always Ramanujan but not sparse, 
whereas $d$-regular or bipartite $(d_1, d_2)$-regular graphs are sparse. 
Our results show that  
almost every $(d_1, d_2)$-regular graph is ``almost'' Ramanujan.

Beyond the first two eigenvalues, 
we should mention that \citet*{bordenave2010}
studied the limiting spectral distribution
of large sparse graphs.
They obtained a set of two coupled equations
that can be solved for the eigenvalue distribution of any
$(d_1, d_2)$-regular random graph. 
The solution of the coupled equations for fixed $d_1$ and $d_2$ shows convergence of the spectral distribution of a random regular bipartite graph to the Mar\v{c}enko-Pastur law. 
This was first observed by \citet*{godsil1988}. 
For $d_1, d_2 \to \infty$ with
$d_1/d_2$ converging to a constant, \citet*{dumitriu2016} showed that the limiting spectral distribution converges to a transformed version of the Mar\v{c}enko-Pastur law. When $d_1 = d_2 = d$, this is equal to the
Kesten-McKay distribution (\citet*{mckay1981}),
which becomes the 
semicircular law as $d \to \infty$
(\citet*{godsil1988, dumitriu2016}). 
Notably, \citet*{mizuno2003} 
obtained the same results when they calculated the asymptotic
distribution of eigenvalues for bipartite, biregular graphs of high girth.
However, their results are not applicable to 
random bipartite biregular graphs as these 
asymptotically almost surely have low girth (\citet*{dumitriu2016}).

Our techniques borrow heavily from the results of \citet*{bordenave2015a}
and \citet*{bordenave2015},
who simplified the trace method of \citet*{friedman2003a}
by counting non-backtracking walks built up of segments with at most one cycle, and by relating the eigenvalues of the adjacency matrix to the eigenvalues of the non-backtracking one via the Ihara-Bass identity.
The combinatorial methods we use to bound the number of 
such walks are similar to how 
\citet*{brito2015} counted self-avoiding walks
in the context of community recovery in a regular stochastic block model.

Finally, we should mention that similar techniques have been employed by \citet*{coste2017} 
to study the spectral gap of the Markov matrix of a random directed multigraph. 
The non-backtracking operator of a bipartite biregular graph 
could be seen as the adjacency matrix of a directed multigraph, 
whose eigenvalues are a simple scaling away from the eigenvalues of the Markov matrix of the same. 
However, the block structure of our non-backtracking matrix means that 
the corresponding multigraph is bipartite, 
and this makes it different from the model used in \citet*{coste2017}. 

\subsection{Configuration versus random lift model}\label{sec:1.2}

Random lifts are a model that allows the construction of large,
random graphs by repeatedly lifting the vertices of a base graph
and permuting the endpoints of copied edges.
See
\citet*{bordenave2015} for a recent overview.
A number of spectral gap results have been obtained for random lift models,
e.g.\ \citet*{friedman2003a,angel2007,friedman2014}, and \citet*{bordenave2015}.

Random lift models are contiguous with the configuration model 
in very particular cases.
See Section~\ref{sec:configuration_model} for a definition of 
the configuration model; 
this is a useful substitute for the uniform model and is practically 
equivalent.
For even $d$, 
random $n$-lifts of a single vertex with $d/2$ self-loops
are equivalent to the $d$-regular 
configuration model.
For odd $d$, no equivalent lift construction is known
or even believed to exist.

For $(d_1, d_2)$-biregular, bipartite
graphs the situation is more complicated.
A celebrated result due to \citet*{marcus2013} 
showed the existence of infinite
families of $(d_1, d_2)$-regular bipartite graphs that are Ramanujan.
That is, with 
$\eta = \sqrt{d_1-1} + \sqrt{d_2 - 1}$
by taking repeated lifts of the complete bipartite graph on $d_1$ left and $d_2$ right vertices $K_{d_1,d_2}$.
If $d_1=d_2=d$, then the configuration model is contiguous to 
the random lift of the multigraph with two vertices and
$d$ edges connecting then.
Certainly, for a biregular bipartite graph with
$n/d_2=m/d_1=k$ not an integer,
we cannot construct it by lifting $K_{d_1, d_2}$
as considered by \citet*{marcus2013}.
But even for $k$ integer, it seems likely the 
two models are not contiguous,
for the reasons we now explain.

Suppose there were a base graph $G$ that could be lifted to produce
any $(3,2)$-biregular, bipartite graph.
Consider another graph $H$ which is
a union of 2 complete bipartite graphs $K_{2,3}$.
Then $H$ is a $(3,2)$-biregular, bipartite graph and occurs
in the configuration model with nonzero probability. 
The only $G$ that $H$ could be a lift of is $K_{2,3}$, 
because it is a disconnected union and 
$K_{2,3}$ itself is not a lift of any graph or multigraph
(note that $2+3=5$ is prime). 
Therefore, $G$ would have to be $K_{2,3}$.
Figure~\ref{fig:disallowed_graph}
shows an example of another graph $H'$ with
the same number of vertices as $H$ which is $(3,2)$-biregular, 
bipartite but is {\em not} a lift of $K_{2,3}$.
Now, $H$ and $H'$ both occur in the configuration model with equal, 
nonzero probability.
Therefore, we cannot construct {\em every example} of a
$(3,2)$-biregular, bipartite graph
by repeatedly lifting a {\em single} base graph $G$. 
\begin{figure}[t!]
    \centering
    \includegraphics[width=0.5\linewidth]{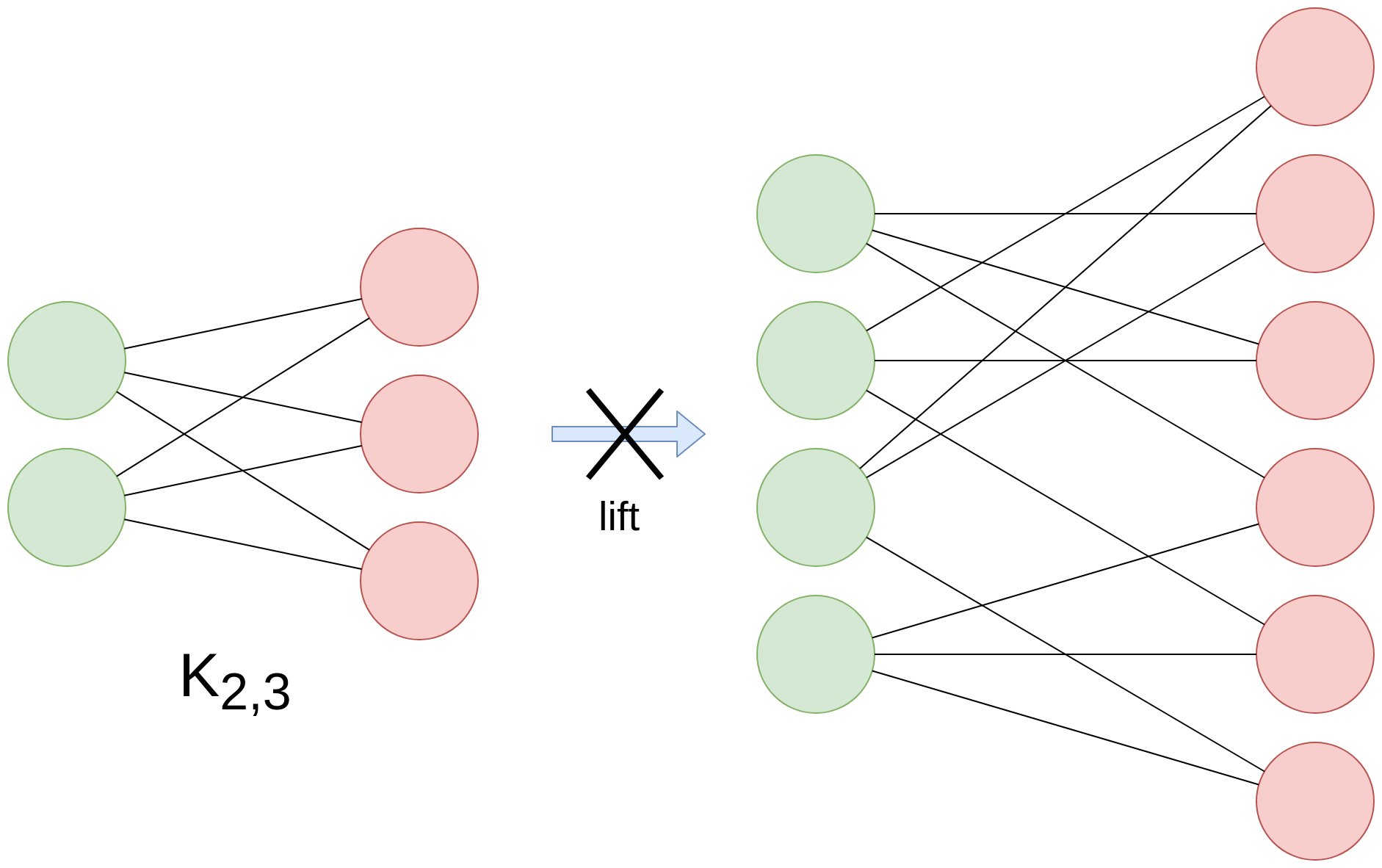}
    \caption{The $(3,2)$-bipartite, biregular 
        graph on the right is not a 2-lift of $K_{2,3}$.
        Every pair of left vertices shares a neighbor on the right.
    }
    \label{fig:disallowed_graph}
\end{figure}

Since the eventual goal of any argument based on lifts that also applies 
to the configuration model would have to show that almost all
bipartite, biregular graphs can be obtained by lifting and 
are sampled asymptotically uniformly from the lift model,
the above considerations suggest this argument would be 
highly non-trivial. 
We in fact doubt such an argument can be made.
Intuitively, in the configuration model edges occur
``nearly independently,'' 
whereas for random lifts there are strong dependencies 
due to the fact that many edges are not allowed;
see
\citet*{bordenave2015}.

\subsection{Structure of the paper}

Briefly, we now lay out the method of proof that the 
bipartite, biregular random graph is Ramanujan.
The proof outline is given in detail in Section~\ref{sec:proof_outline},
after some important preliminary terms and definitions given in 
Section~\ref{sec:preliminaries}.
The bulk of our work builds to Theorem~\ref{thm:gap_B},
which is actually a bound
on the second eigenvalue of the non-backtracking matrix $B$,
as explained in Section~\ref{sec:non_backtracking_matrix}.
The Ramanujan bound on the second eigenvalue of $A$ then follows 
as Theorem~\ref{cor:gap_A}.
As a side result, we find that row- and column-regular, rectangular 
matrices 
(the off-diagonal block $X$ of the adjacency matrix in Eqn.~\eqref{eq:matrix_blocks})
with aspect ratio smaller than one 
($d_1 \neq d_2$) have full rank with high probability.

To find the second eigenvalue of $B$,
we subtract from it a matrix $S$ that is formed from the leading eigenvectors,
and examine the spectral norm of the 
``almost-centered'' matrix $\bar{B} = B - S$.
We then proceed to use the trace method to bound 
the spectral norm of the matrix $\bar{B}^\ell$ by its trace.
However, since $\bar{B}$ is not positive definite, this leads us to consider
\[
\E\left(\|\bar{B}^{\ell}\|^{2k}\right)
\leq 
\E\left( \mathrm{Tr} \left( (\bar{B}^{\ell}) (\bar{B}^{\ell})^* \right)^{k} \right)~.
\]
On the right hand side, the terms in $\bar{B}^\ell$ refer to circuits built up of
$2k$ segments, each of length $\ell+1$,
since an entry $B_{ef}$ is a walk on two edges.
Because the degrees are bounded, it turns out that, for $\ell = O(\log (n))$,
the depth $\ell$ neighborhoods of every vertex contain at most one cycle---they are ``tangle-free.''
Thus, we can bound the trace by computing the expectation of the circuits that contribute,
along with an upper bound on their multiplicity, 
taking each segment to be $\ell$-tangle-free.

Finally, to demonstrate the usefulness of the spectral gap,
we highlight three applications of our bound.
In Section~\ref{sec:communities}, we show a community detection application. 
Finding communities in networks is important
for the areas of social network, bioinformatics,
neuroscience, among others.
Random graphs offer tractable models to study when detection and 
recovery are possible. 

We show here how our results lead to 
community detection in regular stochastic block models with 
arbitrary numbers of groups, using a very general theorem by
\citet*{wan2015}. 
Previously, \citet*{newman2014} studied the spectral density of such models,
and the community detection problem of the special case of 
two groups was previously studied by \citet*{brito2015} and \citet*{barucca2017}.

In Section~\ref{sec:codes}, 
we examine the application to linear
error correcting codes built from sparse expander graphs.
This concept was first introduced by \citet*{gallager1962}
who explicitly used random bipartite biregular graphs.
These ``low density parity check'' codes
enjoyed a renaissance in the 1990s, 
when people realized they were well-suited to modern computers.
For an overview, see \citet*{richardson2003,richardson2008}.
Our result yields an explicit lower bound on the minimum distance
of such codes, i.e.\ the number of errors that can be corrected.

The final application, in Section~\ref{sec:matrix_completion},
leads to generalized error bounds for matrix completion.
Matrix completion is the problem of reconstructing a matrix from
observations of a subset of entries.
\citet*{heiman2014} gave an algorithm for reconstruction of a
square matrix with low complexity as measured by
a norm $\gamma_2$, which is similar to the trace norm
(sum of the singular values,
also called the nuclear norm or Ky Fan $n$-norm).
The entries which are observed are at the nonzero entries 
of the adjacency matrix of a bipartite, biregular graph. 
The error of the reconstruction is bounded above by a factor which is proportional
to the ratio of the leading two eigenvalues, so that a graph with larger
spectral gap has a smaller generalization error.
We extend their results to rectangular graphs, 
along the way strengthening them by a constant factor of two. 
The main result of the paper gives an explicit bound in terms of $d_1$ and $d_2$.

As this paper was being prepared for submission, 
we became aware of the work of \citet*{deshpande2018}.
In their interesting paper, they use the {\it smallest} positive eigenvalue 
of a random bipartite lift to study convex relaxation techniques for random
not-all-equal-3SAT problems.
It seems that our main result addresses the
configuration model version of this constraint satisfaction problem,
the first open question listed at the end of \citet*{deshpande2018}.

\section{Non-backtracking matrix $B$}\label{sec2:B}

\label{sec:non_backtracking_matrix}

Given $G \sim \mathcal{G}(n,m,d_1,d_2)$, 
we define the non-backtracking operator $B$.
This operator is a linear endomorphism of $\R^{|\vec{E}|}$,
where $\vec{E}$ is the set of oriented edges of $G$ and $|\vec{E}| = 2|E|$.
Throughout this paper, we will use $V(H)$, $E(H)$, and $\vec{E}(H)$
to denote the vertices, edges, and oriented or directed edges
of a graph, subgraph, or path $H$.
For oriented edges $e=(u,v)$,
where $u$ and $v$ are the starting and ending vertices of $e$,
and $f=(s,t)$, define:
 \[
B_{ef}=
\begin{cases}
  1,&\textrm{if $v=s$ and $u\neq t$};\\
  0,& \textrm{otherwise}.
\end{cases}
\]
We order the elements of $\vec{E}$ as $\{e_1, e_2,\cdots,e_{2|E|}\}$, so that the first $|E|$ have end point in the set $V_2$. In this way, we can write
\[
B = 
\left(
\begin{array}{cc} 
0 & B^{(12)} \\
B^{(21)} & 0 
\end{array} 
\right) .
\]
for $|E| \times |E|$ matrices 
$B^{(12)}$ and $B^{(21)}$ with entries equal to $0$ or $1$.

We are interested in the spectrum of $B$. 
% Denote the spectrum of an operator $L$ by $\sigma(L)$. Set: 
% \[
% \sigma(B)=\{\lambda_1\geq \lambda_2\geq \cdots \geq \lambda_{2|E|}\}
% \]
Denote by $\onealpha$ the vector with first $|E|$ coordinates equal to $1$ and the last $|E|$ equal to $\alpha = \sqrt{d_1-1}/\sqrt{d_2-1}$. We can check that
\[
B \onealpha = B^* \onealpha =\lambda \onealpha
\]
for $\lambda=\sqrt{(d_1-1)(d_2-1)}$. 
By the Perron-Frobenius Theorem,
we conclude that $\lambda_1=\lambda$ 
and the associated eigenspace has dimension one. 
Also, one can check that if $\lambda$ is an eigenvalue of $B$ 
with eigenvector $v=(v_1,v_2)$, 
$v_i\in \mathbb{R}^{|E|}$ 
then $-\lambda$ is also an eigenvalue with eigenvector $v'=(-v_1,v_2)$. 
Thus, $\sigma(B)=-\sigma(B)$ and $\lambda_{2|E|}=-\lambda_1$. 

\subsection{Connecting the spectra of $A$ and $B$}\label{sec: connecting}

\begin{figure}[t]
\centering
\includegraphics[trim={5 0 30 20},clip,width=0.49\linewidth]
{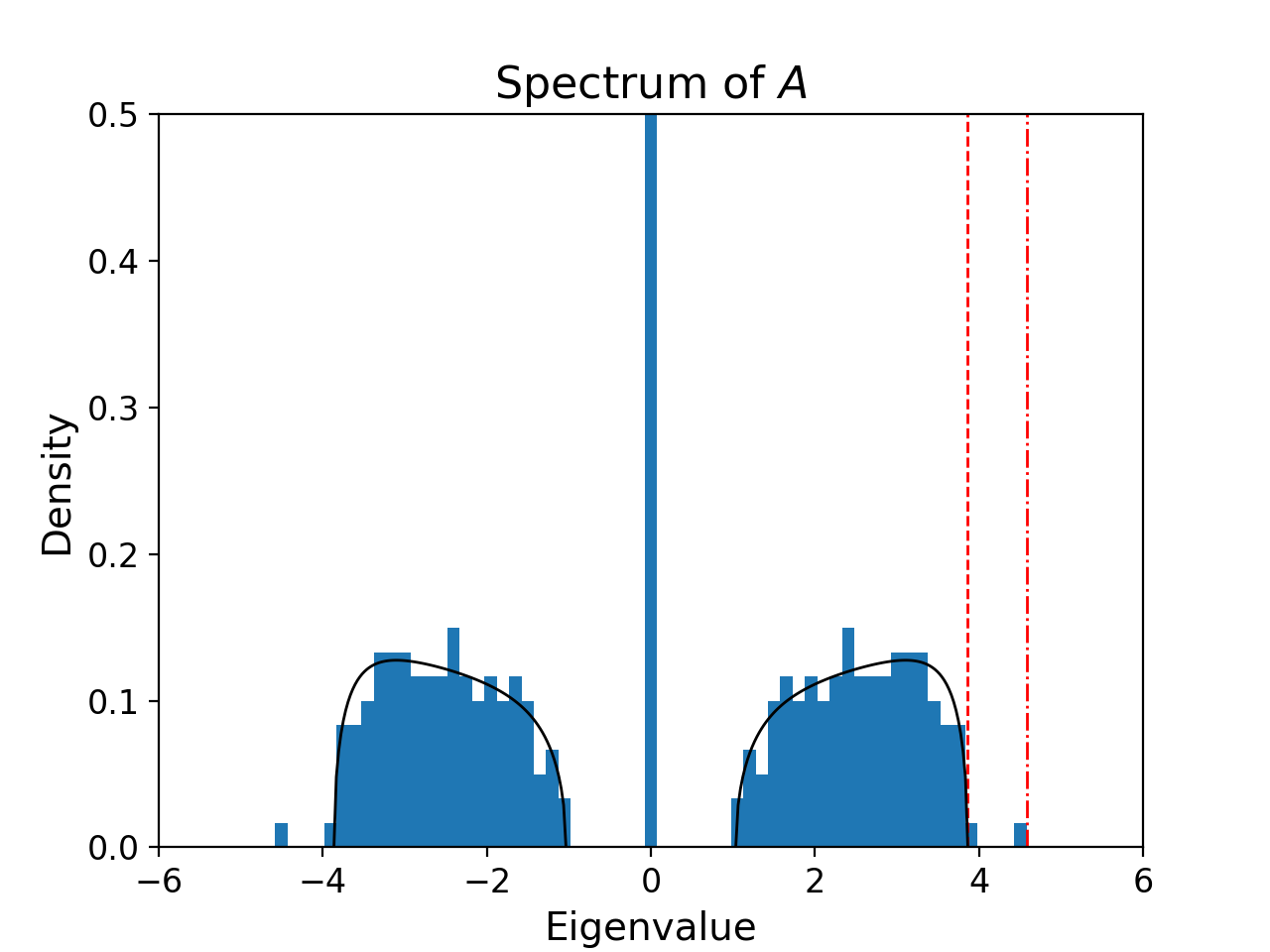}
\hfill
\includegraphics[trim={5 0 30 20},clip,width=0.49\linewidth]
{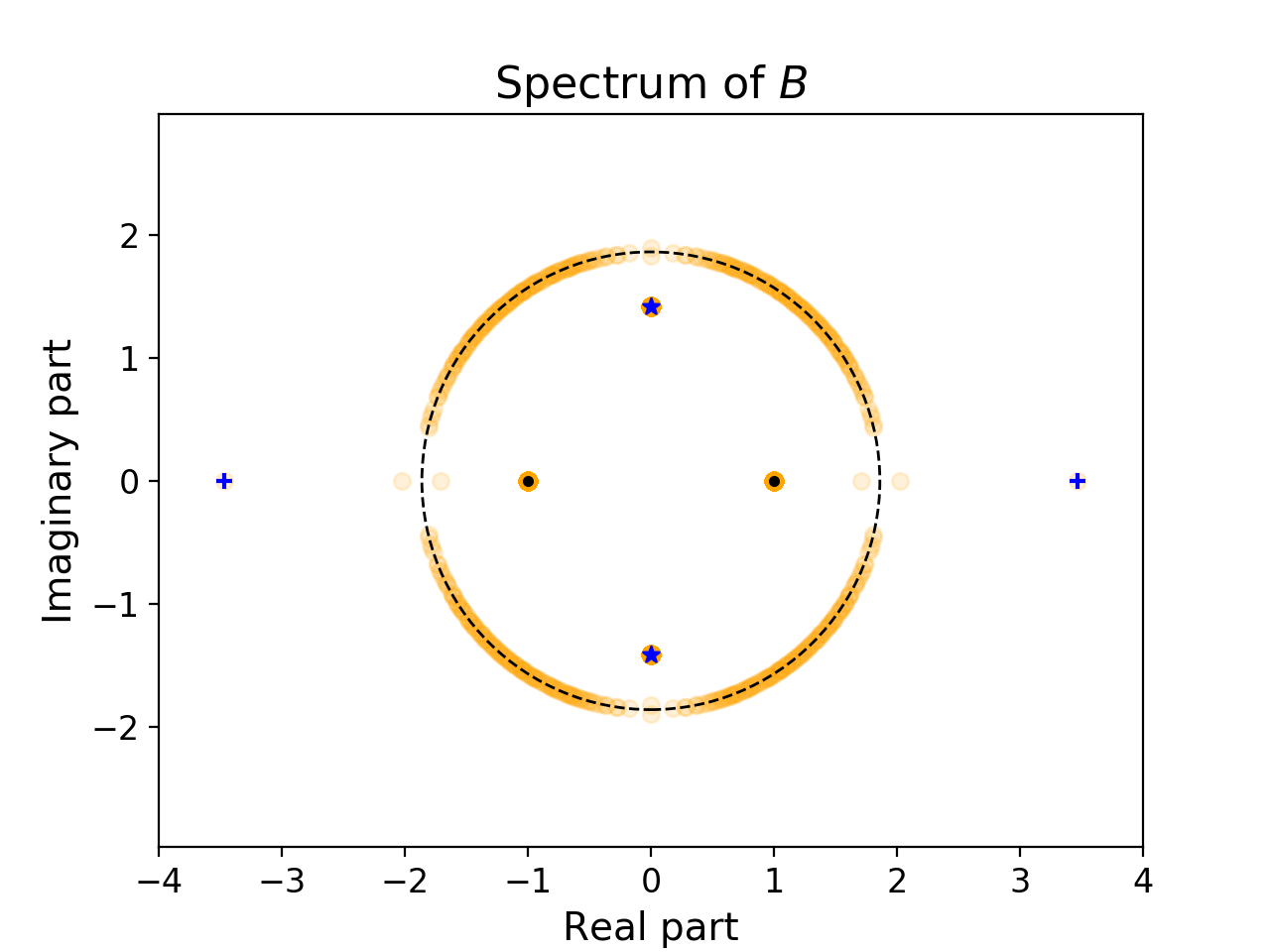}
\caption{
Example spectra for a sample graph $G \sim \mathcal{G}(120, 280, 7, 3)$.
Left, we depict the spectrum of the adjacency matrix $A$.
The dash-dotted line marks the leading eigenvalue, while the 
dashed line marks our bound for the second eigenvalue, 
Theorem~\eqref{cor:gap_A}.
The Mar\u{c}enko-Pastur
limiting spectral density, Eqn.~\eqref{eq:bipartite_esd}, is shown in black.
Right, we depict the spectrum of the non-backtracking matrix
$B$ for the same graph. 
Each eigenvalue is shown as a transparent orange circle, 
the leading eigenvalues are marked with blue crosses, 
and the eigenvalues arising from zero eigenvalues of $A$ are marked with blue stars.
Our main result, Theorem~\ref{thm:gap_B}, 
proves that with high probability the non-leading eigenvalues
are inside, on, or very close to the black dashed circle.
In this case there are 8 outliers of the circle, which arise from 2 pairs of eigenvalues 
below and above the Mar\u{c}enko-Pastur bulk. 
}
\label{fig:spectrum_of_sample}
\end{figure}

Understanding the spectrum of $B$ turns out to be a challenging question. 
A useful result in this direction 
is the following theorem proved by \citet*{bass1992}, and subsequently in \citet*{watanabe2009} and \citet*{kotani2000};
see also Theorem 3.3 in \citet*{angel2015}.
\begin{theorem}[Ihara-Bass formula]\label{thm:ihara}
Let $G=(V,E)$ be any finite graph and $B$ be its non-backtracking matrix. 
Then
\[
\mathrm{det}(B-\lambda I) = (\lambda^2-1)^{|E|-|V|} 
  \mathrm{det}(D-\lambda A+\lambda^2 I),
\]
where $D$ is the diagonal matrix with $D_{vv}=d_v - 1$ 
and $A$ is the adjacency matrix of $G$. 
\end{theorem}
We use the Ihara-Bass formula to analyze the relationship of the spectrum of $B$ to the spectrum of $A$ in the case of a bipartite biregular graph.
It will turn out that this relationship can be completely unpacked. 
From Theorem~\ref{thm:ihara}, we get that
\[
\sigma(B)=\{\pm 1\}\bigcup\{\lambda:~ D-\lambda A+\lambda^2 I~\textrm{is not invertible}\}.
\]
Note that there are precisely $2(m+n)$ eigenvalues of $B$ that are determined by $A$, and that $\lambda = 0$ is not in the spectrum of $B$, since the graph has no isolated vertices ($\det(D) \neq 0$).

We use the special structure of $G$ to get a more precise description of $\sigma(B)$. 
The matrices $A$ and $D$ are equal to:
\[
A=\left(\begin{array}{cc} 0 & X \\
X^* & 0 \end{array}\right ),\ \
D=\left(\begin{array}{cc} (d_1-1)I_n & 0 \\
0 & (d_2-1)I_m \end{array}\right ),
\]
where $I_k$ is the $k\times k$ identity matrix.
Let $\lambda\in \sigma(B)\backslash\{-1,1\}$. 
Then there exists a nonzero vector $v$ such that
\[
(D-\lambda A+\lambda^2 I)v = 0 .
\]
Writing $v=(v_1,v_2)$ with $v_1\in \mathbb{C}^n$, $v_2\in \mathbb{C}^m$, we obtain:
\begin{eqnarray} 
\label{eq:sv_1}
Xv_2 & =& \frac{d_1-1+\lambda^2}{\lambda}v_1, \\
\label{eq:sv_2} 
X^*v_1 & = &\frac{d_2-1+\lambda^2}{\lambda}v_2~.
\end{eqnarray}    
The above imply that, provided that the right hand side is non-zero, 
\begin{eqnarray} 
\label{eq:quadratic_1}
\xi^2=\frac{(d_1-1+\lambda^2)(d_2-1+\lambda^2)}{\lambda^2}
\end{eqnarray}
is a nonzero eigenvalue of both $XX^*$ with eigenvector $v_1$ and $X^*X$, with eigenvector $v_2$.
We can rewrite Eqn.~\eqref{eq:quadratic_1} as
\begin{eqnarray} 
\label{eq:quadratic_2}
\lambda^4 - (\xi^2 - d_1 - d_2+2) \lambda^2 + (d_1-1)(d_2-1) = 0~.
\end{eqnarray}

%From this relation we can derive the lower bound in \eqref{lower_bound} as follows, for any $\epsilon>0$, if $|\lambda_2|<\sqrt{d_1-1}+\sqrt{d_2-1}-\epsilon$ we get $|\eta^2-d_1-d_2-2|< 2\sqrt{(d_1-1)(d_2-1)}$

We will now detail how the eigenvalues of $A$ (denoted $\xi$ here) 
map to eigenvalues of $B$ and vice-versa.
Let us examine the special case $\xi = 0$. Assume $n \leq m$ for simplicity.
Assume that the rank of $X$ is $r$. 
Then $X$ has $m-r$ independent vectors in its nullspace. 
Let $u$ be one such vector. 
Now, if we pick $v_2 = u$, $v_1 = 0$, and $\lambda = \pm i \sqrt{d_2-1}$,
Eqns.~\eqref{eq:sv_1} and \eqref{eq:sv_2} are satisfied. 
Hence, $\pm i \sqrt{d_2-1}$ are eigenvalues of $B$, both with multiplicity $m-r$. 

Since the rank of $X$ is $r$, it follows that the nullity of $X^*$ is $n-r$, 
so there are $n-r$ independent vectors $w$ for which $X^*w = 0$. 
Now, note that picking $v_1 = w$, $v_2 = 0$, and $\lambda = \pm i \sqrt{d_1 -1}$, 
we satisfy Eqns.~\eqref{eq:sv_1} and \eqref{eq:sv_2}. 
Thus, $\pm i \sqrt{d_1-1}$ are eigenvalues of $B$, both with multiplicity $n-r$.

The remaining $4r$ eigenvalues of $B$ determined by $A$ come from nonzero
eigenvalues of $A$.
For each $\xi^2$ with $\xi$ a nonzero eigenvalue of $A$, 
we will have precisely $4$ complex solutions to Eqn.~\eqref{eq:quadratic_1}. 
Since there are $2r$ such eigenvalues, 
coming in pairs $\pm \xi$, they determine a total of $4r$ eigenvalues of $B$, 
and the count is complete.  
To summarize the discussion above, we have the following Lemma:
\begin{lemma} 
\label{A2B} 
Any eigenvalue of $B$ belongs to one of the following categories:
\begin{enumerate}
\item $\pm 1$ are both eigenvalues with multiplicities $|E|-|V| = nd_1 - m -n$,
\item $\pm i \sqrt{d_1-1}$ are eigenvalues with multiplicities $m-r$, where $r$ is the rank of the matrix $X$,
\item $\pm i \sqrt{d_2-1}$ are eigenvalues with multiplicities $r$, and
\item every pair of non-zero eigenvalues $(-\xi, \xi)$ of $A$ generates exactly $4$ eigenvalues of $B$. 
\end{enumerate}
\end{lemma}

\section{Main result}

We spend the bulk of this paper in the proof of the following:
\begin{theorem}
\label{thm:gap_B} 
If $B$ is the non-backtracking matrix of a bipartite, biregular random graph
$G \sim \mathcal{G}(n,m,d_1,d_2)$,
then its second largest eigenvalue
\[
|\lambda_2 (B)| \leq ((d_1-1)(d_2-1))^{1/4}+\epsilon_n
\]
asymptotically almost surely, with $\epsilon_n \to 0$ as $n \to \infty$. Equivalently, there exists a sequence $\epsilon_n \rightarrow 0$ as $n \rightarrow \infty$ so that 
\[
\mathbb{P}\left [ |\lambda_2(B)| - ((d_1-1)(d_2-1))^{1/4} >\epsilon_n \right ] \rightarrow 0 ~~\mbox{as}~~ n \rightarrow \infty~.
\]
\end{theorem}
% * <dumitriu@uw.edu> 2017-11-29T18:49:27.533Z:
%
% ^.
\begin{remark}
For the random lift model, Theorem \ref{thm:gap_B} was proved by \citet*{bordenave2015}, which applies to random bipartite graphs only when $d_1=d_2=d$ as discussed in section \ref{sec:1.2}. 
\end{remark}
We combine Theorems \ref{thm:ihara} and \ref{thm:gap_B} to prove our main result concerning the spectrum of $A$.   
\begin{theorem}[Spectral gap]
\label{cor:gap_A}
Let
$
A=\left(\begin{array}{cc} 0 & X \\
X^* & 0 \end{array}\right )
$
be the adjacency matrix of a bipartite, biregular random graph
$G \sim \mathcal{G}(n,m,d_1,d_2)$. 
Without loss of generality, assume $d_1 \geq d_2$ or, equivalently, $n \leq m$.
Then:
\begin{enumerate}[(i)]
\item Its second largest eigenvalue $\eta = \lambda_2 (A)$ satisfies
%\label{gap}
\[
\eta \leq \sqrt{d_1-1}+\sqrt{d_2-1}+ \epsilon_n'
\]
asymptotically almost surely, with $\epsilon_n' \to 0$ as $n \to \infty$.
\item  Its smallest positive eigenvalue $\eta^+_{\rm min} = \min ( \{\lambda \in \sigma(A): \lambda > 0\} )$ satisfies
\[
\eta_{\rm min}^+\geq \sqrt{d_1-1}-\sqrt{d_2-1}- \epsilon_n''
\]
asymptotically almost surely, with $\epsilon_n''\to 0$ as $n \to \infty$. (Note that this will be almost surely positive if $d_1>d_2$; no further information is gained if $d_1 = d_2$.)
\item If $d_1 \neq d_2$, the rank of $X$ is $n$ with high probability.
\end{enumerate}
\end{theorem}

\begin{remark}
Since the first draft of this work came out, considerable advances haven been made regarding the question of singularity of random regular graphs. It was conjectured in \citet*{costello2008} that, for $3\leq d\leq n-3$, the adjacency matrix of uniform $d$-regular graphs is not singular with high probability as $n$ grows to infinity. For directed $d$-regular graphs and growing $d$, this is now known to be true, following the results of \citet*{cook2017} and \citet*{litvak2016, litvak2017}. 
For constant degree $d$, \citet*{huang2018a,huang2018b} proved the asymptotic non-singularity of the adjacency matrix for both undirected and directed $d-$regular graphs. The last case can be interpreted as singularity of the adjacency matrix of random $d-$regular bipartite graph. To the best of our knowledge, Theorem~\ref{cor:gap_A}(iii) is the first result concerning the rank of rectangular random matrices with $d_1$ nonzero entries in each row and $d_2$ in each column.  
%The problem of investigating the rank of adjacency matrices of  random regular graphs is quite challenging. It was conjectured in \citet*{costello2008} that, for $3\leq d\leq n-3$, the adjacency matrix of uniform $d$-regular graphs is not singular with high probability as $n$ grows to infinity. For directed $d$-regular graphs and growing $d$, this is now known to be true, following the results of \citet*{cook2017} and \citet*{litvak2016, litvak2017}. For constant $d$, \citet*{litvak2018} proved that the rank is at least $n-1$ with high probability as $n$ grows to infinity. To the best of our knowledge, Theorem~\ref{cor:gap_A}(iii) is the first result concerning the rank of rectangular random matrices with $d_1$ nonzero entries in each row and $d_2$ in each column.  
\end{remark}

\begin{remark} The analysis of the Ihara-Bass formula for Markov matrices of bipartite biregular graph appeared before in \citet*{kempton2016}. We have independently proven Lemma~\ref{A2B} and extracted from it more information than is given in \citet*{kempton2016}, including Theorem~\ref{cor:gap_A}(iii). 
\end{remark}

\begin{proof}
Eqn.~\eqref{eq:quadratic_2}, describing those eigenvalues of $B$ which are neither $\pm 1$ and do not correspond to $0$ eigenvalues of $A$, is equivalent to
\begin{eqnarray} 
\label{eq:quadratic_in_xy}
0 = x^2 + \alpha \beta - (y - \alpha - \beta ) x,
\end{eqnarray}
where $x = \lambda^2$, $y = \xi^2$, $\alpha = d_1 - 1$, and $\beta = d_2 -1$. 
% Solving for $y$, we find that
% \[
% y = f(x) = x + \frac{\alpha \beta}{x} + \alpha + \beta .
% \]
% It is a simple exercise to see that in order to have $y$ real, we need to have either that $x$ is real, or that $x$ is on the circle of radius $\sqrt{\alpha \beta}$. 
A simple discriminant calculation and analysis of 
Eqn.~\eqref{eq:quadratic_in_xy}, keeping in mind that $y \neq 0$,  
leads to a number of cases in terms of $y$:
\begin{enumerate}[{Case} 1:]
%\item $y = 0$ leads to solutions $x_{+} = -\beta$ and $x_{-} = -\alpha$,
%meaning the eigenvalues $\lambda$ are $i \sqrt{\beta}$ and $i \sqrt{\alpha}$. 
\item $y \in ((\sqrt{\alpha} -\sqrt{\beta})^2, (\sqrt{\alpha} + \sqrt{\beta})^2)$,
i.e.\ roughly speaking, $\eta$ is in the bulk,
means that $x$ is on the circle of radius $\sqrt{\alpha \beta}$ and the
corresponding pair of eigenvalues $\lambda$ are on a circle of radius 
$(\alpha \beta)^{1/4}$.

\item $y \in (0, (\sqrt{\alpha} - \sqrt{\beta})^2]$ means that $x$ is real and negative, so $\lambda$ is purely imaginary.
%\item $|\lambda_2| \leq ((d_1-1)(d_2-1))^{1/4} + \epsilon$ 
%with high probability implies that if 
%$d_1 \neq d_2$, 
%$r=n$ with high probability.
%Otherwise, we would have eigenvalues of $B$ with 
%absolute value $\sqrt{d_1-1}$ 
%and this is larger than $((d_1-1)(d_2-1))^{1/4}$.

In this case, one may also show that the smaller of the two possible values 
for $x$ is increasing as a function of $y$ and
$x_{-} \in (-\alpha, -\sqrt{\alpha \beta}]$. 
The larger of the two values of $x$ is decreasing and $x_{+} \in [-\sqrt{\alpha \beta}, - \beta)$.
Correspondingly, the largest in absolute value that 
$\lambda$ could be in this case 
is $\pm i \alpha^{1/4} = \pm i (d_1-1)^{1/4}$. 

\item $y \geq (\sqrt{\alpha} + \sqrt{\beta})^2$ means that 
both solutions $x_{\pm}$ are real,
and the larger of the two is larger than $\sqrt{\alpha \beta}$.
\end{enumerate}

Note that Eqn.~\ref{eq:quadratic_in_xy} 
shows there is a continuous dependence between $x$ and $y$, 
and consequently between $\xi$ and $\lambda$.
Putting these cases together with Lemma \ref{A2B}, a few things become apparent:
\begin{enumerate}
\item $\xi>\sqrt{d_1-1}+ \sqrt{d_2-1}$ means that 
$\lambda> ((d_1-1)(d_2-1))^{1/4}$.
\item $|\lambda_2| \leq ((d_1-1)(d_2-1))^{1/4} + \epsilon$ 
implies that all eigenvalues except for the largest two 
will be either $0$, or in a small neighborhood 
$[\sqrt{\alpha} -\sqrt{\beta} - \delta, 
\sqrt{\alpha} + \sqrt{\beta} + \delta]$ 
of the bulk, with $\delta$ small if $\epsilon$ is small since the dependence of $\delta$ on $\epsilon$ can be deduced from Eqn.~\ref{eq:quadratic_in_xy}. 
\item $|\lambda_2| \leq ((d_1-1)(d_2-1))^{1/4} + \epsilon$ 
with high probability implies that if 
$d_1 \neq d_2$, 
$r=n$ with high probability.
Otherwise, we would have eigenvalues of $B$ with 
absolute value $\sqrt{d_1-1}$ 
and this is larger than $((d_1-1)(d_2-1))^{1/4}$.
\end{enumerate}
This completes the proof, with results (i) and (ii) following from the point 2 and (iii) following from point 3.
\end{proof}

In Fig.~\ref{fig:spectrum_of_sample}, 
we depict the spectra of $A$ and $B$ for a sample graph
$G \sim \mathcal{G}(120,280,7,3)$.
Looking at the non-backtracking spectrum, 
we observe the two leading eigenvalues $\pm \sqrt{(d_1 - 1 )(d_2 - 1)}$
(blue crosses) outside the circle of radius 
$((d_1 - 1)(d_2 - 1))^{1/4}$
along with a number of zero eigenvalues (black dots).
There are also multiple purely imaginary eigenvalues 
which can arise from $|\xi| \in (0, \sqrt{d_1 - 1} - \sqrt{d_2 - 1}]$
as well as $\xi = 0$.
However, due to Theorem~\ref{cor:gap_A},
only the smaller of $i \sqrt{d_1 - 1}$ and $i \sqrt{d_2 - 1}$ 
is observed with non-negligible probability,
implying that $X$ has rank $r = n$ with high probability
(shown as blue stars).
Furthermore, we observe two pairs of real eigenvalues of $B$ which are connected to a pair of eigenvalues of $A$ from ``above'' the bulk, as well as two pairs of imaginary eigenvalues of $B$ which are connected to a pair of eigenvalues of $A$ from ``below'' the bulk.

\section{Preliminaries}
\label{sec:preliminaries}

We describe the standard configuration model for 
constructing such graphs.
We then define the ``tangle-free'' property of random graphs.
Since small enough neighborhoods are tangle-free with high probability, 
we only need to count tangle-free paths when we eventually employ the trace
method.

\subsection{The configuration model}

\label{sec:configuration_model}
The configuration or permutation model is a practical procedure to sample 
random graphs with a given degree distribution. 
Let us recall its definition for bipartite biregular graphs. 
Let $V_1=\{v_1,v_2,\dots,v_n\}$ and $V_2=\{w_1,w_2,\dots, w_m\}$ 
be the vertices of the graph. 
We define the set of \textit{half edges} out of $V_1$
to be the collection of ordered pairs
\[
\vec{E}_1=\{(v_i,j)~\mbox{for $1\leq i\leq n$ and $1\leq j\leq d_1$}\}
\]
and analogously the set of half edges out of $V_2$:
\[
\vec{E}_2=\{(w_i,j)~\mbox{for $1\leq i\leq m$ and $1\leq j\leq d_2$}\},
\]
see Figure~\ref{fig:bipartite_schematic}.
Note that $|\vec{E}_1|=|\vec{E}_2|=nd_1=md_2 = |E|$. 
To sample a graph, we choose a random permutation $\pi$ of 
$[n d_1]$. 
We put an edge between $v_i$ and $w_j$ in the $G$ whenever 
\[
\pi((i-1)d_1+s)=(j-1)d_2+t
\]
for any pair of values $1\leq s\leq d_1$, $1\leq t\leq d_2$. 
For {\em specific} half edges $e = (v_i, j)$ and $f = (w_s, t)$, 
we use the notation $\pi(e) = f$ as shorthand for
$\pi((i-1)d_1 + j) = (s-1)d_2 + t$ and say that
``$e$ matches to $f$.''

The graph obtained may not be simple, since multiple half edges may be matched 
between any pair of vertices. 
However, conditioning on a simple graph outcome, 
the distribution is uniform in the set of all 
simple bipartite biregular graphs. 
Furthermore, for fixed $d_1, d_2$ and $n,m \to \infty$,
the probability of getting a simple graph is bounded away from zero
\citep*{bollobas2001}.

%Since half edges may be identified with a particular ordering of the set of oriented edges ($\vec{E}_1 \cup \vec{E}_2$ and $\vec{E}$ are isomorphic), this allows to define the non-backtracking operator $B$ independent of the graph realization. 
Consider the random 
$B \in \R^{2 |E|\times 2 |E|}$ 
whose first $|E|$ rows are indexed by the elements of $\vec{E}_1$ 
and the last $|E|$ rows are indexed by those of $\vec{E}_2$, 
in lexicographic order. 
Columns are indexed in the same way. 
Entry $B_{ef}$ with 
$e=(v_i,j) \in \vec{E}_1$ and $f=(w_s,t) \in \vec{E}_2$ 
is defined as
\[
B_{ef}=
\begin{cases}
  1,& \mbox{if $\pi(e) = f' = (w_s, t')$ and $t'\neq t$};\\
  0,& \textrm{otherwise}.
\end{cases}
\]
This defines the upper half of $B$. We define the lower half similarly, by putting
\[
B_{fe}=
\begin{cases}
  1,&\textrm{if $\pi^{-1}(f) = e' = (v_i, j')$ and $j'\neq j$};\\
  0,& \textrm{otherwise}.
\end{cases}
\]
This is the same definition used in \cite{bordenave2015}. 
In words, it says that the directed edge given by $e$ 
followed by the directed edge given by $f$ 
are connected by some half edge $f' = (w_s, t')$,
and the path they form does not backtrack.
This is therefore the same matrix introduced in Section~\ref{sec2:B},
ordered according to the half edges.
Notice that the randomness comes from the matching only. 

We consider two symmetric matrices $M = M(\pi)$ and $N$, 
indexed the same as $B$, and defined by:
\[
M_{ef} = 1_{ \{ \pi(e) = f \} } 
\mbox{ and }
M_{fe} = 1_{ \{ \pi^{-1} (f) = e \} }
~\mbox{ for $e\in \vec{E}_1$ and $f\in \vec{E}_2$}
\]
and
\[
N_{gh} = 1_{ \{u = v \mbox{ and } i\neq j \}}
~\mbox{ for $g = (u,i), h = (v,j) \in \vec{E}_1 \cup \vec{E}_2$} .
\]
We see that a term like $M_{eg} N_{gf}$ corresponds to
$M$ matching the directed edge $e$ to $g$ by $\pi$, 
and $N$ taking us out of the vertex of $g$ along the directed edge $f$, 
which is different from $g$.
Thus the rule of matrix multiplication means that
\begin{align}\label{eq:B=MN}
B=MN.
\end{align}
This equality will be useful in Section~\ref{sec:matrix_decomp}
when working with products of the matrix $B$.

\subsection{Tangle-free paths}
\label{sec:tangle_free}

Sparse random graphs, including bipartite graphs,
have the important property of being ``tree-like''
in the neighborhood of a typical vertex. 
Formally, consider a vertex $v \in V_1 \cup V_2$. 
For a natural number $\ell$, 
we define the ball of radius $\ell$ centered at $v$ to be:
\[
B_{\ell}(v) = \{w \in V_1 \cup V_2 : ~ d_{G}(v,w) \leq \ell\}
\]
where $d_G(\cdot,\cdot)$ is the graph distance.
\begin{definition}
A graph $G$ is \textit{$\ell$-tangle-free} if
$B_{\ell}(v)$ contains at most one cycle for any vertex $v$.
\end{definition}

The next lemma says that most bipartite biregular graphs are $\ell$-tangle-free
up to logarithmic sized neighborhoods.
\begin{lemma}
Let $G \sim \mathcal{G}(n,m,d_1,d_2)$ be a bipartite, biregular random graph. 
Let $\ell < \frac{1}{8} \log_d (n)$, for $d=\max\{d_1, d_2\}$.
%Fix a constant $c < 1/8$ and let $\ell=c\log_{d}(n)$.
Then $G$ is $\ell$-tangle-free with probability at least $1-n^{-1/2}$.
\end{lemma}
\begin{proof} 
This is essentially the proof given in \citet*{lubetzky2010}, Lemma 2.1.
Fix a vertex $v$. 
We will use the so called {\it exploration process} to discover the ball $B_{\ell}(v)$. More precisely, we order the set $\vec{E}_1$ lexicographically: 
$(v_i,j) < (v_{i'},j')$ 
if $i\leq i'$ and $j\leq j'$. 
The exploration process reveals $\pi$ one edge at the time, by doing the following:
\begin{itemize}
\item A uniform element is chosen from $\vec{E}_2$ and it is declared equal to $\pi(1)$.
\item A second element is chosen uniformly, now from the set 
$\vec{E}_2\backslash\{\pi(1)\}$ and set equal to $\pi(2)$.
\item Once we have determined $\pi(i)$ for $i\leq k$, we set $\pi(k+1)$ equal to a uniform element sampled from the set $\vec{E}_2\backslash\{\pi(1),\pi(2),\dots,\pi(k)\}$.
\end{itemize}
We use the final $\pi$ to output a graph as we did in the configuration model. 
The law of these graphs is the same.
With the exploration process, we expose first the neighbors of $v$, 
then the neighbors of these vertices, and so on.
This breadth-first search reveals all vertices in $B_k (v)$ before 
any vertices in $B_{j > k} (v)$.
Note that, although our bound is for the family $\mathcal{G}(n,m,d_1,d_2)$,
the neighborhood sizes are bounded above by those of  
the $d$-regular graph with $d=\max(d_1,d_2)$.
% Observe that, when matching half edges of those vertices 
% at distance $(2 t-1)$ from $v$, 
% we are looking at $\pi^{-1}(w,j)$ for a given $(w,j)$. 
% \todo{I have no idea why the previous sentence is relevant}

Consider the matching of half edges attached to vertices in
the ball $B_i(v)$ at depth $i$
(thus revealing vertices at depth $i+1$).
In this process,
we match a maximum
$m_i \leq d^{i+1}$ pairs of half edges total. 
Let $\mathcal{F}_{i,k}$ be the filtration
generated by matching up to the $k$th half edge
in $B_i(v)$,
for $1 \leq k \leq m_i$. 
Denote by $A_{i,k}$ the event that the $k$th
matching creates a cycle at the current depth.
For this to happen, the matched vertex must have appeared 
among the $k-1$ vertices already revealed at depth $i+1$.
The number of unmatched half edges is at least $nd - 2 d^{i+1}$.
We then have that:
\[
\mathbb{P}(A_{i,k}) 
\leq \frac{(k-1) (d-1)}{nd - 2 d^{i+1}}
\leq \frac{(d-1) m_i}{(1 - 2 d^{i+1} n^{-1}) nd}
\leq \frac{m_i}{n}.
\]
So, we can stochastically dominate the sum
\[
\sum_{i=1}^{\ell-1}\sum_{k=1}^{m_i} A_{i,k}
\]
by $Z \sim \mathrm{Bin} \left( d^{\ell+1},\ n^{-1} d^{\ell} \right)$. 
So the probability that 
$B_{\ell}(v)$ is $\ell$-tangle-free has the bound:
\[
\mathbb{P} (
	\mbox{$B_{\ell}(v)$ is not $\ell$-tangle-free}
 ) =
\mathbb{P} \left(
  \sum_{i=1}^{\ell-1}\sum_{k=1}^{m_i} A_{i,k}> 1
  \right)
  \leq 
\mathbb{P} (Z>1) = 
O\left( \frac{d^{4\ell+1}}{n^2} \right) =
O\left( n^{-3/2} \right) ,
\]
which follows using that $\ell = c \log_d n$ with $c < 1/8$. 
The Lemma follows by taking a union bound over all vertices. 
\end{proof}

\section{Proof of Theorem~\ref{thm:gap_B}}

\subsection{Outline}
\label{sec:proof_outline}

We are now prepared to explain the main result.
To study the second largest eigenvalue of the non-backtracking matrix,
we examine the spectral radius of the matrix obtained by subtracting
off the dominant eigenspace.
 We use for this:
\begin{lemma}[\citet*{bordenave2015a}, Lemma 3]
\label{lem:subspace}
Let $T$ and $R$ be matrices such that $\Img(T)\subset \Ker(R)$, 
$\Img(T^*)\subset \Ker(R)$. 
Then all eigenvalues $\lambda$ of $T+R$
that are not eigenvalues of $T$ satisfy:
\[
|\lambda|\leq \max_{x\in \Ker(T)}\frac{\|(T+R)x\|}{\|x\|} .
\]
\end{lemma}
Throughout the text,
$\| \cdot \|$ is the spectral norm for matrices and $\ell^2$-norm for vectors.
Recall that the leading eigenvalues of $B$, in magnitude, are 
$\lambda_1 = \sqrt{(d_1 -1)(d_2-1)}$ and $\lambda_{2|E|} = -\lambda_1$
with corresponding eigenvectors $\1_{\alpha}$ and $\mathbf{1}_{-\alpha}$.
Applying Lemma~\ref{lem:subspace} with 
$T 
= 
\frac{\lambda_1^{\ell}}{\onealpha^* \onealpha}
(\onealpha \onealpha^* +
(-1)^\ell \onemalpha \onemalpha^*)$
and 
$R = B^\ell - T$,
we get that
\begin{align}
\label{eq:upper_bound_lambda2}
\lambda_2(B)
\leq 
\max_{
\footnotesize
\begin{array}{c}
x \in \Ker(T) \\ 
\|x\| = 1 
\end{array}
} 
\left( \| B^{\ell}x \| \right)^{1/\ell}.
\end{align}
It will be important later to have a more precise description of the set $\Ker(T)$. 
It is not hard to check that 
\begin{align*}
\Ker(T) &= \{x:
  ~ \langle x, \onealpha \rangle = \langle x, \onemalpha \rangle=0\} \\
& =\{(v,w)\in \mathbb{R}^{2|E|}: 
  ~ \langle v, \1 \rangle = \langle w, \1 \rangle=0\}.
\end{align*}
In the last line, the vectors $v$, $w$ and $\1$ are $|E|-$dimensional, 
and $\1$ is the vector of all ones.

In order to use Eqn.~\ref{eq:upper_bound_lambda2},
we must bound $\| B^\ell x \|$ for large powers $\ell$
and $x \in \Ker(T)$.
This amounts to counting the contributions of certain non-backtracking walks.
We will use the tangle free property in order to only count
$\ell$-tangle-free walks. 
We break up $B^\ell$ into two parts
in Section~\ref{sec:matrix_decomp}, 
an ``almost'' centered matrix $\bar{B}^\ell$ 
and the remainder $\sum_j R^{\ell,j}$,
and we bound each term independently.

To compute these bounds, 
we need to count the contributions of many different non-backtracking walks.
We will use the trace technique, so only circuits which return to the starting
vertex will contribute. 
In Section~\ref{sec:expectation_bounds},
we compute
the expected contribution of products of $B$ along such circuits, 
employing a result from \cite{bordenave2015}.
% In Section~\ref{sec:expectation_bounds}, 
% we apply a useful result from \citet*{mckay1981a}
% to compute the probability, during the exploration process,
% of revealing a new edge $e$ 
% given that we have already observed a certain subgraph $H$. 
% In particular, we find different probabilities depending on whether $e$
% shares one of more endpoints in $H$.
% We use this to bound the expectation of the of the product of
% entries $\bar{B}_{ef}$ along segments $ef$ of a non-backtracking walk.
% A similar argument appears later, in the proof of Theorem~\ref{spec_R}, 
% for products of $R^{\ell,j}_{ef}$.

Section~\ref{sec:path_counting} covers the combinatorial component of the proof. 
The total contributions $\| B^\ell x \|$ come from many 
non-backtracking circuits of different flavors,
depending on their number of vertices, edges, cycles, etc.
Each circuit is broken up into $2k$ segments of tangle-free
walks of length $\ell$.
We need to compute not only the expectation along the circuit,
but also upper-bound the number of circuits of each flavor.
We introduce an injective
encoding of such circuits that depends on the number of vertices,
length of the circuit, and, crucially, the tree excess of the circuit.
An important part of these calculations is to keep track
of the imbalance between left and right vertices
visited in the circuit, since this controls
the powers of $d_1$ and $d_2$ in the result.

Finally, in Section~\ref{sec:norm_bounds} we put all of these ingredients together
and use Markov's inequality to bound each matrix norm with high probability.
We find that 
$\| \bar{B}^\ell \|$ 
contributes a factor that goes as 
$((d_1 - 1)(d_2 - 1))^{\ell/4}$,
whereas $\| R^{\ell,j} \|$ contributes only a factor of $(d-1)^\ell/n$,
up to polylogarithmic factors in $n$.
Thus, the main contribution to the circuit counts comes from the mean
and, in fact, comes from circuits which are exactly trees traversed 
forwards and backwards.
Interestingly, this
is analogous to what happens when using the trace method
on random matrices of independent entries. 

In the proof, we are forced to consider tangled paths
but which are built up of tangle-free components.
This delicate issue was first made clear by \citet*{friedman2004}
who introduced the idea of tangles and a ``selective trace.''
\citet*{bordenave2015a}, who we follow closely in this part of our analysis,
also has a good discussion of these issues and their history.
We use the fact that
\begin{equation}
\label{eq:trace_bound}
\E\left(\|\bar{B}^{\ell}\|^{2k}\right)
\leq 
\E\left( 
\mathrm{Tr} 
\left[
\left( (\bar{B}^{\ell}) (\bar{B}^{\ell})^* \right)^{k} 
\right]
\right),
\end{equation}
and so deal with circuits built up of $2k$ segments which are $\ell$-tangle-free.
Notice that the first segment comes from $\bar{B}^\ell$, the second
from $(\bar{B}^{\ell})^*$, etc.
Because of this, the directionality of the edges along each segment alternates.
See Figure~\ref{fig:circuit} for an illustration of a path
which contributes for $k=2$ and $\ell=2$.
Also, while each segment is $\ell$-tangle-free,
the overall circuit may be tangled.

\begin{figure}[t!]
\centering
\includegraphics[width=0.4\linewidth]{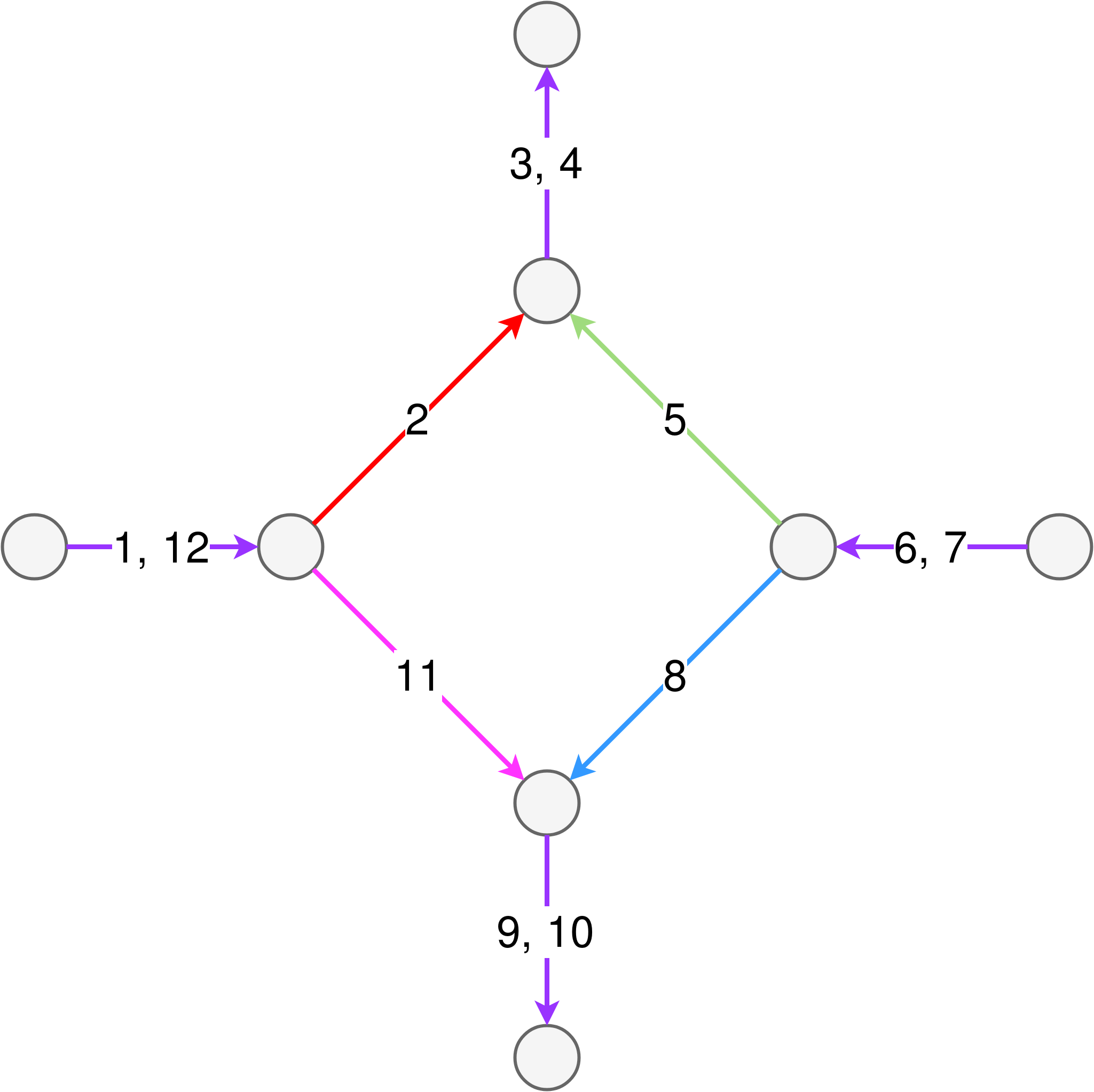}

\caption{An example circuit that contributes to the trace in 
	Eqn.~\eqref{eq:trace_bound},
 for $k = 2$ and $\ell = 2$.
 Edges are numbered as they occur in the circuit.
 Each segment $\{ \gamma_i \}_{i=1}^4$ 
 is of length $\ell + 1 = 3$ and made up of edges $3(i-1)+1$ through $3i$.
 The last edge of each $\gamma_i$ is the first edge of $\gamma_{i+1}$, 
 and these are shown in purple.
 Every path $\gamma_{i}$ with $i$ even
 follows the edges backwards due to the matrix transpose.
 However, this detail turns out not to make any difference since the underlying
 graph is undirected.
 Our example has no cycles in each segment for clarity,
 but, in general, each segment can have up to one cycle, 
 and the overall circuit may be tangled.
}
\label{fig:circuit}
\end{figure}

\subsection{Matrix decomposition}\label{sec:matrix_decomp}

We start this section by defining the set of paths that will be relevant to 
bound the norm of $\|\bar{B}^{\ell}\|$.
We closely follow \cite{bordenave2015}.
%For this section, we assume $G$ is 
%$\ell$-tangle-free, which will hold with high probability.
\begin{definition}
Define $\Gamma^{\ell}_{ef}$ to be the set of all 
non-backtracking paths of $2\ell+1$ half edges, 
starting at $e$ and ending at $f$. 
A path in this set will be denoted by
$\gamma= (e_1,e_2,\dots,e_{2\ell},e_{2\ell+1})$,
where $e_1 = e$ and $e_{2\ell + 1} = f$. 
The non-backtracking property means that,
for all $1\leq i\leq \ell$, $e_{2i}$ and $e_{2i+1}$ share the same vertex 
but $e_{2i}\neq e_{2i+1}$.
Similarly, let $\Gamma^\ell = \bigcup_{e,f} \Gamma^\ell_{ef}$.
%this is to ensure we end on the right side. We use the transpose to concatenate!
\end{definition}

Each path in $\Gamma^\ell_{ef}$ uses $2 \ell + 1$ half edges,
corresponding to $\ell + 1$ edges in the graph.
To be clear, the above definition counts all possible
non-backtracking sequences of half edges.
These are different than the usual non-backtracking paths 
{\em and do not necessarily exist in the graph}.
Some of these paths might backtrack along a duplicate
edge which utilizes a different half edge.

We now have
%%%%%%%%%%%%%%%%%%%%%%%%%%%%%%%%%%%%%%%%%%%%%%%%%%%%%%%%%%%%%%%%%%%%%%%%
\begin{comment}
non-backtracking paths
of length $2\ell+1$, starting at half edge $e$ and ending at $f$.  
For a path $\gamma \in \Gamma^{\ell}_{ef}$,
we write
$\gamma = (e_1, e_2, \ldots, e_{\ell+1})$ 
where $e_i \in \vec{E}$ for all $i$,
$e_1 = e$ and $e_{\ell+1} = f$.
Similarly, define
$F^{\ell}_{ef} \subset \Gamma^\ell_{ef}$ 
be the set of all non-backtracking, tangle-free paths
in $G$
of length $\ell+1$, starting at oriented edge $e$ and ending at $f$. 
\end{comment}
%%%%%%%%%%%%%%%%%%%%%%%%%%%%%%%%%%%%%%%%%%%%%%%%%%%%%%%%%%%%%%%%%%%%%%%%%%
\begin{equation}
\label{eq:B_in_terms_of_M}
(B^{\ell})_{ef} = 
\sum_{\gamma \in \Gamma^\ell_{ef}}
\prod_{t=1}^{\ell} B_{e_{2t-1} e_{2t+1}}
= 
\sum_{\gamma \in \Gamma^\ell_{ef}}
\prod_{t=1}^{\ell} M_{e_{2t-1} e_{2t}}N_{e_{2t}e_{2t+1}}
=
\sum_{\gamma \in \Gamma^\ell_{ef}}
\prod_{t=1}^{\ell} M_{e_{2t-1} e_{2t}}
\end{equation}
%where we note the last equality requires $G$ to be $\ell$-tangle-free.
where we used Eqn.~\ref{eq:B=MN}
and the fact that $\Gamma^{\ell}_{ef}$ is non-backtracking, 
so
$
N_{e_{2t}e_{2t+1}}=1
$.

Recall that we will use
Eqn.~\eqref{eq:upper_bound_lambda2} and
Lemma \ref{lem:subspace} to bound $\lambda_2$.
Denote by $\bar{B}$ the matrix with entries equal to
$\bar{B} = B - S$,
where
\[
 S = \frac{1}{|E|}
 \left(
 \begin{array}{cc} 
 0 & (d_2 - 1) \1 \1^* \\
 (d_1 - 1) \1 \1^* & 0 
 \end{array}
 \right ).
\]
Note that $\bar{B}$ is an \textit{almost} centered version of $B$,
and $\Ker ( S ) = \Ker(T) = \mathrm{span} (\onealpha, \onemalpha ) $, 
where $T$ is the matrix from Lemma \ref{lem:subspace}.
To apply the lemma, we wish to get an expression like Eqn.~\eqref{eq:B_in_terms_of_M} for $\bar{B}^{\ell}$. To do so, we write:
\[
\bar{B} = B - S = M N - S = (M - S') N
\]
the above matrix equation in the unknown $S'$ can be solved by simple manipulations. We get
\[
S'=\frac{1}{|E|} \left(
 \begin{array}{cc} 
 0 & \1 \1^* \\
  \1 \1^* & 0 
 \end{array}
 \right ).
\]
Using again that $N$ is identically one over the elements of the set 
$\Gamma^{\ell}_{ef}$, 
we find a similar formula to Eqn.~\eqref{eq:B_in_terms_of_M}:
\begin{align}\label{eq:barB_in_terms_of_M}
(\bar{B}^\ell)_{ef}
=
\sum_{\gamma \in \Gamma^\ell_{ef}} 
\prod_{t=1}^{\ell}
\left( B - S \right)_{e_{2t-1} e_{2t+1}}
=
\sum_{\gamma \in \Gamma^\ell_{ef}} 
\prod_{t=1}^{\ell}
\bar{M}_{e_{2t-1} e_{2t}},
\end{align}
where $\bar{M}=M-S'$.

The following telescoping sum formula is a simple algebraic manipulation 
and appears in \citet*{massoulie2013} 
and \citet*{bordenave2015a}: 
\begin{equation*}
\prod_{s=1}^\ell x_s = \prod_{s=1}^\ell y_s + 
\sum_{j=1}^\ell \prod_{s=1}^{j-1} y_s (x_j - y_j) \prod_{t=j+1}^\ell x_t .
\end{equation*}
Using this, 
with $x_s = B_{e_{2s-1} e_{2s+1}}$ 
and $y_s = \bar{B}_{e_{2s-1} e_{2s+1}}$,
we obtain the following relation:
\begin{align}
\label{B_dec1}
( B^{\ell} )_{ef} 
= 
( \bar{B}^{\ell} )_{ef} 
+ 
\sum_{\gamma\in \Gamma^{\ell}_{ef}} 
\sum_{j=1}^{\ell} 
\prod_{s=1}^{j-1}\
\bar{B}_{e_{2s-1} e_{2s+1}}
S_{e_{2j-1} e_{2j+1}}
\prod_{t=j+1}^{\ell} B_{e_{2t-1} e_{2t+1}}.
\end{align}
This decomposition breaks the elements in $\Gamma^{\ell}_{ef}$ 
into two subpaths, also non-backtracking, 
of length $j$ and $\ell-j$, respectively. 

\begin{definition}
Let $F_{ef}^\ell \subset \Gamma^\ell_{ef}$ 
denote the subset of paths
which are tangle-free,
with $F^\ell = \bigcup_{e,f} F^\ell_{ef}$. 
\end{definition}

We will take the parameter $\ell$ to be small enough so that
the path $\gamma$ is tangle-free with high probability.
Thus the sums in Eqns.~\eqref{eq:B_in_terms_of_M}
or \eqref{eq:barB_in_terms_of_M}
need only be over the paths 
$\gamma \in F_{ef}^\ell$.
However,
to recover the matrices $B$ and $\bar{B}$ by rearranging 
Eqn.~\eqref{B_dec1},
we need to also count those tangle-free 
subpaths that arise from splitting tangled paths.
While breaking a tangle-free path will necessarily
give us two new tangle-free subpaths, the converse is not always true. 
This extra term generates a remainder that we define now.

\begin{definition}
\label{def:tangled_paths}
Let $T^{\ell,j}_{ef}$ 
be the set of non-backtracking paths containing $2\ell+1$ half edges,
starting at $e$ and ending at $f$,
such that overall the path is tangled
but the first $2j-1$, middle three, and last $2(\ell - j)+1$ 
half edges form tangle-free subpaths:
$\gamma = (e_1, \ldots, e_{2\ell + 1}) \in T^{\ell,j}$
if and only if
$\gamma' = (e_1, \ldots, e_{2j-1}) \in F^{j-1}$,
$\gamma'' = (e_{2j-1}, e_{2j}, e_{2j+1}) \in F^{1}$,
and $\gamma''' = (e_{2j+1}, \ldots, e_{2\ell + 1}) \in F^{\ell - j}$.
Set $T^{\ell,j} = \bigcup_{e,f} T^{\ell,j}_{ef}$.
\end{definition}
Set the remainder
\begin{align}
    \label{error_term}
    R^{\ell,j}_{ef} 
    &= 
    \sum_{\gamma\in T^{\ell,j}_{ef}} 
    \sum_{j=1}^{\ell}
    \prod_{s=1}^{j-1}
    \bar{B}_{e_{2s-1} e_{2s+1}}
    S_{e_{2j-1} e_{2j+1}}
    \prod_{t=j+1}^{\ell} B_{e_{2t-1} e_{2t+1}}
    \nonumber \\
    &=
    \sum_{\gamma\in T^{\ell,j}_{ef}} 
    \sum_{j=1}^{\ell} 
    \prod_{s=1}^{j-1}\
    \bar{M}_{e_{2s-1} e_{2s}}
    S_{e_{2j-1} e_{2j+1}}
    \prod_{t=j+1}^{\ell} 
    M_{e_{2t-1} e_{2t}}.
\end{align}
Since the paths are non-backtracking, 
the $N$ terms are all unity.

Adding and subtracting 
$\sum_{j=1}^{\ell} R^{\ell,j}_{ef}$ 
to Eqn.~\eqref{B_dec1}
and rearranging the sums, we obtain
\begin{align}
\label{B_dec2}
B^{(\ell)} = 
\bar{B}^{(\ell)} + 
\sum_{j=1}^{\ell} 
\bar{B}^{(j)} S B^{(\ell-j)}
- 
\sum_{j=1}^{\ell} R^{\ell,j},
\end{align}
where the matrices $B^{(\ell)}$ and $\bar{B}^{(\ell)}$
are tangle-free versions of $B^\ell$ and $\bar{B}^\ell$,
i.e.\ element $ef$ in both matrices
only counts paths $\gamma \in F_{ef}^\ell$.
Multiplying Eqn.~\eqref{B_dec2} on the right by 
$x \in \Ker (T)$
and using that 
$B^{(\ell - j)} x$
is also within $\Ker(S)$,
since it is just the space spanned by the leading eigenvectors,
we find that the middle term is identically zero.
Thus for $x \in \Ker(T)$,
\begin{equation}\label{decomposition}
\| B^{(\ell)}x \| 
\leq 
\| \bar{B}^{(\ell)} x \|+
\left\| 
%\|
\sum_{j=1}^{\ell}R^{\ell,j} x 
%\|
\right\|
\leq 
\| \bar{B}^{(\ell)} \|+
\sum_{j=1}^{\ell} \| R^{\ell,j} \| .
\end{equation}

\subsection{Expectation bounds}
\label{sec:expectation_bounds}

Our goal is to find a bound on the expectation of certain random variables
which are products of $\bar{B}_{ef}$ along a circuit.
To do this, we will need to bound the probabilities of different
subgraphs when exploring $G$.
This requires us to introduce the concept of consistent
edges and their multiplicity.

\begin{definition}\label{def:incon_edges}
Let $\gamma = (e_1, \ldots, e_{2k})$ be a sequence of half edges
of even length, with 
$\vec{E}(\gamma)$ its set of half edges,
and $E(\gamma) = \{ \{e_{2i-1}, e_{2i}\} \mbox{ for $i\in [k]$} \}$ its set of edges (unordered pairs, thus undirected).
\begin{itemize}
    \item The {\em multiplicity of a half edge} $e \in \vec{E}(\gamma)$
    is $m_\gamma(e) = \sum_{t=1}^{2k} 1_{\{ e_t = e \}}$.
    \item The {\em multiplicity of an edge} $\{h_1, h_2\} \in E(\gamma)$,
    is 
    $m_\gamma(\{h_1,h_2\}) = 
    \sum_{t=1}^k 1_{\{ \{e_{2t-1}, e_{2t}\} = \{h_1, h_2\} \}}$.
    \item An edge $\{h_1, h_2\}$ is {\em consistent} if 
    $m_\gamma(h_1) = m_\gamma(h_2) = m_\gamma(\{h_1, h_2\})$.
\end{itemize}
\end{definition}

\begin{lemma}
\label{expectation_bound}
Let  
$\gamma = (e_1, \ldots, e_{2k})$ be a sequence of half edges
of even length, with
$M$ and $\bar{M}$ the matching matrix and its
centered version generated by a uniform matching in the 
configuration model. 
Then for $1 \leq k \leq \sqrt{|E|}$ and $0 \leq t_0 \leq k$
we have that
\[
\left| 
\E 
\prod_{t=1}^{t_0} \bar{M}_{e_{2t - 1} e_{2t}}
\prod_{t=t_0+1}^{k} M_{e_{2t - 1} e_{2t}}
\right|
\leq 
C \cdot 2^b \cdot 
\left( \frac{1}{|E|} \right)^\nedge
\left( \frac{3 k}{\sqrt{|E|}} \right)^{\nedge_1}
\]
where
$b = $ number of inconsistent edges of multiplicity one occuring
before $t_0$,
$\nedge_1 = $ number of consistent edges with multiplicity one
occuring before $t_0$,
$\nedge = |E(\gamma)|$,
and $C$ is a universal constant.
\end{lemma}

\begin{proof}
Recall the form of the matrices
\begin{equation*}
    M = 
    \left(
    \begin{array}{cc} 
        0 & M_1 \\
        M_1^* & 0 
    \end{array}
    \right )
    \quad \mbox{and} \quad
    \bar{M} = M - 
    \frac{1}{|E|} \left(
    \begin{array}{cc} 
        0 & \1 \1^* \\
        \1 \1^* & 0 
    \end{array}
    \right ).
\end{equation*}
Matrix $M_1 \in \R^{|E| \times |E|}$ is a random permutation matrix between 
$n d_1 = |E|$ and $m d_2 = |E|$ half edges.
Therefore, $M_1$ is distributed exactly the same as a matching matrix
of a random $|E|$-lift of a single edge,
and the same holds for its centered version $M_1 - \frac{1}{|E|} \1 \1^*$.
The only paths $\gamma$ that contribute in this bipartite setting must alternate
between the bipartite sets and avoid the 0 blocks,
otherwise the bound holds trivially.
For one of these paths $\gamma$
assume, without loss of generality, that the path starts in set $V_1$.
Then define the transformed path 
$\gamma' = (e_1', \ldots, e_{2k}') = (e_1, e_2, e_4, e_3, e_5, \ldots)$,
i.e.\ with every other pair in $\gamma$ in reverse order.
Note that
\begin{equation}
\prod_{t=1}^{t_0} \bar{M}_{e_{2t - 1} e_{2t}}
\prod_{t=t_0+1}^{k} M_{e_{2t - 1} e_{2t}}
=
\prod_{t=1}^{t_0} (\bar{M}_1)_{e'_{2t - 1} e'_{2t}}
\prod_{t=t_0+1}^{k} (M_1)_{e'_{2t - 1} e'_{2t}} .
\end{equation}
Then the Lemma holds by
\cite{bordenave2015}, Proposition 28.
\end{proof}

\subsection{Path counting}
\label{sec:path_counting}

This section is devoted to counting
the number of ways non-backtracking
walks can be concatenated to obtain a circuit as in 
Section~\ref{sec:matrix_decomp}.
We will follow closely the combinatorial analysis used in 
\citet*{brito2016}. 
In that paper, the authors needed a similar count for 
self-avoiding walks. 
We make the necessary adjustments to our current scenario.
% This is similar to the ``cycling times'' arguments of \citet*{bordenave2015a}.

Our goal is to find a reasonable bound for the number of circuits which
contribute to the trace bound, 
Eqn.~\eqref{eq:trace_bound} and shown graphically in Figure~\ref{fig:circuit}.
Define $\mathcal{C}_{\nvertex, \nedge}^\nvright$ as those circuits
which visit exactly 
$\nvertex = |V(\gamma)|$ different vertices, 
$\nvright = |V(\gamma) \cap V_2|$ of them in the right set, and
$\nedge = |E(\gamma)|$ different edges. Note, these are undirected edges in $E(G)$.
This is a set of circuits of length $2 k \ell$ obtained as the
concatenation of $2k$ non-backtracking,
tangle-free walks of length
$\ell$.
We denote such a circuit as
$\gamma= ( \gamma_1,\gamma_2,\cdots, \gamma_{2k} )$,
where each $\gamma_j$ is a length $\ell$ walk.

To bound $C_{\nvertex, \nedge}^\nvright = | \mathcal{C}_{\nvertex, \nedge}^\nvright |$,
we will first choose the set of vertices and order them.
The circuits which contribute are indeed directed non-backtracking walks.
However, by considering undirected walks along a fixed ordering of vertices,
that ordering sets the orientation of the first and thus the rest of the
directed edges in $\gamma$.
Thus, we are counting the directed walks which contribute to 
Eqn.~\eqref{eq:trace_bound}.
We relabel the
vertices as $1,2, \ldots, \nvertex$ as they appear in $\gamma$.
Denote by $\mathcal{T}_{\gamma}$ 
the spanning tree of those edges
leading to new vertices
as induced by the path $\gamma$.
The enumeration of the vertices tells us how we
traverse the circuit
and thus defines $\mathcal{T}_{\gamma}$ uniquely.

We encode each walk $\gamma_j$
by dividing it into sequences of subpaths of three types,
which in our convention 
{\it must always occur} as type 1 $\to$ type 2 $\to$ type 3,
although some may be empty subpaths.
Each type of subpath is encoded with a number, 
and we use the encoding to upper bound the number of such paths that can occur.
Given our current
position on the circuit, i.e.\ the label of the current vertex, and
the subtree of $\mathcal{T}_{\gamma}$ already discovered 
(over the whole circuit $\gamma$
not just the current walk $\gamma_j$), 
we define the types and their encodings:
\begin{enumerate}[{Type} 1:]
\item These are paths with the property that all of their edges
are edges of $\mathcal{T}_{\gamma}$ and have been traversed already in
the circuit. 
These paths can be encoded by their end vertex.
Because this is a path contained in a tree,
there is a unique path connecting its initial and final vertex. 
We use 0 if the path is empty.
\item These are paths with all of their
edges in $\mathcal{T}_{\gamma}$ 
but which are traversed for the first time in the circuit. 
We can encode these paths by their length,
since they are traversing new edges, and we know in what order the
vertices are discovered. 
We use 0 if the path is empty.
\item These paths are simply a single edge, 
not belonging to $\mathcal{T}_\gamma$,
that connects the end of a path of type 1 or 2 
to a vertex that has been already discovered. 
Given our position on the circuit, 
we can encode an edge by its final vertex. 
Again, we use 0 if the path is empty.
\end{enumerate}
Now, we decompose $\gamma_j$ into an ordered sequence of triples
to encode its subpaths:
\[ 
(p_1,q_1,r_1) (p_2,q_2,r_2) \cdots (p_t,q_t,r_t),
\]
where each $p_i$ characterizes subpaths of type 1, $q_i$ characterizes
subpaths of type 2, and $r_i$ characterizes subpaths of type 3.
These subpaths occur in the order given by the triples.
We perform this decomposition using the minimal possible number of triples.

Now, $p_i$ and $r_i$ are both numbers in $\{0,1,...,\nvertex \}$, since
our cycle has $\nvertex$ vertices. On the other hand, $q_i \in \{ 0,1,...,\ell \}$
since it represents the length of a subpath of a non-backtracking
walk of length $\ell$. 
Hence, there are  $(\nvertex + 1)^2 (\ell+1)$ possible triples. 
Next, we want to bound how many of these triples occur in $\gamma_j$. 
We will use the following lemma.

\begin{lemma} 
Let $(p_1,q_1,r_1) (p_2,q_2,r_2) \cdots (p_t,q_t,r_t)$ 
be a minimal encoding of a non backtracking walk $\gamma_j$, as described above. 
Then $r_i = 0$ can only occur in the last triple $i = t$.
\end{lemma}
\begin{proof} 
We can check this case by case. 
Assume that for some
$i<t$ we have $(p_i,q_i,0)$,
and consider the concatenation with
$(p_{i+1},q_{i+1},r_{i+1})$. 
First, notice that both $p_{i+1}$ and
$q_{i+1}$ cannot be zero since then we will have
$(p_i,q_i,0)(0,0,v^*)$ which can be written as $(p_i,q_i,v^*)$. 
If $q_i\neq 0$,
then we must have $p_{i+1} \neq 0$.
Otherwise, we split a path of new edges (type 2),
and the decomposition is not minimal.
This implies that we visit new edges and move to edges
already visited, hence we need to go through a type 3 edge, 
implying that $r_i \neq 0$. 
Finally, if $p_i \neq 0$ and $q_i = 0$, 
then we must have
$p_{i+1}=0$; 
otherwise, we split a path of old edges (type 1). 
We also require
$q_{i+1} \neq 0$, 
but
$(p_i,0,0)(0,q_{i+1},r_{i+1})$ is the same as $(p_i,q_{i+1},r_{i+1})$, 
which contradicts the minimality condition. 
This covers all possibilities and
finishes the proof.
\end{proof}

\begin{figure}
\centering
\includegraphics[width=0.75\linewidth]{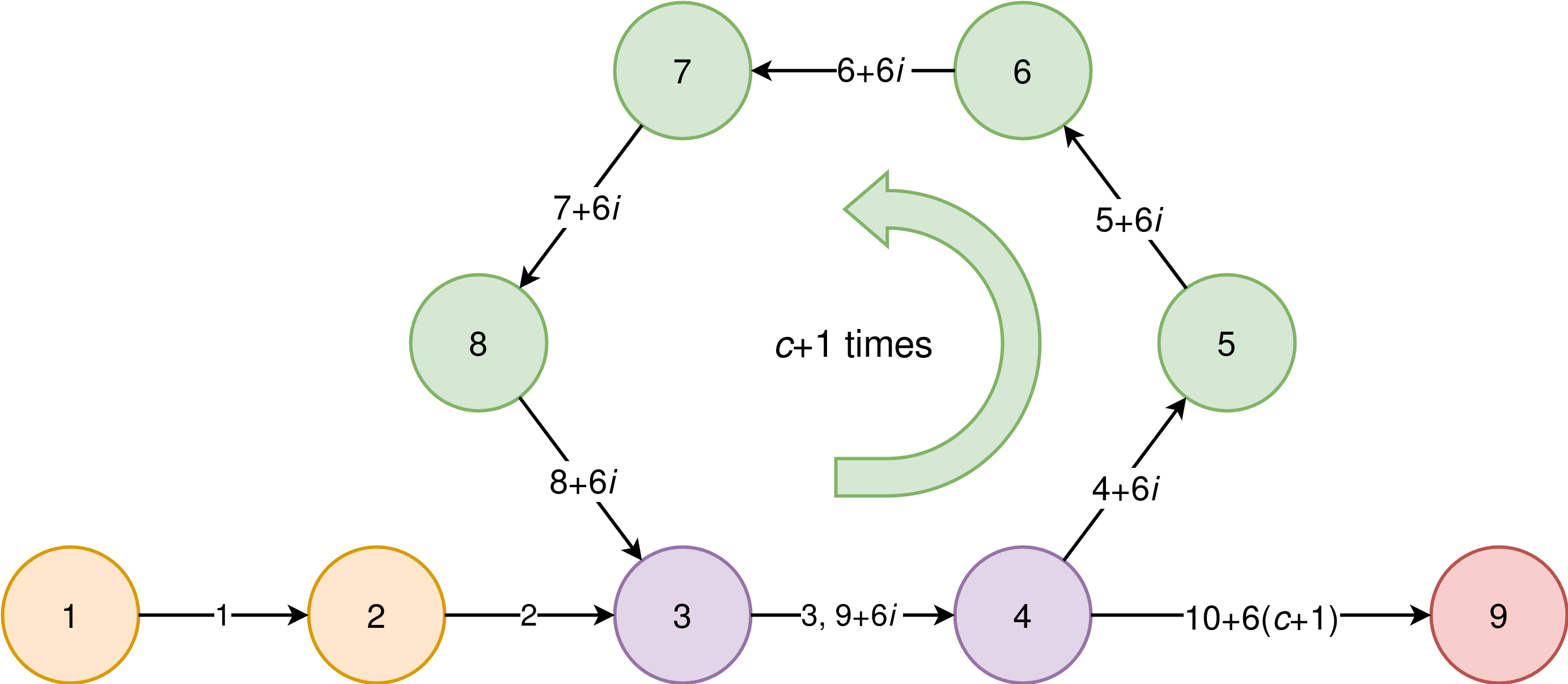}
\caption{
Encoding an $\ell$-tangle-free walk, 
in this case the first walk in the circuit $\gamma_1$, 
when it contains a cycle. 
The vertices and edges are labeled in the order of their traversal.
The segments $\gamma^a$, $\gamma^b$, and $\gamma^c$ occur on edges numbered
$(1, 2, 3)$; $(4 + 6i, 5+6i, 6+6i, 7+6i, 8+6i, 9+6i)$ for $i = 0, 1, \ldots c$;
and $(10+6c)$, respectively.
The encoding is $(0,3,0) | (0,4,3)(4,0,0) \| (0,1,0)$.
Suppose $c = 1$. 
Then $\ell = 22$
and the encoding is of length $3 + (4+1+1)(c+1) + 1$, 
we can back out $c$ to find that the cycle is repeated twice.
The encodings become more complicated later in the circuit as vertices 
see repeat visits. 
}
\label{fig:path_encoding}
\end{figure}

Using the lemma, any encoding of a non-backtracking walk
$\gamma_j$ has
at most one triple with $r_i=0$. 
All other triples indicate the
traversing of a type 3 edge. 
We now give a very rough upper bound for 
how many of such encodings there can be. 
To do so, we will use the tangle-free property and
slightly modify the encoding of the paths with cycles.
Consider the two cases:
\begin{enumerate}[{Case} 1:]
\item {Path $\gamma_j$ contains no cycle.} 
This implies that we traverse each edge within
$\gamma_j$ once. 
Thus, we can have at most $\chi= \nedge - \nvertex + 1$ many
triples with $r_i\neq 0$. 
This gives a total
of at most
\[ 
 \left((\nvertex+1)^2 (\ell+1) \right)^{\chi+1}
\]
many ways to encode one of these paths.
\item 
{Path $\gamma_j$ contains a cycle.} 
Since we are dealing with non-backtracking, tangle-free walks, 
we enter the cycle once, 
loop around some number of times, 
and never come back. 
We change the encoding of such paths slightly. 
Let $\gamma_j^{a}$, $\gamma_j^{b}$, and $\gamma_j^{c}$
be the segments of the path before, during, and after the cycle.
We mark the start of the cycle with $|$ and its end with $\|$. 
The new encoding of the path is:
\[ 
(p^a_1,q^a_1,r^a_1) \cdots (p^a_{t^a},q^a_{t^a},r^a_{t^a}) 
\, | \,
(p^b_1,q^b_1,r^b_1) \cdots (p^b_{t^b},q^b_{t^b},r^b_{t^b}) 
\, \| \,
(p^c_1,q^c_1,r^c_1) \cdots (p^c_{t^c},q^c_{t^c},r^c_{t^c}),
\]
where we encode the segments separately. 
Observe that each a subpath is connected and self-avoiding.
The above encoding tells
us all we need to traverse $\gamma_j$,
including how many times to loop around the cycle:
since the total length is $\ell$, 
we can back out the number of circuits around the cycle
from the lengths of $\gamma_j^{a}$, $\gamma_j^{b}$, and $\gamma_j^{c}$.
See Figure~\ref{fig:path_encoding}.
Following the analysis made for Case 1, 
the subpaths $\gamma_j^{a}$, $\gamma_j^{b}$, $\gamma_j^{c}$
are encoded by at most $\chi + 1$ triples,
but we also have at most $\ell$ choices each for our marks
$|$ and $\|$. 
We are left with at most
\[
 \ell^2 \left((\nvertex+1)^2 (\ell+1) \right)^{\chi+1}
\]
ways to encode any path of this kind.
\end{enumerate}
Together, these two cases mean there are less than
$
%2 \ell^2 \left((\nvertex+1)^2 (\ell+1) \right)^{\chi+1} \leq
2 \ell^2 \left((\nvertex+1)^2 (\ell+1) \right)^{\chi+1}
$
such paths.

Now we conclude by encoding the entire circuit 
$\gamma = (\gamma_1, \ldots, \gamma_{2k})$.
We first choose $\nvertex$ vertices, $\nvright$ in the set $V_2$, and order them,
which can occur in 
$(m)_\nvright (n)_{\nvertex-\nvright} \leq m^\nvright n^{\nvertex - \nvright}$ different ways. 
Finally, in the whole path $\gamma$ 
we are counting concatenations of
$2k$ paths which are $\ell$-tangle-free.
Therefore, we conclude with the following Lemma:

\begin{lemma}
\label{lem:circuit_count}
Let $\mathcal{C}_{\nvertex, \nedge}^\nvright$ 
be the set of circuits
$\gamma = (\gamma_1, \ldots, \gamma_{2k})$
of length $2 k \ell$ obtained as the
concatenation of $2k$ non-backtracking,
tangle-free walks of length
$\ell$,
i.e.\ $\gamma_s \in F^\ell$ for all $s \in [2k]$,
which visit exactly 
$\nvertex = |V(\gamma)|$ different vertices, 
$\nvright = |V(\gamma) \cap V_2|$ of them in the right set, and
$\nedge = |E(\gamma)|$ different edges.
If $C_{\nvertex, \nedge}^\nvright = |\mathcal{C}_{\nvertex, \nedge}^\nvright|$, then
\begin{equation}
C_{\nvertex, \nedge}^\nvright
\leq
m^\nvright n^{\nvertex - \nvright} (2 \ell)^{4k} 
\left( (\nvertex+1)^2 (\ell+1) \right)^{2k (\chi + 1)},
\end{equation}
where $\chi = \nedge - \nvertex + 1$.
\end{lemma}

The circuits that contribute to the remainder term
$R^{\ell, j}$ 
are slightly different.
In this case, each length $\ell$ segment
is an element of $T^{\ell, j}$ rather than $F^{\ell}$.
We have to slightly modify the previous argument for this case.

\begin{lemma}\label{lem:combinatorial_of_reminder}
Let $\mathcal{D}_{\nvertex,\nedge}^\nvright$ 
be the set of circuits
$\gamma = (\gamma_1, \ldots, \gamma_{2k})$
of length $2k\ell$
obtained as the concatenation
of $2k$ elements 
$\gamma_s \in T^{\ell, j}$
for $s = 1, \ldots, 2k$,
that visit exactly $\nvertex$ vertices, 
$\nvright$ of which are in $V_2$, and $\nedge$ different edges. 
Then for $D_{\nvertex,\nedge}^\nvright = | \mathcal{D}_{\nvertex,\nedge}^\nvright |$,
we have
\begin{align}
\label{bound:D}
D_{\nvertex,\nedge}^\nvright
\leq
m^\nvright n^{\nvertex-\nvright}
(2 \ell)^{6k} 
( (\nvertex + 1)^2 (\ell + 1) )^{6k(\chi + 1)} .
\end{align}
\end{lemma}

\begin{proof}
Since each 
$\gamma^s = (e_1, \ldots, e_{2\ell + 1}) \in T^{\ell,j}$,
we have that
$\gamma' = (e_1, \ldots, e_{2j-1}) \in F^{j-1}$,
$\gamma'' = (e_{2j-1}, e_{2j}, e_{2j+1}) \in F^{1}$,
and $\gamma''' = (e_{2j+1}, \ldots, e_{2\ell + 1}) \in F^{\ell - j}$.
Encoding $\gamma'$, $\gamma''$, and $\gamma'''$ as before,
we have the generous upper bound of at most
\[
\left[ (2\ell) ((\nvertex + 1)^2(\ell+1))^{\chi+1} \right]^3
\]  
many encodings for each $\gamma_s$. 
Choosing and ordering the vertices, then concatenating $2k$ of these paths gives
the final result.
\end{proof}

\subsection{Half edge isomorphism counting}
   
We have constructed the circuits in 
$\mathcal{C}_{\nvertex, \nedge}^\nvright$
and
$\mathcal{D}_{\nvertex, \nedge}^\nvright$
by choosing the vertices and edges that participate in them.
However, the expectation bound applies to matchings
of half-edges in the configuration model.
Since there are multiple ways to configure the half edges 
into such a circuit, this must
be taken into account in the combinatorics.

\begin{lemma}
\label{lem:edge_iso}
Let $I_{\nvertex, \nedge}^\nvright$
be the number of half edge choices for the graph
induced by
$\gamma \in \mathcal{C}_{\nvertex, \nedge}^\nvright
\cup \mathcal{D}_{\nvertex, \nedge}^\nvright$.
Then,
\begin{equation}
    I_{\nvertex, \nedge}^\nvright 
    \leq 
    d_1^{\nvertex - \nvright} 
    (d_1 - 1)^{\nedge - \nvertex + \nvright}
    d_2^{\nvright}
    (d_2 - 1)^{\nedge - \nvright} .
    % (d_1 (d_1 - 1))^{\nvertex - \nvright} 
    % (d_2 (d_2 - 1))^{\nvright}
    % d^{2 (\chi + k)},
\end{equation}
\end{lemma}

\begin{proof}
% First, consider the simple case where every vertex has degree at most two,
% i.e.\ the circuit $\gamma$ is a path.
% For a left vertex we will have $d_1 (d_1 - 1)$ choices for 
% the two half edges which are used in its outgoing edges, 
% and there are $\nvertex - \nvright$ many of them,
% which leads to the term $(d_1 (d_1 - 1))^{\nvertex - \nvright}$.
% Similarly, every right vertex has $d_2 (d_2 - 1)$ choices for 
% the two half edges, 
% and there are $\nvright$ many of them,
% which gives the term $(d_2 (d_2 - 1))^{\nvright}$.

% Now consider the case when some vertices have degree greater than two. 
% These occur due to either cycles in the circuit
% or backtracking.
% Recall that $\chi$ corresponds to the number of 
% fundamental cycles in the circuit.
% Every fundamental cycle either adds one to the degree of two vertices
% or adds two to the degree of one vertex.
% Either of these cases has at most $d^2$ choices for the half edges,
% giving a term of $d^{2\chi}$.
% Every backtrack can add one to the degree of at most one vertex,
% and since the circuit can backtrack at most $2k$ times
% this leads to a term of $d^{2k}$.

For every left vertex $v$, with degree $g_v$ on the graph induced by $\gamma$, 
the number of choices of half edges is $d_1(d_1-1)\dots(d_1- g_v+1)\leq d_1 (d_1-1)^{g_v-1}$. 
Note that the choice of half edges are independent for the left vertices. 
We then get that there are 
$d_1^{\nvertex-\nvright} (d_1-1)^{\nedge - \nvertex+\nvright}$ 
many choices, 
where we used that the sum of all the degrees on one component of a bipartite graph 
equals the number of edges: $\nedge = \sum_v g_v$. 
Similarly, for right vertices we get 
$d_2^{\nvright}(d_2-1)^{(\nedge-\nvright)}$.
\end{proof}

\begin{corollary}
\label{cor:edge_iso}
We have that
\begin{equation}
    I_{\nvertex, \nedge}^\nvright \leq 
    (d_1 (d_1 - 1))^{\nvertex - \nvright}
    (d_2 (d_2 - 1))^\nvright
    (d - 1)^{2 (\chi - 1)}.
\end{equation}
\end{corollary}

\subsection{Bounding the imbalance $\psi$}

We focus now on the quantity defined as 
$\psi = \nvright - \nedge / 2 $. 
Informally, $\psi$ captures the {\it imbalance} between the number of vertices 
on each partition of the bipartite graph visited by the circuit $\gamma$. 
We show that this imbalance is not too large.
\begin{lemma}\label{lem:imbalance}
Let $\ell < \frac{1}{32} \log_d(n)$,
then $\psi \leq 16 k^2$ with high probability.
\end{lemma}
\begin{proof}
For any subgraph $H$ define $\psi(H) = \nvright(H)-\nedge(H)/2$. 
We set $\psi=\psi(\gamma)$. 
To bound this quantity, we analyse the subgraph 
$\gamma_{\leq i}$, 
obtained by the concatenation of the first $i$ walks in $\gamma$,
i.e.\ the union of the graphs induced by $\gamma_1, \ldots, \gamma_i$.
Our choice of $\ell$ 
implies that every neighborhood of radius $4\ell$ is tangle-free with high probability. Hence, every non-backtracking walk $\gamma_j$ 
is either a path or a path with exactly one loop.
It is not hard to conclude that 
$\psi(\gamma_j) \leq 2$ for all $j$ and 
$\psi(\gamma_{\leq 1}) = \psi(\gamma_1) \leq 2$. 
We now proceed inductively to add walks to our graph, one by one, 
as they appear on the circuit.
We will upper bound the increment 
$\psi(\gamma_{\leq i+1}) - \psi(\gamma_{\leq i})$ 
by looking at how the addition of $\gamma_{i+1}$ changes the imbalance. 

To analyse this, consider the intersection of 
$\gamma_{i+1}$ and each $\gamma_j$, $1\leq j\leq i$. 
Notice that $\psi$ may increase only if there are vertices 
at which the two walks split apart. 
We claim that there are at most two such vertices. 
Assume that at $v_1,v_2$ and $v_3$ the two walks split. 
Then there are two disjoint cycles in the union of $\gamma_{i+1}$ and $\gamma_j$, 
obtained by following each first from $v_1$ to $v_2$ 
and then from $v_2$ to $v_3$. 
But this is a contradiction, 
since the diameter of this union is less than $2 \ell < \frac{1}{8} \log(n)$,
which implies that their union is tangle-free. 
We conclude that $\psi(\gamma_{i+1} \cup \gamma_j) \leq 8$
since there are at most two splits and 
each split contributes at most four to the imbalance. 
Then
$\psi(\gamma_{\leq i+1}) \leq \psi(\gamma_{\leq i}) + 8i + 2$,
which implies that $\psi(\gamma)\leq 16k^2$, as desired.
\end{proof}

\subsection{Bounding the inconsistent edges} 
We will need a bound on the number of inconsistent edges of multiplicity one,
which we get in the following lemma. Recall Definition~\ref{def:incon_edges}, which introduced inconsistent edges. 

\begin{lemma} 
Let $b_\mathcal{C}$ 
denote the number of inconsistent edges of multiplicity one 
on a circuit 
$\gamma = (\gamma_1, \ldots, \gamma_{2k})$ 
consisting of $2k$ non-backtracking walks of length $\ell$ each. 
It holds that
$ b_\mathcal{C} \leq 4 (k + \chi)$, 
where $\chi=\nedge-\nvertex+1$.
\label{lem:inconsistent}
\end{lemma}
\begin{proof}
Let $\{e,f\}$ be an inconsistent edge of multiplicity one, where
$e$ and $f$ are its half edges.
For inconsistency and without loss of generality,
there must exist 
another edge
$\{e,f'\}$ in $\gamma$,
so that
$m_\gamma (e) \neq 1$.
We may assume that $\{e,f\}$ is traversed before $\{e,f'\}$. 
Let $v$ be the vertex of $e$ and consider the two possible scenarios:
\begin{enumerate}[{Case} 1:]
    \item There is no cycle containing $v$ in $\gamma$.
    Then the edge $\{e,f\}$ may only be inconsistent
    if $v$ is visited 
    at the end of one of the $2k$ non-backtracking walks $\gamma_i$ and
    $\{e,f'\}$ is at the beginning of $\gamma_{i+1}$. 
    Hence, in this case we have at most two inconsistent edges of multiplicity one. 
    This yields at most $4k$ such edges.
    \item There is a cycle passing through $v$. 
    For each such cycle there is an edge that does not belong to the tree 
    $T_{\gamma}$ (defined in Section~\ref{sec:path_counting}). 
    Furthermore, each cycle creates at most four inconsistent edges. 
    Combining these two facts we get at most $4\chi$ and the proof follows.
\end{enumerate}
\end{proof}

\begin{lemma}\label{lem:inconsistent_edges_reminder}
Consider a circuit $\gamma = (\gamma_1, \ldots, \gamma_{2k})$
with $\gamma_s \in T^{\ell,j}$ for $j \in [2k]$,
with $\gamma_s$ decomposed as $\gamma_s', \gamma_s'', \gamma_s'''$
as in the Definition~\ref{def:tangled_paths}.
Let 
$b_\mathcal{D}$ 
denote the number of inconsistent edges of multiplicity one 
in the union of segments
$\bigcup_{s=1}^{2k} (\gamma_s' \cup \gamma_s'')$.
Then 
$b_\mathcal{D} 
\leq 
16k+4 \chi$, 
where
$\chi = \nedge - \nvertex + 1$.
\label{lem:inconsistent_tangled}
\end{lemma}

\begin{proof}
The argument is similar to the above; however, now there are $4k$ segments
in $\bar\gamma = \bigcup_{s=1}^{2k} (\gamma_s' \cup \gamma_s'')$,
counting $\gamma_s'$ and $\gamma_s''$ separately.
As above, each of these $4k$ segments may yield at most 2 inconsistent edges.
Furthermore, the graph induced by $\bar\gamma$ may not be connected;
let $C$ be the number of connected components.
Each edge that creates a cycle may yield at most 4 inconsistent edges,
and there are at most $\nedge - \nvertex + C$ non-tree edges.
Then, we have that
\[
b_\mathcal{D} 
\leq 8 k + 4 (\nedge - \nvertex + C) 
\leq 8k + 4 (\nedge - \nvertex + 2k)
\leq 16 k + 4 \chi,
\]
as claimed.
\end{proof}

% \textcolor{red}{Is the last equality an inequality? (because of the $-1$ on $\chi$). Also, the proof yields a 16, why do we want a 24? I think the 16 is fine.}
% \todo{For some reason Bordenave, Lemma 21 has 
% the constant 24, 
% do you see why he gets this constant? The constant doesn't end up mattering, however.} \textcolor{red}{He uses $p-1< p\leq 2m$ and gets the 24. In our case we are fine with $16$, like you already pointed out.}

\subsection{Bounds on the norm of $\bar{B}^{\ell}$ and $R^{\ell,j}$.}
\label{sec:norm_bounds}

All of the ingredients are gathered to bound the matrix norms.

\begin{theorem}
\label{spec_B_bar}
Let 
$\ell = \lfloor c\log(n) \rfloor$ 
where $c < \frac{1}{32}$ 
is a universal constant. 
It holds that
\[
\|\bar{B}^{(\ell)}\|
\leq 
\left((d_1-1)(d_2-1)\right)^{\ell/4}
\exp( \log^{3/4} n )
\]
asymptotically almost surely.
\end{theorem}

\begin{proof}
The following holds for any natural number $k$, but for our proof, we will take
\begin{equation}
 k = \lfloor \log (n)^{1/3} \rfloor \quad \mbox{ and } \quad
 \ell = \lfloor c \log n \rfloor \mbox{ for some } c < \frac{1}{32}.
\end{equation}
We have
\begin{align}
\label{trace_expectation}
\E \left(\| \bar{B}^{(\ell)} \|^{2k} \right)
\leq 
\E
\left( 
\mathrm{Tr} 
 \left[ 
 \left(
 \bar{B}^{(\ell)} \bar{B}^{(\ell)^*} 
 \right)^{k} 
 \right]
\right)
=
\E\left( 
\sum_{\gamma \in \mathcal{C}} 
\prod_{t=1}^{2 k \ell} \bar{B}_{e_{2t-1} e_{2t+1}}
\right).
\end{align}
The sum is taken over the set $\mathcal{C}$ of all circuits 
$\gamma$
of length $2k\ell$,
where
$\gamma= ( \gamma_1, \gamma_2, \ldots, \gamma_{2k} )$ 
is formed by concatenation of $2k$ tangle-free segments
$\gamma_s \in F^{\ell}$, 
with the convention $e^{s+1}_1=e^s_{\ell+1}$.
Again, refer to Figure~\ref{fig:circuit} for clarification.

As in Section~\ref{sec:path_counting}, we will break these into circuits
which visit exactly 
$\nvertex = |V(\gamma)|$ different vertices, 
$\nvright = |V(\gamma) \cap V_2|$ of them in the right set, and
$\nedge = |E(\gamma)|$ different edges.
We define three disjoint sets of circuits:
\begin{align*}
%\begin{itemize}
	%\item 
    \mathcal{C}_1 & =
    \{\gamma \in \mathcal{C}:~\mbox{all edges in $\gamma$ are traversed at least twice} \}, \\
	%\item 
    \mathcal{C}_2 & =
    \{\gamma \in \mathcal{C}:~\mbox{at least one edge in $\gamma$ is traversed exactly once 
    and $\nvertex \leq kl+1$} \},~\mbox{and} \\
	%\item 
    \mathcal{C}_3 & =
    \{\gamma \in \mathcal{C}:~\mbox{at least one edge in $\gamma$ is traversed exactly once 
    and $\nvertex > kl+1$} \} .
\end{align*}
Define the quantities
\[
I_j
= 
\left|
\E
  \sum_{\gamma\in \mathcal{C}_j} 
  \prod_{t=1}^{2k\ell} \bar{B}_{e_{2t-1} e_{2t+1}} 
\right|
\leq
\sum_{\gamma \in \mathcal{C}_j}
\left|
\E
\prod_{t=1}^{2k\ell}
\bar{M}_{e_{2t-1} e_{2t}}
\right|
\]
for $j = 1, 2$ and 3, so that
\eqref{trace_expectation}
can be bounded as
\begin{align}\label{bound:I1+I2+I3}
\E\left(\|\bar{B}^{(\ell)}\|^{2k}\right)\leq I_1+I_2+I_3.
\end{align}

We will bound each term on the right hand side above.
The reason for this division is that, 
by Theorem~\ref{expectation_bound}, 
when we have any two-path traversed exactly once,
the expectation of the corresponding circuit is smaller,
because the matrix $\bar{B}$ is nearly centered.
We will see that the leading order terms in Eqn.~\eqref{trace_expectation} 
will come from circuits in $\mathcal{C}_1$. 
From Lemmas~\ref{expectation_bound} and
\ref{lem:circuit_count} and Corollary~\ref{cor:edge_iso}, we get that
\begin{align}
\nonumber
I_j
&\leq
\sum_{\nvertex,\nedge,\nvright}
\sum_{\gamma\in \mathcal{C}_j \cap \mathcal{C}_{\nvertex,\nedge}^\nvright} 
\left|
\mathbb{E} 
\prod_{t=1}^{2k\ell} \bar{M}_{e_{2t-1} e_{2t}}
\right|
\\
&\leq
\sum_{\nvertex,\nedge,\nvright}
C_{\nvertex,\nedge}^\nvright \,     
I_{\nvertex, \nedge}^\nvright \,
\left|
\mathbb{E} 
\prod_{t=1}^{2k\ell} \bar{M}_{e_{2t-1} e_{2t}}
\right|
\nonumber
\\
&\leq
\sum_{\nvertex,\nedge,\nvright}
n^{\nvertex - \nvright}
m^{\nvright}
(2\ell)^{4k} ((\nvertex + 1)^2 (\ell + 1))^{2k(\chi + 1)} \nonumber \\
& \qquad\qquad
\cdot
(d_1 (d_1 - 1))^{\nvertex - \nvright}
(d_2 (d_2 - 1))^\nvright
(d - 1)^{2(\chi - 1)} \nonumber \\
& \qquad\qquad
\cdot 
C
2^{b_\mathcal{C}} \left( \frac{1}{|E|} \right)^\nedge
\left( \frac{6 k \ell}{\sqrt{|E|}} \right)^{\nedge_1} \nonumber \\
&\leq
C
\sum_{\nvertex,\nedge,\nvright}
(2\ell)^{4k} ((\nvertex + 1)^2 (\ell + 1))^{2k(\chi + 1)}
(d_1 - 1)^{\nvertex - \nvright}
(d_2 - 1)^\nvright
(d - 1)^{2(\chi - 1)} \nonumber \\
& \qquad\qquad
\cdot 
2^{b_\mathcal{C}} 
\left( \frac{1}{|E|} \right)^{\nedge - \nvertex}
\left( \frac{6 k \ell}{\sqrt{|E|}} \right)^{\nedge_1} \nonumber \\
&\leq
C
\sum_{\nvertex,\nedge,\nvright}
(2\ell)^{4k} ((\nvertex + 1)^2 (\ell + 1))^{2k(\chi + 1)}
\left( (d_1 - 1)(d_2 - 1) \right)^{\nedge / 2}
\nonumber \\
& \qquad\qquad
\cdot
\left( \frac{d_2 - 1}{d_1 - 1} \right)^\psi
(d_1 - 1)^{1-\chi}
(d - 1)^{2(\chi - 1)}
2^{b_\mathcal{C}}
\left( \frac{1}{|E|} \right)^{\nedge - \nvertex}
\left( \frac{6 k \ell}{\sqrt{|E|}} \right)^{\nedge_1} 
\nonumber \\
&\leq
c_0 n
\ell^{4k} 
c_1^{k^2}
c_2^k
\sum_{\nvertex,\nedge}
\alpha^{2\nedge}
\nvertex
((\nvertex + 1)^2 (\ell + 1))^{2k(\chi + 1)}
\left( \frac{c_3}{n} \right)^\chi
\left( \frac{c_4 k \ell}{\sqrt{n}} \right)^{\nedge_1} 
\label{bound:I_j}
\end{align}
We use $C, c_0, c_1, c_2, c_3, c_4$ to denote constant terms
and set $\alpha = ((d_1-1)(d_2-1))^{1/4}$.
In the last line we used Lemma~\ref{lem:imbalance}
and Lemma~\ref{lem:inconsistent}
to bound 
$\psi$ and $b_\mathcal{C}$ 
in terms of 
$k$ and $\chi$
and remove the sum over $\nvright$,
which contains at most $\nvertex$ terms.
We will use Eqn.~\ref{bound:I_j} to bound each $I_j$. 

% Now we will separate the two sums and change variables to $\chi$,
% yielding \textcolor{red}{(check what bound we got for $b$.)}
% \begin{equation}
% I_j \leq
% c_0 n
% \ell^{4k} 
% c_1^{k^2}
% c_2^k
% \sum_{\nvertex}
% \alpha^{2 (\nvertex - 1)}
% \nvertex
% \sum_\chi
% ((\nvertex + 1)^2 (\ell + 1))^{2k(\chi + 1)}
% \left( \frac{c_3}{n} \right)^\chi
% \left( \frac{c_4 k \ell}{\sqrt{n}} \right)^{\nedge_1} .
%     \label{bound:I_j}
% \end{equation}
% \textcolor{red}{My only suggestion will be to omit equation 23.}
% \todo{I was thinking to absorb this into $c_3$}
% 

\subsubsection{Bounding $I_1$}

Here, $\nedge_1 = 0$ since every edge is traversed twice.
% In all cases, each circuit traverses $2k \ell$ \ps. 
% Hence, for each $\gamma\in \mathcal{C}_1$, 
% where each \2 is repeated twice,
% we have at most $k\ell$ different \ps. 
% Furthermore, since each edge can be in multiple \ps, 
% we have that the total number of different \ps 
% ~is greater than or equal to the total number of edges traversed by $\gamma$. 
We then have that 
$\nvertex - 1 \leq \nedge \leq k\ell$. 
Since $\gamma$ is connected,
we have 
$1 \leq \nvertex \leq k\ell+1$. 
Thus, on the right hand side of 
Eqn.~\eqref{bound:I_j} we get 
\begin{align*}
I_1 
&\leq 
c_0 n
\ell^{4k} 
c_1^{ k^2}
c_2^k
\sum_{\nvertex = 1}^{k\ell+1} 
\sum_{\nedge=\nvertex-1}^{k\ell} 
\alpha^{2\nedge}
\nvertex
((\nvertex + 1)^2 (\ell + 1))^{2k(\chi + 1)}
\left( \frac{c_3}{n} \right)^\chi
\\
&=
c_0 n
\ell^{4k} 
c_1^{ k^2}
c_2^k
\sum_{\nvertex = 1}^{k\ell+1} 
\alpha^{2(\nvertex - 1)}
\nvertex
((\nvertex + 1)^2 (\ell + 1))^{2k}
\sum_{\chi= 0}^{k\ell - \nvertex + 1} 
\left( 
    \alpha^2
    ((\nvertex + 1)^2 (\ell + 1))^{2k} 
    \frac{c_3}{n} 
\right)^\chi
\end{align*}
The second sum
is upper-bounded by 
$
\sum_{\chi=0}^\infty 
\left( C \frac{ ((\nvertex+1)^2 (\ell+1) )^{2k} }{n} \right)^\chi
$,
where $C$ is some constant,
which we show next is bounded by a common constant for our choices of $k$ and $\ell$ 
and all $\nvertex \in [k\ell + 1]$. 
To see this, it will suffice to show that 
$((\nvertex+1)^2 (\ell+1) )^{2k}=o(n)$. 
But $((\nvertex+1)^2 (\ell+1) ) = O(k^2 \ell^3)$
which, for our choices of $k$ and $\ell$, yields
\[
2k \log( (\nvertex+1)^2 (\ell+1) )=O({\log(n)^{1/3}}\log(\log(n)))=o(\log(n))
\]
as desired.

Finally, the first summand is maximized for $\nvertex = k\ell + 1$ and there are at $k\ell + 1$ many terms in that sum.
Therefore, modifying the constant $c_0$ yields
\begin{align}\label{bound:I_1}
I_1
\leq 
c_0' n 
\ell^{4k}
c_1^{ k^2}
c_2^k
(k \ell + 1)^2
\left( (k\ell+2)^2 (\ell+1) \right)^{2k}
\alpha^{2 k \ell}.
\end{align}

\subsubsection{Bounding $I_2$}

Here there is at least one edge traversed exactly once, so
we have $\nedge \geq \nvertex$ for $\gamma\in \mathcal{C}_2$. 
Taking $\nedge_1 = 0$ only increases the right hand side on
Eqn.~\eqref{bound:I_j}; it becomes
\begin{align*}
I_2
&\leq 
c_0 n
\ell^{4k} 
c_1^{ k^2}
c_2^k
\sum_{\nvertex=1}^{k\ell+1} \sum_{\nedge=\nvertex}^{2k\ell} 
\alpha^{2\nedge}
\nvertex
((\nvertex + 1)^2 (\ell + 1))^{2k(\chi + 1)}
\left( \frac{c_3}{n} \right)^\chi
\\
&=
c_0 n
\ell^{4k} 
c_1^{ k^2}
c_2^k
\sum_{\nvertex = 1}^{k\ell+1} 
\alpha^{2(\nvertex-1)}
\nvertex
((\nvertex + 1)^2 (\ell + 1))^{2k}
\sum_{\chi= 1}^{2k\ell - \nvertex + 1} 
\left( 
    \alpha^2
    ((\nvertex + 1)^2 (\ell + 1))^{2k} 
    \frac{c_3}{n} 
\right)^\chi
\end{align*}
Notice that this last term is almost identical to the one in the bound of $I_1$, 
except that now we start the second sum at $\chi=1$,
which leads to an extra factor of
$
O( ((\nvertex + 1)^2 (\ell + 1))^{2k} / n )
$.
This allow us to factor out another geometric series
and proceed as we did for $I_1$. 
This yields 
\begin{align}
I_2 
\leq 
c_0'
\ell^{4k}
c_1^{ k^2}
c_2^k
(k \ell+1)^2
\left( (k\ell+2)^2 (\ell+1) \right)^{4k}
\alpha^{2k\ell},
\label{bound:I_2}
\end{align} 
since there are $k\ell+1$ terms in the first sum.

\subsubsection{Bounding $I_3$}

This set will require more delicate treatment,
since circuits in $\mathcal{C}_3$ 
visit potentially many vertices and edges, 
yet we need to keep the power of $\alpha$ at most $2k\ell$.

We first show that, in this case, $\nedge_1$ is also large. 
We have 
$\nedge \geq \nvertex$,
and let $\nvertex = k\ell + t$.
Define
$\nedge_1'$ as the number of edges 
traversed once in $\gamma$, 
so that 
$\nedge_1' = b + \nedge_1$.
Since $\gamma$ has length $2k \ell$, we deduce that
$2 (\nedge - \nedge'_1) + \nedge'_1 \leq 2k\ell$,
which implies that $\nedge'_1 \geq 2t$.
Finally, Lemma~\ref{lem:inconsistent}
yields
$\nedge_1 \geq (2t - 4 (\chi + k))_+$.
Eqn.~\eqref{bound:I_j} then gives,
\begin{align*}
I_3 
&\leq 
c_0 n
\ell^{4k} 
c_1^{ k^2}
c_2^k
\sum_{\nvertex=k\ell+1}^{2k\ell} 
\sum_{\nedge=\nvertex}^{2k\ell} 
\alpha^{2\nedge}
\nvertex
((\nvertex + 1)^2 (\ell + 1))^{2k(\chi + 1)}
%& \qquad\qquad \cdot
\left( \frac{c_3}{n} \right)^\chi 
\left( \frac{(c_4 k \ell)^2}{n} \right)^
{(\nvertex - k\ell - 2(\chi+k))_+} 
\nonumber \\
&=
c_0 n
\ell^{4k} 
c_1^{ k^2}
c_2^k
\sum_{t=1}^{k\ell} 
\sum_{\chi=1}^{k\ell - t+1} 
\alpha^{2 (k\ell + \chi + t-1)}
(k\ell + t)
((k\ell + t + 1)^2 (\ell + 1))^{2k(\chi + 1)}\\ 
& \qquad\qquad\qquad\qquad\qquad\qquad\qquad\qquad\qquad\qquad \cdot
\left( \frac{c_3}{n} \right)^{\chi} 
\left( \frac{(c_4 k \ell)^2}{n} \right)^{(t - 2(\chi+k))_+}
\nonumber \\
&=
c_0
n
\ell^{4k} 
c_1^{ k^2}
c_2^k
\alpha^{2 (k\ell-1)}
% \nonumber \\
% & \qquad\qquad \cdot
\sum_{t=1}^{k\ell} 
(k\ell + t)
((k\ell + t + 1)^2 (\ell + 1))^{2k}
\alpha^{2t} 
\\
& \qquad\qquad\qquad\qquad\qquad\qquad \cdot
\sum_{\chi=1}^{k\ell - t+1}
\left( \alpha^2 ((k\ell + t + 1)^2 (\ell + 1))^{2k} \frac{c_3}{n} \right)^{\chi} 
% \left( C'
% \frac{ (C_{\nvertex}\ell)^{2k}}{n}
% \right)^{\chi} 
\left( \frac{(c_4 k \ell)^2}{n} \right)^{(t - 2(\chi+k))_+}.
\end{align*}

To simplify our notation, we will write 
\[
F(k,\ell)=c_0
n
\ell^{4k} 
c_1^{ k^2}
c_2^k
\alpha^{2 (k\ell-1)}
(2k\ell)
((2k\ell + 1)^2 (\ell + 1))^{2k}.
\]
Observe now that we can write:
\[
I_3\leq F(k,\ell)\sum_{t=1}^{k\ell} \alpha^{2t}\sum_{\chi=1}^{k\ell-t+1} \left(\frac{c(k,\ell)}{n}\right)^{\chi} \left(\frac{(c_4k\ell)^2}{n}\right)^{(t-2k-2\chi)_+}
\]
where $c(k,\ell)=c_3\alpha^2((2k\ell+1)^2(\ell+1))^{2k}$.
To bound the double sum on the right hand side above, 
we start by removing a factor of $\frac{c(k,\ell)}{n}$, 
whichs leaves
\begin{equation}\label{I_3bound}
I_3\leq F(k,\ell)\frac{c(k,\ell)}{n} \sum_{t=1}^{k\ell} \alpha^{2t}\sum_{\chi=0}^{k\ell-t} \left(\frac{c(k,\ell)}{n}\right)^{\chi} \left(\frac{(c_4k\ell)^2}{n}\right)^{(t-2k-2-2\chi)_+}
\end{equation}
The $n$ in the denominator is crucial to cancel the linear term in $F(k,\ell)$, keeping the upper bound for $I_3$ small. We focus on bounding the double sum. We split the sum in $t$ in two parts.\\
{\bf Case 1: $t< 2k+2$.} For these values of $t$, we have $(t-2k-2-2\chi)_+=0$, hence
\begin{equation}\label{I_3:eq1}
\sum_{t=1}^{2k+1} \alpha^{2t}\sum_{\chi=0}^{k\ell-t} \left(\frac{c(k,\ell)}{n}\right)^{\chi} \left(\frac{(c_4k\ell)^2}{n}\right)^{(t-2k-2-2\chi)_+}=\sum_{t=1}^{2k+1} \alpha^{2t}\sum_{\chi=0}^{k\ell-t} \left(\frac{c(k,\ell)}{n}\right)^{\chi}=O(\alpha^{4k})
\end{equation}
where the last equality uses again the same geometric upper bound for the sum over $\chi$.\\
{\bf Case 2: $t\geq 2k+2$.} 
We split the second sum, from $\chi=0$ to $N=\lfloor t/2-k-1\rfloor$ and
the terms with $\chi > N$ and analyse the two separately. 
The first can be upper bounded by
\[
\sum_{t=2k+2}^{k\ell} \alpha^{2t}\sum_{\chi=0}^{N} 
\left(\frac{c(k,\ell)}{n}\right)^{\chi} 
\left(\frac{c_4 k \ell}{\sqrt{n}}\right)^{(t/2-k-1-\chi)}
\leq
\alpha^{4k+4} \sum_{t=0}^{k\ell-2k-2}
\alpha^{2t} (N+1)
\left( \frac{c_4 k \ell}{\sqrt{n}} \right)^{\frac{t}{2}}.
\]
The last inequality can be checked in two steps: 
We first factor out the power of $\alpha^{4k+4}$, 
and then use that 
$\frac{c(k,\ell)}{n}\leq \frac{c_4 k \ell}{\sqrt{n}}$,
which holds for large enough $n$,
to simplify the second sum to the addition of $N+1$ equal terms. 
To bound the right hand side, 
we one more time upper bound by a 
geometric series of ratio less than one to get
\begin{equation}\label{I_3eq2}
\sum_{t=2k+2}^{k\ell} \alpha^{2t}\sum_{\chi=0}^{N} 
\left(\frac{c(k,\ell)}{n}\right)^{\chi} 
\left(\frac{c_4 k\ell}{\sqrt{n}}\right)^{(t/2-k-1-\chi)}
\leq
O(k\ell\alpha^{4k}).
\end{equation}
We are left with the terms $\chi > N$. In this case we get $(t-2k-2-2\chi)_+=0$, so
\[
\sum_{t=2k+2}^{k\ell} \alpha^{2t}\sum_{\chi=N+1}^{k\ell-t}
\left(\frac{c(k,\ell)}{n}\right)^{\chi} 
\left(\frac{c_4 k \ell}{\sqrt{n}}\right)^{(t/2-k-1-\chi)_+}
=
\sum_{t=2k+2}^{k\ell} \alpha^{2t}\sum_{\chi=N+1}^{k\ell-t} \left(\frac{c(k,\ell)}{n}\right)^{\chi}.
\]
The sum over $\chi$ is of the order of $\left(\frac{c(k,\ell)}{n}\right)^{N+1}\leq \left(\frac{c(k,\ell)}{n}\right)^{t/2-k-1}$. 
Substituting this into the above, we are left with
\[
\sum_{t=2k+2}^{k\ell} \alpha^{2t}\sum_{\chi=N+1}^{k\ell-t} \left(\frac{c(k,\ell)}{n}\right)^{\chi}
=
C\sum_{t=2k+2}^{k\ell} \alpha^{2t}\left(\frac{c(k,\ell)}{n}\right)^{t/2-k-1}
\]
for some universal constant $C$. After factoring $\alpha^{4k+4}$ and changing variables in the summation, we conclude that
\begin{equation}\label{I_3eq3}
\sum_{t=2k+2}^{k\ell} \alpha^{2t}\sum_{\chi=N+1}^{k\ell-t} \left(\frac{c(k,\ell)}{n}\right)^{\chi}
=
C\alpha^{4k+4}\sum_{t=0}^{k\ell-2k-2} \alpha^{2t}\left(\frac{c(k,\ell)}{n}\right)^{t/2}
=
O(\alpha^{4k}).
\end{equation}

Using \eqref{I_3bound} and the results for
case 1 \eqref{I_3:eq1} and case 2 (\ref{I_3eq2} and \ref{I_3eq3}),
we conclude that
\begin{equation}\label{I_3final_bound}
I_3\leq F(k,\ell)\frac{c(k,\ell)}{n}O(k\ell \alpha^{4k})
=
c_0'\ell^{4k} 
c_1^{ k^2}
(c_2')^k
(k\ell)^2
((2k\ell + 1)^2 (\ell + 1))^{4k}
\alpha^{2k\ell}
\end{equation}

\subsubsection{Finishing the proof of Theorem~\ref{spec_B_bar}}

We have bounded the three pieces we need to prove the theorem. 
From \eqref{bound:I_1}, \eqref{bound:I_2}, and \eqref{I_3final_bound},
with $n$ sufficiently large, we get
\begin{align*}
\E \left( \| \bar{B}^{(\ell)} \|^{2k}\right)
&\leq 
I_1 + I_2 + I_3  \\
&\leq 
\alpha^{2k\ell}
n
\cdot
C
c_1^{ k^2}
c_5^k
\ell^{4k}
(k \ell)^{4}
\left(k^2\ell^3 \right)^{4k}
\\
&\leq
\alpha^{2k\ell} 
n
\cdot
C'
c_1^{(\log n)^{2/3}}
c_6^{(\log n)^{1/3}}
(\log n)^{4 (\log n)^{1/3}}
(\log n)^{16/3}
\,
((\log n)^{11/3})^{4 (\log n)^{1/3}}
\\
&\leq
\alpha^{2k\ell} 
n
\cdot
C''
c_1^{(\log n)^{2/3}}
c_6^{(\log n)^{1/3}}
(\log n)^{20 (\log n)^{1/3} + 6} 
\\
& := 
\alpha^{2k\ell} n \cdot f(n) ,
\end{align*}
where $c_5, c_6, C, C', C''$ are universal constants.

Take any $\epsilon > 0$.
It can be checked that 
$(\log n)^{20 (\log n)^{1/3} + 6}= o(n^{\epsilon})$, and 
$f(n) = o(n^{\epsilon})$ as well.
Let $g(n) = \exp((\log n)^{3/4})$;
then $g(n) = o(n^\epsilon)$ but
$g(n)^{2k} \gg n^\epsilon$.
We apply Markov's inequality, so that
\begin{align}
\P \left[
\| B^{(\ell)} \| 
\geq 
\alpha^{\ell} g(n)
\right]
&\leq 
\frac{ \E \left( \| \bar{B}^{(\ell)} \|^{2k} \right) }{ \alpha^{2k\ell} g(n)^{2k}}
\leq  
\frac{n f(n)}{g(n)^{2k}}=o(1),
\label{eq:f_over_g}
\end{align}
which is the statement of the theorem.
\end{proof}

\begin{theorem}
\label{spec_R}
Let $1\leq j\leq \ell = \lfloor c \log(n) \rfloor$ 
where $c < \frac{1}{32}$ is a universal constant.
Then
\[
\|R^{\ell,j}\| 
\leq 
\frac{(d-1)^\ell}{n} \exp((\log n)^{3/4})
\]
asymptotically almost surely.
\end{theorem}

\begin{proof}
The proof is analogous to the proof of Theorem \ref{spec_B_bar}. Recall the definition of $R^{\ell,j}$ in Eqn.~\ref{error_term}.
For any integer $k$, we have that
\begin{align}
\E\left(\|R^{\ell,j}\|^{2k}\right)
&\leq 
\E 
\left(
\mathrm{Tr}
\left[
 \left( (R^{\ell,j}) (R^{\ell,j})^* \right)^{k}
\right]
\right)
\nonumber \\
&\leq
\sum_{ \gamma \in \mathcal{D} }
\prod_{s=1}^{2k}
\left|
\E
\prod_{i=1}^{j-1} \bar{B}_{e^s_{2i-1} e^s_{2i+1}} 
S_{e^s_{2j-1} e^s_{2j+1}}
\prod_{i=j+1}^{\ell} B_{e^s_{2i-1} e^s_{2i+1}} 
\right| 
\nonumber \\
&=
\sum_{ \gamma \in \mathcal{D} }
\prod_{s=1}^{2k}
\left|
\E
\prod_{i=1}^{j-1} \bar{M}_{e^s_{2i-1} e^s_{2i}} 
S_{e^s_{2j-1} e^s_{2j+1}}
\prod_{i=j+1}^{\ell} M_{e^s_{2i-1} e^s_{2i}} 
\right| 
\label{trace_expectation:R^l,j}
\end{align}
Now, the sum is over 
the set $\mathcal{D}$ of circuits 
$\gamma= (\gamma_1,\gamma_2,\dots,\gamma_{2k})$ of length $2k\ell$
formed from $2k$ elements of $T^{\ell,j}$, 
$\gamma_s = (e^s_1, e^s_2, \ldots, e^s_{2\ell+1} )$ for $s \in [2k]$,
with the convention $e^{s+1}_1 = e^s_{2\ell+1}$ (see definition ~\ref{def:tangled_paths}).

% We proceed as before.
% Let
% \begin{equation}
%     K_j = \sum_{\gamma \in \mathcal{D}_j}
%     \prod_{s=1}^{2k}
%     \left|
%     \E
%     \prod_{i=1}^{j-1} \bar{M}_{e^s_{2i-1} e^s_{2i}} 
%     S_{e^s_{2j-1} e^s_{2j+1}}
%     \prod_{i=j+1}^{\ell} M_{e^s_{2i-1} e^s_{2i}} 
%     \right| .
% \end{equation}
% Then,
% \begin{equation}
%     \E\left(\|R^{\ell,j}\|^{2k}\right)
%     \leq
%     K_1 + K_2 + K_3.
% \end{equation}
% \todo{Hmm... above probably not necessary}

Proceeding as before, using Lemma \ref{lem:combinatorial_of_reminder} and Corollary \ref{cor:edge_iso} we get that
\begin{align*}
    \E\left(\|R^{\ell,j}\|^{2k}\right)
    &\leq
    \sum_{\nvertex, \nedge, \nvright}
    \sum_{\gamma \in \mathcal{D}_{\nvertex, \nedge}^\nvright }
    \prod_{s=1}^{2k}
    \left|
    \E
    \prod_{i=1}^{j-1} \bar{M}_{e^s_{2i-1} e^s_{2i}} 
    S_{e^s_{2j-1} e^s_{2j+1}}
    \prod_{i=j+1}^{\ell} M_{e^s_{2i-1} e^s_{2i}}
    \right|  \\
    &\leq
    \sum_{\nvertex, \nedge, \nvright}
    D_{\nvertex, \nedge}^\nvright 
    I_{\nvertex, \nedge}^\nvright 
    \prod_{s=1}^{2k}
    \left|
    \E
    \prod_{i=1}^{j-1} \bar{M}_{e^s_{2i-1} e^s_{2i}} 
    S_{e^s_{2j-1} e^s_{2j+1}}
    \prod_{i=j+1}^{\ell} M_{e^s_{2i-1} e^s_{2i}}
    \right| \\
    &\leq
    \sum_{\nvertex, \nedge, \nvright}
    m^\nvright n^{\nvertex-\nvright}
    (2 \ell)^{6k} 
    ( (\nvertex + 1)^2 (\ell + 1) )^{6k(\chi + 1)} \\
    & \qquad \qquad \cdot
    (d_1 (d_1 - 1))^{\nvertex - \nvright}
    (d_2 (d_2 - 1))^\nvright
    (d - 1)^{2 (\chi - 1)} \\
    & \qquad \qquad \cdot
    \prod_{s=1}^{2k}
    \left( \frac{d-1}{|E|} \right)
    C
    2^{b_s}
    \left(\frac{1}{|E|} \right)^{\nedge_s}
    \left( \frac{6\ell}{\sqrt{|E|}} \right)^{\nedge_{1,s}} .
\end{align*}
The last inequality uses Lemma \ref{expectation_bound} on each path $\gamma_s$, 
with  $\nedge_s$, $b_s$, and $\nedge_{1,s}$ 
defined analogous to the quantities $\nedge$, $b$ and $\nedge_1$ in the same lemma.
Note that $S_{ef} \leq \frac{d-1}{|E|}$ for any $e,f$ and 
$\sum_{s=1}^{2k} b_s \leq b_\mathcal{D} \leq 16 (k + \chi)$
by Lemma~\ref{lem:inconsistent_tangled}. 
Setting
$\nedge = \sum_{s=1}^{2k} \nedge_s$, 
taking
$\nedge_1 = \sum_{s=1}^{2k} \nedge_{1,s} \geq 0$,
using $d_1, d_2 \leq d$,
and combining terms, 
we find that
\begin{align*}
    \E\left(\|R^{\ell,j}\|^{2k}\right)
    & \leq
    \sum_{\nvertex, \nedge, \nvright}
    (2 \ell)^{6k} 
    ( (\nvertex + 1)^2 (\ell + 1) )^{6k(\chi + 1)} \\
    & \qquad \qquad \cdot
    (d_1 - 1)^{\nvertex - \nvright}
    (d_2 - 1)^\nvright
    (d - 1)^{2 (\chi - 1)} \\
    & \qquad \qquad \cdot
    \left( \frac{C'}{|E|} \right)^{2k}
    2^{b_\mathcal{D}}
    \left(\frac{1}{|E|} \right)^{\nedge - \nvertex} \\
    &\leq
    c_0 
    \ell^{6k} 
    c_1^k 
    n^{1 - 2k}
    \sum_{\nvertex, \nedge}
    \nvertex
    ( (\nvertex + 1)^2 (\ell + 1) )^{6k(\chi + 1)}
    (d - 1)^{\nvertex}
    \left(
    \frac{c_2}{n}
    \right)^{\chi} \\
    &\leq
    c_0 
    \ell^{6k} 
    c_1^k 
    n^{1 - 2k}
    \sum_{\nvertex = 1}^{2k\ell}
    \sum_{\nedge = \nvertex - 1}^{2k\ell}
    \nvertex
    ( (\nvertex + 1)^2 (\ell + 1) )^{6k(\chi + 1)}
    (d - 1)^{\nvertex}
    \left(
    \frac{c_2}{n}
    \right)^{\chi} \\
    &\leq
    c_0 
    \ell^{6k} 
    c_1^k 
    n^{1 - 2k}
    \sum_{\nvertex = 1}^{2k\ell}
    \nvertex
    ( (\nvertex + 1)^2 (\ell + 1) )^{6k}
    (d - 1)^{\nvertex}
    \sum_{\chi = 0}^{2k\ell}
    \left(
    ( (\nvertex + 1)^2 (\ell + 1) )^{6k}
    \frac{c_2}{n}
    \right)^{\chi}
    ,
    % & \leq
    % (2 k \ell)^3
    % c_0 
    % \ell^{6k} 
    % c_1^k 
    % n^{1-2k}
    % ( (2 k \ell + 1)^2 (\ell + 1) )^{6k}
    % (d - 1)^{2 k \ell} 
    % \\
    % & = 
    % (k \ell)^3
    % c_3 
    % \ell^{6k} 
    % c_1^k 
    % n
    % ( (2 k \ell + 1)^2 (\ell + 1) )^{12 k^2 \ell}
    % \left(
    %     \frac{(d - 1)^\ell}{n}
    % \right)^{2 k}
\end{align*}
for some constants $c_0, c_1, c_2$.
Note that the important $n^{-2k}$ comes from
the $|E|^{-2k}$ deterministic term.
Now, like before we have that $( (\nvertex + 1)^2 (\ell + 1) )^{6k} = o(n)$,
so for sufficiently large $n$ we can bound the sum over $\chi$
by a constant.
The sum over $\nvertex$ is easily bounded as before, leading to
\begin{align*}
    \E\left(\|R^{\ell,j}\|^{2k}\right)
    &\leq 
    (d-1)^{2k\ell}
    n^{1-2k}
    \cdot
    c_0'
    \ell^{6k} 
    c_1^k 
    (2k\ell)^2
    ((2k\ell + 1)^2 (\ell + 1))^{6k}
    \\
    &\leq
    (d-1)^{2k\ell}
    n^{1-2k}
    \cdot
    c_0''
    c_1'^{(\log n)^{1/3}}
    (\log n)^{6(\log n)^{1/3}}
    (\log n)^{8/3}
    ((\log n)^{11/3})^{6(\log n)^{1/3}} \\
    &:= 
    (d-1)^{2k\ell}
    n^{1-2k}
    \cdot 
    f(n).
\end{align*}
To finish the proof, we note that $f(n) = o(n^\epsilon)$ 
for any $\epsilon > 0$.
Take $g(n) = \exp((\log n)^{3/4})$.
Then, applying Markov's inequality,
\begin{equation}
    \P 
    \left[ 
    \| R^{\ell,j} \| 
    \geq \frac{(d - 1)^{\ell} }{n} g(n)
    \right]
    \leq
    \frac{\E \left( \| R^{\ell,j} \|^{2k} \right) n^{2k}}
    { (d-1)^{2k\ell} g(n)^{2k} }
    \leq
    \frac{n f(n)}{g(n)^{2k}} = o(1).
\end{equation}

\subsection{Proof of the main result, Theorem~\ref{thm:gap_B}}
We will again take $\ell = \lfloor c \log n \rfloor$, with $c$ chosen so that the
graph is $\ell$-tangle-free 
with high probability.
By Eqns.~\eqref{eq:upper_bound_lambda2} and \eqref{decomposition},
\[
|\lambda_2|^{\ell}\leq \| \bar{B}^{(\ell)}\|+
\sum_{k=1}^{\ell} \| R^{\ell,j}\|.
\]
Notice that
$
(d-1)^\ell \leq (d-1)^{c \log n} \leq  n^{c \log d}
$,
so take 
$c < \min \left(\frac{1}{32}, \frac{1}{\log d} \right)$.
Then $(d-1)^\ell = O(n^\epsilon)$ for some $0 < \epsilon < 1$,
and recall that 
$\exp((\log n)^{3/4}) = o(n^{\epsilon'})$ for any 
$\epsilon' > 0$.
We apply
Theorems \ref{spec_B_bar} and \ref{spec_R} to get
\begin{eqnarray}
\nonumber
|\lambda_2| 
& \leq & 
\left(
\exp((\log n)^{3/4}) \left( (d_1-1)(d_2-1) \right)^{\ell/4} 
+ 
\ell 
\exp((\log n)^{3/4})
\frac{(d-1)^\ell}{n}
\right)^{1/\ell} \\
\nonumber
& = & \left( (d_1-1)(d_2-1) \right)^{1/4}+\epsilon_n,
\end{eqnarray}
where $\epsilon_n \to 0$ as $n \to \infty$.
\end{proof}

\section{Application: Community detection}
\label{sec:communities}

In many cases, such as online networks, 
we would like to be able to recover specific communities in those graphs.
In the typical setup, a community is a set of vertices that are more densely connected
together than to the rest of the graph.

The model we present here
is inspired by the planted partition or
stochastic blockmodel (SBM, \citet*{holland1983}).
In the SBM, each vertex belongs to a class or community, 
and the probability that two vertices are connected is a function
of the classes of the  vertices.
It is a generalization of the \ER random graph.
The classes or blocks in the SBM make it a good model for
graphs with community structure, 
where nodes preferentially connect to other nodes
depending on their communities
(\citet*{newman2010}).

There are many methods for detecting a community given a graph.
For an overview of the topic, see \citet*{fortunato2010}.
Spectral clustering is a common method which can be applied to any set of data
$\{ \zeta_i \}_{i=1}^n$.
Given a symmetric and non-negative similarity function $S$,
the similarity is computed for every pair of data points, forming a matrix
$A_{ij} = S(\zeta_i, \zeta_j ) = S(\zeta_j, \zeta_i ) \geq 0$.
The spectral clustering technique is to compute the leading eigenvectors of
$A$, or matrices related to it, and use the eigenvectors to cluster the data.
In our case, the matrix in question is just the 
Markov matrix of a graph, defined soon.
We will show that we can guarantee the success of the technique if the degrees
are large enough.

Our graph model is a regular version of the SBM.
We build it on a ``frame,''
which is a small, weighted graph that
defines the community structure present in the larger, random graph.
Each class is represented by a vertex in the frame.
The edge weights in the frame define the number of edges
between classes.
What makes our model differ from the SBM is that
the connections between classes are described by a 
regular random graph rather than an \ER random graph.
However, the graph itself is not necessarily regular.

A number of authors have studied similar models.
Our model is a generalization of a random lift of the frame,
which is said to {\it cover} the random graph
(e.g. \citet*{marcus2013a,angel2015,bordenave2015a}).
This type of random graph was also studied by 
\citet*{newman2014}, who called it an equitable random graph,
since the community structure is equivalent to an 
equitable partition.
This partition induces a number of symmetries across vertices in 
each community which are useful when studying the 
eigenvalues of the graph.
\citet*{barrett2017} studied the effect of these symmetries 
from a group theory standpoint.
The work of \citet*{barucca2017} is closest to ours:
they consider spectral properties of such graphs and their 
implications for spectral community clustering. 
In particular, they show that the spectrum of what we call the ``frame'' 
(in their words, the discrete spectrum, which is deterministic) 
is contained in that of the random graph.
They use the resolvent method
(called the cavity method in the physics community)
to analyze the continuous part of the spectrum in the limit of
large graph size, and argue that community detection is possible
when the deterministic frame eigenvalues all lie outside the bulk.
However, this analysis assumes that there are no stochastic 
eigenvalues outside the bulk, which will only hold with high probability
if the graph is Ramanujan.
Our analysis shows that, if a set of pairwise spectral gaps hold between
all communities, then this will be the case.

\subsection{The frame model}

We define the {\it \kframe} distribution 
$\mathcal{G}(n,H)$
as a distribution of simple graphs on $n$ 
vertices parametrized by the ``frame'' $H$.
The frame $H=(V,E,p,D)$
is a weighted, directed graph.
Here, $V$ is the vertex set,
$E \subseteq \{(i,j): i,j \in V\}$ 
is the directed edge set,
the vertex weights are $p$, and the edge weights are $D$.
Note that we drop the arrows on the edge set in this Section, 
since it will always be directed.
The vertex weight vector
$p \in \R^{|V|}$, 
where
$\sum_{i \in V} p_i = 1$,
sets the relative sizes of the classes.
The edge weights are a matrix of degrees
$D \in \mathbb{N}^{|V| \times |V|}$.
These assign the number of edges between each class in the random graph:
$D_{ij}$
is the number of edges from each vertex in class $i$
to vertices in class $j$.
The degrees must satisfy the
balance condition
\begin{equation}
  \label{eq:detailedbalance}
  p_i D_{ij} = p_j D_{ji}
\end{equation}
for all
$i,j \in V$ where $(i,j)$ or $(j,i)$ are in $E$.
This requires that, for every edge $e \in E$,
its reverse orientation also exists in $H$.
We also require that
$n_i = n p_i \in \mathbb{N}$
for every $i \in V$, 
so that the number of vertices in each type is integer.

Given the frame $H$,
a \kframe $G \sim \mathcal{G}(n,H)$
is a simple graph on $n$ vertices with $n_i$
vertices in class $i$.
It is chosen uniformly among graphs 
with the constraint that each vertex in class $i$ makes
$D_{ij}$ connections among the vertices in class $j$.
In other words, if $i = j$, we sample that block of the adjacency
matrix as the adjacency matrix of a $D_{ii}$-regular random graph
on $n_i$ vertices.
For off-diagonal blocks $i \neq j$,
these are sampled as bipartite, biregular random graphs
$\mathcal{G} (n_i, n_j, D_{ij}, D_{ji})$.

\begin{figure}[t!]
  {\bf \Large A}\qquad {\large Frame}
  \hspace{2.05in}
  {\bf \Large B}\qquad {\large Random regular frame graph}
  \begin{center}
  \hspace{.34in}
  \includegraphics[height=2.6in]{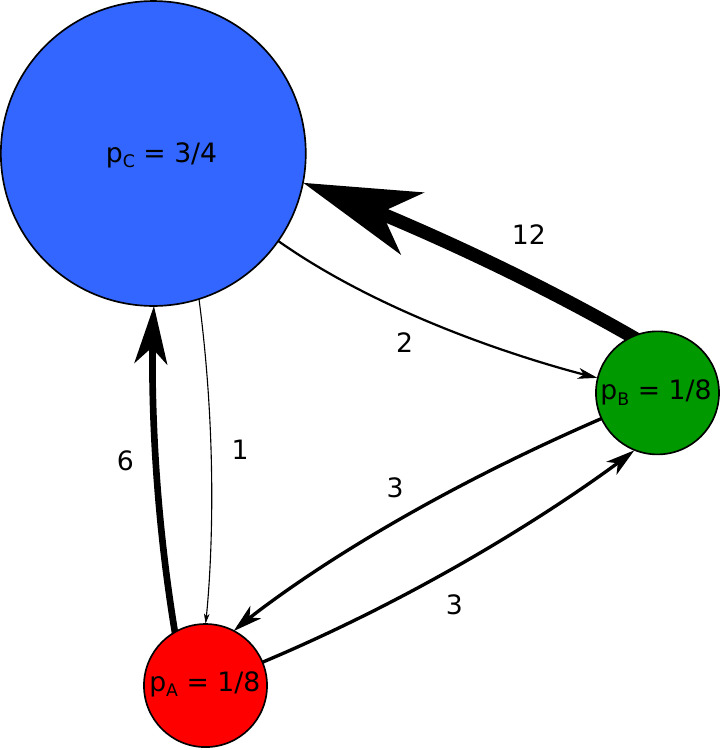}
  \hfill
  \includegraphics[height=2.6in]{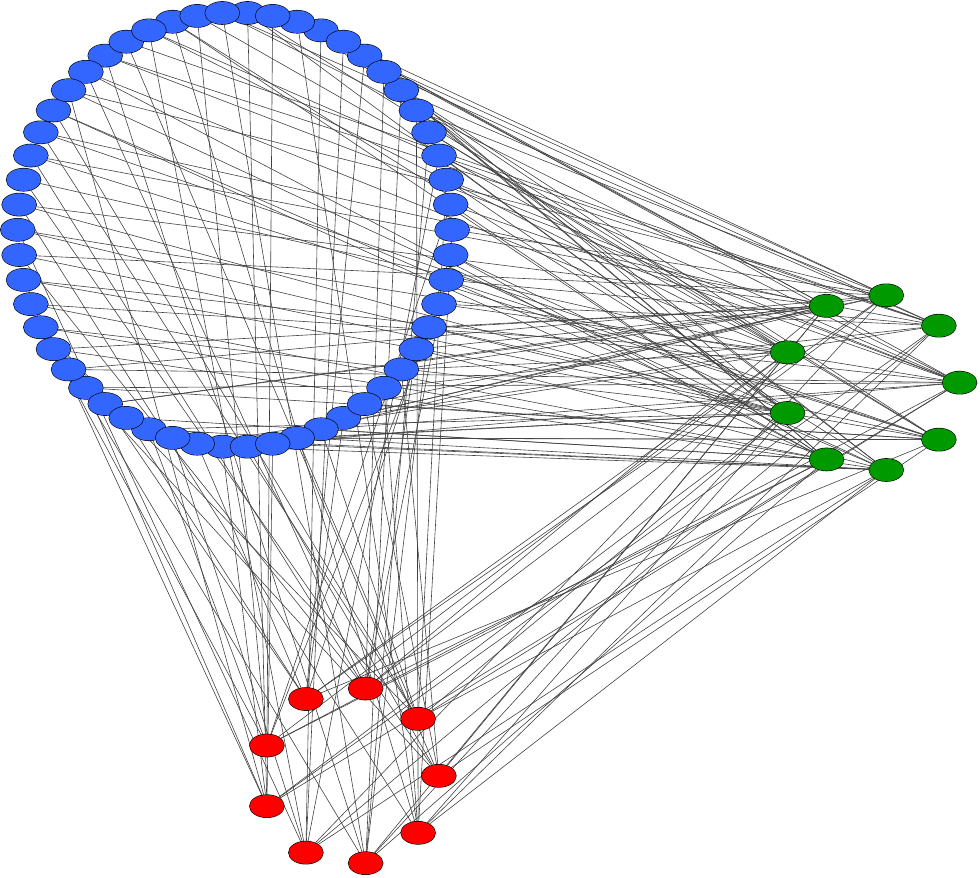}
\end{center}
  \caption{Schematic and realization of a \kframe. 
    {\bf A}, 
    the frame graph.
    The vertices of the frame (red = A, green = B, blue = C)
    are weighted according to their proportions $p$
    in the \kframe.
    The edge weights $D_{ij}$ 
    set the between-class vertex degrees
    in the \kframe.
    This frame will yield a random tripartite graph.
    {\bf B}, realization of the graph on 72 vertices. 
    In this instance, there are 
    $1/8 \times 72=9$ green and red vertices
    and
    $3/4 \times 72=54$ blue vertices.
    Each blue vertex connects to 
    $k_{CA}=1$ red vertex and
    $k_{CB}=2$ green vertices.
    This is actually a multigraph; 
    with so few vertices, the probability that
    the configuration model algorithm yields parallel edges is high.
  }
  \label{fig:frame}
\end{figure}

Sampling from 
$\mathcal{G}(n,H)$
can be performed similar to the configuration model,
where each node is assigned as many half-edges as its degree,
and these are wired together with a random matching 
(\citet*{newman2010}).
The detailed balance condition 
Eqn.~\eqref{eq:detailedbalance} 
ensures that this matching is possible.
Practically, we often have to generate many candidate matchings
before the resulting graph is simple, 
but the probability of a simple graph is bounded away from zero for
fixed $D$.

An example of a \kframe is the bipartite, biregular random graph.
The family $\mathcal{G}(n, m, d_1, d_2)$ 
is a \kframe $\mathcal{G}(n + m, H)$,
where the frame $H$ is the directed path on two vertices:
$V= \{ 1, 2 \}$ and $E = \{ (1,2), (2,1) \}$.
The weights are taken as 
$p_1 = n/(n+m)$, $p_2 = m/(n+m)$, $D_{12} = d_1$, and $D_{21} = d_2$.

Another example \kframe is shown in
Figure~\ref{fig:frame}.
In this case, the frame $H$ has 
$V= \{ A, B, C \}$ and 
$E = \{ (A,B), (A,C), (B,A), (B,C), (C,A), (C,B) \}$
with weights $p$ and $D$ as shown in Figure~\ref{fig:frame}A.
We see that this generates a random tripartite graph with 
regular degrees between vertices in different independent sets,
shown in Figure~\ref{fig:frame}B.

\subsection{Markov and related matrices of frame graphs}

Now, we define a number of matrices associated with the frame 
and the sample of the \kframe.

Let $G$ be a simple graph.
Define $D_G = \mathrm{diag}(d_G)$,
the diagonal matrix of degrees in $G$.
The Markov matrix $P=P(G)$ is defined as
\[
P = D_G^{-1} A,
\]
where $A=A(G)$ is the adjacency matrix.
The Markov matrix is the row-normalized adjacency matrix, and it contains
the transition probabilities of a random walker on the graph $G$.
Let $L = I - \mathcal{L} =  D_G^{-1/2} A \, D_G^{-1/2}$
be a matrix simply related to the normalized Laplacian.
We call this the {\it symmetrized Markov matrix}.
Then $P$ and $L$ have the same eigenvalues,
but $L$ is symmetric, since $L_{ij} = \frac{A_{ij}}{ \sqrt{ d_i d_j } }$. 

Suppose $G \sim \mathcal{G}(n, H)$, where the frame $H = (V,E,p, D)$.
Another matrix that will be useful is what we call the 
{\it Markov matrix of the frame}
$R$, where
$R_{ij} = \frac{ D_{ij} }{\sum_j D_{ij}}$.
% \[
% R = \mathrm{diag}(d)^{-1} D,
% \]
% where now $d_i = \sum_j d_{ij}$, 
Thus, $R$ is a row-normalized $D$, 
in the same way that the Markov matrix $P$ is the row-normalized adjacency matrix $A$.
Furthermore, $R$ is invariant under any uniform scaling of the degrees.
% $D \to \kappa D$.
Because of this equitable partition property of {\kframe}s, 
eigenvectors of the frame matrices $D=D(H)$ or $R=R(H)$ lift to 
eigenvectors of $A = A(G)$ or $P=P(G)$, respectively.
Suppose $D x = \lambda x$,
then it is a straightforward exercise to check that $A \tilde{x} = \lambda \tilde{x}$ 
for the piecewise constant vector
\[
\tilde{x} = 
\left[
\begin{array}{c}
\1_{n_1} x_1 \\
\1_{n_2} x_2 \\
\vdots
\end{array}
\right] .
\]
Using the same procedure, we can
lift any eigenpair of $R$ to an eigenpair of $P$ with the same eigenvalue.

\subsubsection{Bounds on the eigenvalues of frame graphs in terms of blocks}

The following result is due to \citet*{wan2015}:
\begin{proposition}
\label{thm:spurious_eig_bound}
Let $G$ be a random regular frame graph $G(n,H)$,
$P$ its Markov matrix, 
and $L$ the Laplacian
with vertices ordered by class in both cases.
% Define the vector $s \in \R^n$ by $s_i = \sqrt{d_i}$, 
% which is the Frobenius eigenvector of $L$ with eigenvalue 1.
Let $R$ be the Markov matrix of the frame $H=(V,E,p,D)$, 
with $|V(H)| = K$ classes.
% and $\lambda_1 \ge \ldots \ge \lambda_K$ the eigenvalues of $D$.
Define the matrices
$L^{(kl)}$
as the $(k,l)$ 
block of $L$ with respect to the clustering of vertices by class.
For $l \neq k$, let
\[
M^{(kl)} =
\left(
\begin{array}{cc} 
0 & L^{(kl)} \\
L^{(kl)} & 0 
\end{array}
\right) 
=
\left(
\begin{array}{cc} 
0 & L^{(kl)} \\
L^{(lk)*} & 0 
\end{array}
\right)
.
\]
For $l=k$, let $M^{(kk)} = L^{(kk)}$.
Assume that all eigenvalues of $D$ are nonzero and pick a constant $C$ such that
\[
	\frac{| \lambda_2^{(kl)} |}{\lambda_1^{(kl)}} \le C < 1
\]
for every $k,l = 1, \ldots K$, 
where $\lambda_1^{(kl)}$ and $\lambda_2^{(kl)}$ are the
leading and second eigenvalues of $M^{(kl)}$.
Under these conditions, the eigenvalues of $P$ which are not eigenvalues of $R$
%(spurious eigenvalues)
are bounded by 
\[
C \max_{k=1,\ldots, K} \left( R_{kk} + \sum_{l \neq k} \sqrt{R_{kl} R_{lk}} \right)
\leq \frac{C}{2}  \left(1 + \max_{k=1,\ldots, K} \sum_{l=1}^K R_{lk} \right) .
\]
\end{proposition}

The spectrum of the Markov matrix $\sigma(P)$ enjoys a simple connection to $\sigma(A)$ 
when $A$ is the adjacency matrix of a graph drawn from $G(n,m,d_1,d_2)$.
In this case, $L = \frac{ A }{ \sqrt{d_1 d_2} }$, 
so the eigenvalues of $P$ are just the scaled eigenvalues of $A$.
This and the spectral gap for bipartite, biregular random graphs, 
Theorem~\ref{cor:gap_A},
lead to the following remark:

\begin{remark}
For a random regular frame graph, 
$M^{(kl)}$ corresponds to the symmetrized Markov matrix $L$ 
of a bipartite biregular graph 
$G(n_k, n_l, D_{kl}, D_{lk})$.
Thus,
\[
\frac{| \lambda_2^{(kl)} |}{\lambda_1^{(kl)}} 
\leq
\frac{\sqrt{D_{kl} - 1} +\sqrt{D_{lk} - 1}}{\sqrt{D_{kl} D_{lk}}} + \epsilon .
\]
\end{remark}

Suppose we are given a frame that fits the conditions of 
Proposition~\ref{thm:spurious_eig_bound};
namely, $D$ cannot have any zero eigenvalues.
Then we can uniformly grow the degrees, which leaves $R$ invariant,
but allows us to reach an arbitrarily small $C$.
This ensures that the leading $K$ eigenvalues of $P$ are equal to the 
eigenvalues of $R$.
Note that this actually means that the entire \kframe satifsfies 
a weak Ramanujan property.
We now show that this guarantees spectral clustering.

\subsection{Spectral clustering}

Spectral clustering is a popular method of community detection.
Because some eigenvectors of $P$, the Markov matrix of a \kframe,
are piecewise constant on classes, 
we can use them to recover the communities so long as those eigenvectors can be identified.
Suppose there are $K$ total classes in our \kframe.
Then, given the eigenvectors $x^1, x^2, \ldots, x^K$, which are piecewise constant across classes,
we can cluster vertices by class.
For each vertex $v \in V(G)$, associate the vector
$y^v \in \R^K$
where $y^v_j = x^j_v$.
Then if $y^v = y^u$ for $u, v \in V(G)$, 
vertices $u$ and $v$ belong to the same class\footnote{
In the SBM case, the eigenvectors are not piecewise constant, 
but they are aligned with the eigenvectors of $R$
and thus highly correlated across vertices in the same class.
A more flexible clustering method such as $K$-means must be applied to the vectors $y$
in that case.
}.
It is simple to recover these piecewise constant vectors
$x^1, x^2, \ldots, x^K$ when they are the leading eigenvectors.
These facts lead to the following theorem:

\begin{theorem}[Spectral clustering guarantee in frame graphs]
\label{thm:spectral_clustering}
Let $G$ be a random regular frame graph $G(n,H)$
and $P$ its Markov matrix.
Let $R$ be the Markov matrix of the frame $H=(V,E,p,D)$, 
with $|V(H)| = K$ classes
and $\lambda_1 \ge \ldots \ge \lambda_K$ the eigenvalues of $R$
and $|\lambda_K| > 0$. 
Then we can scale the degrees by some $\kappa \in \mathbb{N}$,
$D \to \kappa D$,
so that the vertex classes are recoverable by spectral clustering of 
the leading $K$ eigenvectors of $P$.
\end{theorem}

\begin{remark}
The conditions of Theorem~\ref{thm:spectral_clustering},
while very general, are also weaker than may be expected using more sophisticated methods
tailored to the specific frame model. 
We illustrate this with the following example.
\end{remark}

\subsubsection{Example: The regular stochastic block model}

\citet*{brito2016} and \citet*{barucca2017} studied a regular stochastic block model,
which can be seen as a special case of our frame model.
Let the frame $H$ be the complete directed graph on two vertices, 
including self loops, where
\[
D = \left(
\begin{array}{cc}
d_1 & d_2\\
d_2 & d_1
\end{array}
\right)
\]
and $p = (1/2, 1/2)$.
Define the regular stochastic block model as $\mathcal{G} (2n, H)$.
This is a graph with two classes of equal size, 
representing two communities of vertices,
with within-class degree $d_1$ and between-class degree $d_2$. 
We assume $d_1 > d_2$, since communities are more strongly connected within.
\citet*{brito2016} proved the following theorem:
\begin{theorem}
\label{thm:RSBM_condition}
If $(d_1-d_2)^2 > 4 (d_1 + d_2 -1)$, then there is an efficient algorithm
for strong recovery, i.e.\ recovery of the exact communities with high probability
as $n \to \infty$.
\end{theorem}

Theorem~\ref{thm:RSBM_condition} gives a sharp bound on the degrees for recovery,
which we can compare to our spectral clustering results.
The eigenvalues of $D$ are $d_1 + d_2$ and $d_1 - d_2$,
and the Markov matrix of the frame $R$ has eigenvalues $1$ and $(d_1 - d_2) / (d_1 + d_2)$.
The diagonal blocks $L^{(11)}$ and $L^{(22)}$ each correspond to the
Laplacian matrix of a
$d_1$-regular random graph on $n$ vertices, 
whereas the off-diagonal block term $M^{(12)}$ corresponds to the Laplacian
of a $d_2$-regular bipartite graph on $2n$ vertices.
Using our results and the previously known results for regular random graphs
(\citet*{friedman2003a,friedman2004,bordenave2015a}),
we can pick some 
$C > 2 \sqrt{d_2 - 1}/d_2$ 
since $d_1 > d_2$ and we will eventually take the degrees to be large.
Using Proposition~\ref{thm:spurious_eig_bound},
we find that the spurious eigenvalues of $P$ 
come after the leading 2 eigenvalues if
\[
\frac{2 \sqrt{d_2 - 1}}{d_2} < \frac{d_1 - d_2}{d_1 + d_2},
\]
to leading order in the degrees.
Rearranging, we obtain the condition
\[
(d_1 - d_2)^2 > 4 (d_2 - 1) \left( \frac{d_1 + d_2}{d_2} \right)^2 .
\]
Assuming $d_2/d_1 = \beta < 1$ fixed, and taking the limit $d_1, d_2 \to \infty$,
we find that the result of \citet*{brito2016} becomes
\[
d_1 > 4 \frac{1+\beta}{(1-\beta)^2} + o(1),
\]
whereas our result becomes
\[
d_1 > \frac{4}{\beta} \left( \frac{1+\beta}{1-\beta} \right)^2 + o(1),
\]
illustrating that the spectral threshold is a factor of $(1+\beta)/\beta$ weaker.

\section{Application: Low density parity check or expander codes}
\label{sec:codes}

Another useful application of random graphs is as expanders,
loosely defined as graphs where the neighborhood of a small set of nodes is large.
Expander codes, also called low density parity check (LDPC) codes,
were first introduced by Gallager in his PhD thesis
(\citet*{gallager1962}).
These are a family of linear error correcting codes whose parity-check matrix 
is encoded in an expander graph.
A linear code is a set $\mathcal{C} \subset \Sigma^L$,
where a length $L$ codeword
$x \in \mathcal{C}$ if and only if
$H x = 0$.
The alphabet $\Sigma$ is typically a finite field
and $H \in \Sigma^{P \times L}$ is the parity check matrix.
In the simplest case, $\Sigma = \mathbb{F}_2$ and each row of $H$ 
can be interpreted as a parity constraint on codewords.
The performance of such codes depends on how good an expander that graph is,
which in turn can be shown to depend on the separation of eigenvalues.
For a good introduction and overview of the subject, 
see \citet*{richardson2008}.

Following \citet*{tanner1981}, 
we construct a code $\mathcal{C}$ 
from a $(d_1,d_2)$-regular bipartite graph $G$
on $n + m$ vertices
and two smaller linear codes 
$\mathcal{C}_1$ and $\mathcal{C}_2$ of length
$d_1$ and $d_2$, respectively.
We write 
$\mathcal{C}_1 = [d_1, k_1, \delta_1]$ 
and
$\mathcal{C}_2 = [d_2, k_2, \delta_2]$
with the usual convention
of length, dimension, and minimum distance.
We assume the codes are all binary, using the finite field $\mathbb{F}_2$ 
the codeword is $x \in \mathcal{C} \subset \mathbb{F}_2^{|E|}$
where  $|E| = n d_1 = m d_2$.
That is, we associate a bit to each edge in the graph bipartite graph $G$.
Let $( e_i(v) )_{i=1}^{d_v}$ 
represent the set of edges incident to a vertex $v$
in some arbitrary, fixed order.
Then the vector $x \in \mathcal{C}$ if and only if the vectors
$( x_{e_1(u)}, x_{e_2(u)}, \ldots, x_{e_{d_1}(u)} )^T \in \mathcal{C}_1$
for all $u \in V_1$ and
$( x_{e_1(v)}, x_{e_2(v)}, \ldots, x_{e_{d_2}(v)} )^T \in \mathcal{C}_2$
for all $v \in V_2$.
The final code $\mathcal{C}$ is also linear.
With this construction, the code $\mathcal{C}$ 
has rate at least $k_1 / d_1 + k_2/d_2 - 1$ (\citet*{tanner1981}).

Furthermore, \citet*{janwa2003} proved the following bound on
the minimum distance of the resulting code:
\begin{theorem}
\label{thm:JanwaLal}
Suppose $\delta_1 \geq \delta_2 > \eta/2$, where $\eta$ is the second largest eigenvalue of the adjacency matrix of $G$. Then the code $\mathcal{C}$ has minimum distance
\[
\delta \geq \frac{n}{d_2} \left( \delta_1 \delta_2 - \frac{\eta}{2} (\delta_1 + \delta_2 ) \right).
\]

\end{theorem}

\begin{corollary}
\label{cor:dist_bound}
Suppose the code $\mathcal{C}$ is constructed from a biregular, bipartite random graph 
$G \sim \mathcal{G}(n,m, d_1, d_2)$ and the conditions of Theorem~\ref{thm:JanwaLal} hold.
Then the minimum distance of $\mathcal{C}$ satisfies
\[
\delta \geq \frac{n}{d_2} \left( \delta_1 \delta_2 - 
\frac{\sqrt{d_1 - 1} + \sqrt{d_2 - 1}}{2} (\delta_1 + \delta_2 ) - \epsilon_n \right) .
\]
\end{corollary}

We see that these Tanner codes will have maximal distance for smallest $\eta$,
and used our main result, Theorem~\ref{cor:gap_A}, to obtain the explicit
bound in Corollary~\ref{cor:dist_bound}.
By growing the graph, the above shows a way to construct arbitrarily large codes whose
minimum distance remains proportional to the code size $n d_1$.
That is, the relative distance $\delta / (n d_1)$ 
is bounded away from zero as $n \to \infty$.
However, the above bound will only be useful if it yields a positive result, 
which depends on the codes $\mathcal{C}_1$ and $\mathcal{C}_2$ as well as the degrees.

\begin{remark}
In general, the performance guarantees on LDPC codes that are obtainable 
from graph eigenvalues are weaker than those that come from other methods.
Although our method does guarantee high distance for some high degree codes, 
analysis of specific decoding algorithms or a probabilistic expander analyses
yield better bounds that work for lower degrees (\citet*{richardson2008}).
\end{remark}

\subsection{Example: An unbalanced code based on a $(14,9)$-regular bipartite graph}

We illustrate the applicability of our distance bound with an example.
Let $\mathcal{C}_1 = [14, 8, 7]$ and $\mathcal{C}_2 = [9, 4, 6]$.
These can be achieved by using a Reed-Salomon code
on the common field $\mathbb{F}_{q}$ for any $q > 14$ (\citet*{richardson2008}).
We take $q = 2^4 = 16$ for inputs that are actually binary, and
this means each edge in the graph actually contains 4 bits of information.
Employing Corollary~\ref{cor:dist_bound}, 
the Tanner code $\mathcal{C}$
will have relative minimum distance 
$\delta / (n d_1) \geq 0.0014$ and 
rate at least $0.016$.
Taking $n = 216$ and $m = 336$ gives the code a minimum distance of 
at least 4.

\section{Application: Matrix completion}
\label{sec:matrix_completion}

Assume we have some matrix $Y \in \R^{n \times m}$ 
which has low ``complexity.'' 
Perhaps it is low-rank or simple by some other measure.
If we observe $Y_{ij}$ for a limited set of entries
$(i,j) \in E \subset [n] \times [m]$,
then {\it matrix completion} is any method which
constructs a matrix $\hat{Y}$
so that $\| \hat{Y} - Y\|$ is small, or even zero.
Matrix completion has attracted significant attention in recent years as a tractable
algorithm for making recommendations to users of online systems based on the
tastes of other users (a.k.a.\ the Netflix problem).
We can think of it as the matrix version of 
compressed sensing (\citet*{candes2010, candes2010a}).

Recently, a number of authors have studied the performance of 
matrix completion algorithms where the index set $E$ is the edge
set of a regular random graph 
(\citet*{heiman2014, bhojanapalli2014,gamarnik2017}).
\citet*{heiman2014} describe a 
{\it deterministic} method of matrix completion,
where they can give performance guarantees for a fixed observation set $E$ 
over many input matrices $Y$.
The error of their reconstruction depends on the spectral gap of the graph.
We expand upon the result of \citet*{heiman2014}, 
extending it to rectangular matrix and improving their bounds in the process.

\subsection{Matrix norms as measures of complexity and their relationships}

We will employ a number of different matrix and vector norms in this Section.
These are all related by the properties of the underlying Banach spaces.
The complexity of $Y$ 
is measured using a factorization norm (also called the max-norm):
\[
	\gamma_2 (Y) = \min_{UV^* = Y} \|U\|_{\ell_2 \to \ell_\infty^n} \|V\|_{\ell_2 \to \ell_\infty^m}.
\]
The minimum is taken over all possible factorizations of $Y = UV^*$, 
and the norm 
% $\| \cdot \|_{\ell_2 \to \ell_\infty^n}$
$\| X \|_{\ell_2 \to \ell_\infty^n} = \max_i \sqrt{\sum_j X^2_{ij}}$
returns the largest $\ell_2$ norm of a row.
So, equivalently,
\[
	\gamma_2 (Y) = \min_{U V^* = Y} \max_{i,j} \|u_i \|_2 \, \|v_j \|_2 ,
\]
where $u_i$ and $v_i$ are the rows of $U$ and $V$.
See \citet*{linial2007} for a number of results about the norm $\gamma_2$.
In particular, note that
\begin{align}
\gamma_2(A \circ B) \leq \gamma_2(A) \gamma_2(B) \label{eq:submultiplicative} \\
  \label{eq:trace_norm_equiv}
  \frac{1}{\sqrt{n m}} \| Y \|_{\rm Tr} \leq \gamma_2(Y) \\ % \leq \| Y \|_{\rm Tr} \\
  \label{eq:rank_equiv}
  \gamma_2(Y) \leq \sqrt{\mathrm{rank} (Y) } \| Y \|_{\infty}.
\end{align}
Property \eqref{eq:submultiplicative} says that $\gamma_2$
is sub-multiplicative under the Hadamard product \citep*{lee2008,heiman2014} and will be used in our proof.
Properties \eqref{eq:trace_norm_equiv} and \eqref{eq:rank_equiv}
relate $\gamma_2$ to two common complexity measures of matrices,
the trace norm (sum of singular values, i.e.\ the $\ell^m_2 \to \ell^n_2$ nuclear norm) 
and rank.
Note also the well-known fact that
\[
\| Y \|_{\rm Tr} = \min_{UV^* = Y} \| U \|_F \|V \|_F ,
\]
where 
% $\| X \|_F = \sqrt{\sum_{i=1}^n \sum_{j=1}^m X_{ij}^2}$ 
$\| X \|_F = \sqrt{\sum_{ij} X_{ij}^2}$ 
is the Frobenius norm.
We see that the trace norm constrains factors $U$ and $V$ to be small on average
via $\| \cdot \|_F$,
whereas the norm $\gamma_2$ is similar but constrains factors uniformly
via $\| \cdot \|_{\ell_2 \to \ell_\infty^n}$.
However, we should note that computing $\gamma_2(Y)$ is more costly than the
trace norm, which can be performed with just the singular value decomposition,
although still possible in polynomial time with convex programming \citep*{heiman2014}.

% There is one last norm we must introduce.
% \todo{talk about ``nuclear norm'' $\nu$}

\subsection{Matrix completion generalization bounds}

The method of matrix completion that we study
is to return the matrix $X$ which is the solution to:
\begin{equation}
\label{matrix_completion}
\begin{aligned}
& \underset{X}{\text{minimize}}
& & \gamma_2(X) \\
& \text{subject to}
& & X_{ij} = Y_{ij}, \; (i,j) \in E .
\end{aligned}
\end{equation}
The $\gamma_2$ norm was first proposed by
\citet*{srebro2005,srebro2005a} 
as a robust complexity measure,
and it was shown to be an effective and practical regularization
on real datasets \citep*{lee2010a,recht2013}.

\citet*{heiman2014} analyze the performance
of the convex program \eqref{matrix_completion} 
for a square matrix $Y$
using an expander argument,
assuming that $E$ is the edge set of a $d$-regular graph 
with second eigenvalue $\eta$.
They obtain the following theorem:
\begin{theorem}[\citet*{heiman2014}]
\label{thm:square_MSE}
Let $E$ be the set of edges of a $d$-regular graph with second eigenvalue bound $\eta$.
For every $Y \in \R^{n \times n}$, if $\hat{Y}$ is the output of the optimization problem 
\eqref{matrix_completion}, then 
\[
	\frac{1}{n^2} \| \hat{Y} - Y \|_F^2 \leq c \gamma_2(Y)^2 \frac{\eta}{d},
\]
where $c = 8 K_G \leq 14.3$ is a universal constant and $\| \cdot \|_F$ is the Frobenius norm.
\end{theorem}

Considering sampling following the biadjacency matrix of a bipartite graph,
we find a similar result which also applies to rectangular matrices. 
If $n=m$ and $d_1 = d_2 = d$, our bound is equivalent to that of Theorem~\ref{thm:square_MSE},
but with constants improved by a factor of two due to stronger mixing in bipartite graphs.
Intuitively, using a biadjacency matrix is a ``more random'' way of sampling 
than using an adjacency matrix, since it is not symmetric.

\begin{theorem}
\label{thm:rect_MSE}
Let $E$ be the set of edges of a $(d_1, d_2)$-regular graph
with second eigenvalue bound $\eta$.
For every $Y \in \R^{n \times m}$, if $\hat{Y}$ is the output of the optimization problem 
\eqref{matrix_completion}, then 
\[
	\frac{1}{n m} \| \hat{Y} - Y \|_F^2 \leq 
      c \gamma_2(Y)^2 \frac{\eta}{\sqrt{d_1 d_2}} ,
\]
where $c = 4 K_G \leq 7.13$.
\end{theorem}
\begin{proof}
We start by considering a rank-1 sign matrix $S=u v^*$, where $u,v \in \{-1,1 \}^{n \times m}$.
Let $S' = \frac{1}{2} (S + J)$, where $J$ is the all-ones matrix, 
so that $S'$ has the entries of -1 in $S$ replaced by zeros.
Then
$S' = 1_A 1_B^* + 1_{A^c} 1_{B^c}^*$ for subsets 
$A \subset V_1 = [n]$ and $B \subset V_2 = [m]$,
where $A = \{ i: u_i = 1 \} $ and $B = \{ j : v_j = 1 \}$.
Consider the expression
\begin{align*}
\left| \frac{1}{n m} \sum_{i,j} s_{ij} - 
\frac{1}{|E|} \sum_{(i,j) \in E} s_{ij} \right|
&= 
\left| \frac{1}{n m} \sum_{i,j} (2 s'_{ij} - 1) - 
\frac{1}{|E|} \sum_{(i,j) \in E} (2 s'_{ij} - 1) \right| \\
&= 
2 \left| \frac{1}{n m} \sum_{i,j} s'_{ij} - 
\frac{1}{|E|} \sum_{(i,j) \in E} s'_{ij} \right| \\
&= 2 \left| \frac{|A||B| + |A^c| |B^c|}{n m}  - 
\frac{E(A,B) + E(A^c, B^c)}{|E|} \right| \\
&\leq 2 \left| \frac{|A||B|}{nm} - \frac{E(A,B)}{|E|} \right|
+ 
2 \left| \frac{|A^c| |B^c|}{n m}  - 
\frac{E(A^c, B^c)}{|E|} \right|.
\end{align*}
% Using a version of the expander mixing lemma, \citep*{janwa},
% \[
% \left| \frac{E(A,B)}{|E|} - \frac{|A||B|}{n m} \right|
% \leq \frac{\eta}{2 |E|} 
%    \left( |A| + |B| - \frac{|A|^2}{n} - \frac{|B|^2}{m} \right),
% \]
% \todo{does the last inequality really hold??}
% we obtain that
% \begin{align*}
% \left| \frac{1}{n m} \sum_{i,j} s_{ij} - 
% \frac{1}{|E|} \sum_{(i,j) \in E} s_{ij} \right|
% & \leq \frac{\eta}{|E|} 
% \left(  
% |A| + |B| - \frac{|A|^2}{n} - \frac{|B|^2}{m} 
%  +
%  |A^c| + |B^c| - \frac{|A^c|^2}{n} - \frac{|B^c|^2}{m}
% \right)\\
% & \leq \frac{\eta}{|E|} 
% \left( 
% n + m - \frac{|A|^2 + |A^c|^2}{n} - \frac{|B|^2 + |B^c|^2}{m}
% \right) \\
% &\leq \frac{\eta}{|E|} 
% \left( 
% n + m - n (x^2 + (1-x)^2) - m (y^2 + (1-y)^2)
% \right) \\
% &\leq 
% \frac{\eta}{2 |E|} (n + m) = 
% \frac{\eta}{2} \left( \frac{d_1 + d_2}{d_1 d_2} \right), 
% \end{align*}
% since $\frac{1}{2} \leq x^2 + (1-x)^2 \leq 1$ for $x \in [0,1]$.
The following is a bipartite version of the expander mixing lemma 
\citep*{dewinter2012}:
\[
\left| \frac{E(A,B)}{|E|} - \frac{|A||B|}{n m} \right|
\leq 
\frac{\eta}{\sqrt{d_1 d_2}} 
\sqrt{\frac{|A||B|}{nm} 
\left(1- \frac{|A|}{n} \right) \left( 1- \frac{|B|}{m} \right) } \\
= 
\frac{\eta}{\sqrt{d_1 d_2}} 
\sqrt{\frac{|A||B||A^c||B^c|}{(nm)^2}} .
\]
We find that
\begin{align*}
\left| \frac{1}{n m} \sum_{i,j} s_{ij} - 
\frac{1}{|E|} \sum_{(i,j) \in E} s_{ij} \right|
& \leq 
\frac{4 \eta}{\sqrt{d_1 d_2}}
\sqrt{\frac{|A||B||A^c||B^c|}{(nm)^2}}
\\
& = \frac{4 \eta}{\sqrt{d_1 d_2}}
\sqrt{ x y (1-x) (1-y) }\\
& \leq \frac{\eta}{\sqrt{d_1 d_2}},
\end{align*}
since $x y (1-x) (1-y)$ attains a maximal value of $2^{-4}$ 
for $0 \leq x,y \leq 1$.

The rest of the proof develops identical to the results of \citet*{heiman2014},
which we include for completeness.
We apply the result for rank-1 sign matrices to any matrix $R$.
Let $R = \sum_i \alpha_i S^i$, where $S^i$ is a rank-1 sign matrix and $\alpha_i \in \R$.
For a general matrix $R$, this might require many rank-1 sign matrices.
Define the sign nuclear norm $\nu (R) = \sum_i |\alpha_i|$. 
Then,
\[
\left| \frac{1}{n m} \sum_{i,j} r_{ij} - 
\frac{1}{|E|} \sum_{(i,j) \in E} r_{ij} \right|
\leq
\nu(R) \frac{\eta}{\sqrt{d_1 d_2}} .
\]
It is a consequence of Grothendieck's inequality,
a well-known theorem in functional analysis,
that there exists a universal constant 
$1.5 \leq K_G \leq 1.8$ 
so that
$\gamma_2 (X) \leq \nu(X) \leq K_G \gamma_2(X)$
for any real matrix $X$; see \citet*{heiman2014}.

Now, let the matrix of residuals 
$R = (\hat{Y}-Y) \circ (\hat{Y} - Y)$, where $\circ$ is the 
Hadamard entry-wise product of two matrices,
so that $R_{ij} = (\hat{Y}_{ij} - Y_{ij})^2$.
Since 
\[
\frac{1}{|E|} \sum_{(i,j) \in E} r_{ij} = 0 ,
\]
we conclude that 
\[
\frac{1}{n m} \sum_{i,j} r_{ij} 
\leq
\nu(R) \frac{\eta}{\sqrt{d_1 d_2}} 
\leq 
K_G \gamma_2(R) \frac{\eta}{\sqrt{d_1 d_2}}.
\]
Furthermore, $\gamma_2 (R) \leq \gamma_2 (\hat{Y}-Y)^2 \leq ( \gamma_2(\hat{Y}) + \gamma_2(Y) ) ^2 $
by \eqref{eq:submultiplicative} and the triangle inequality.
Since $\hat{Y}$ is the output of the algorithm and $Y$ is a feasible solution,
$\gamma_2 (\hat{Y}) \leq \gamma_2 (Y)$.
Thus, $\gamma_2(R) \leq 4 \gamma_2(Y)^2$ and the proof is finished.
\end{proof}

% \begin{remark}
% If we minimize the trace norm of the solution, 
% which is a more practical method than working with $\gamma_2$,
% the same bounds hold in terms of $\| Y \|_{\rm Tr}$.
% This is because $\gamma_2 (Y) \leq \| Y \|_{\rm Tr}$. 
% We only need to modify the final part of the proof.
% \end{remark}

\subsection{Noisy matrix completion bounds}

Furthermore, our analysis easily extends to the case where
the matrix we observe is corrupted with noise. 
As mentioned in the above remark, similar results will hold for the trace norm.
In the noisy case, we solve the problem
\begin{equation}
\label{matrix_completion_2}
\begin{aligned}
& \underset{X}{\text{minimize}}
& & \gamma_2(X) \\
& \text{subject to}
& & \frac{1}{|E|} \sum_{(i,j) \in E} ( X_{ij} - Z_{ij})^2 \leq \delta^2
\end{aligned}
\end{equation}
and obtain the following theorem:
\begin{theorem}
Suppose we observe $Z_{ij} = Y_{ij} + \epsilon_{ij}$ with bounded error
\[
 \frac{1}{|E|} \sum_{(i,j) \in E} \epsilon_{ij}^2 \leq \delta^2 .
\]
Then solving the optimization problem \eqref{matrix_completion_2}
will yield a bound of
\[
\frac{1}{n m} \| \hat{Y} - Y \|_F^2 \leq 
      c \gamma_2(Y)^2 \frac{\eta}{\sqrt{d_1 d_2}} + 4 \delta^2,
\]
where $c = 4 K_G \leq 7.13$.
\end{theorem}
\begin{proof}
Denote $\hat{Y}$ the solution to 
P\ref{matrix_completion_2}.
It will be useful to introduce the sampling operator
$\mathcal{P}_E : \R^{n \times m} \to \R^{n \times m}$,
where
$(\mathcal{P}_E (X) )_{ij} = X_{ij}$
if $(i,j) \in E$ and 0 otherwise.
Again let $R = (\hat{Y} - Y) \circ (\hat{Y} - Y)$
be the matrix of squared errors,
then 
\begin{align*}
\left| 
\frac{1}{nm} \| \hat{Y} - Y \|_F^2 - 
\frac{1}{|E|} \| \mathcal{P}_E (\hat{Y} - Y) \|_F^2 
\right| 
&=
\left| \frac{1}{n m} \sum_{i,j} (\hat{Y}_{ij} - Y_{ij})^2 - 
\frac{1}{|E|} \sum_{(i,j) \in E} (\hat{Y}_{ij} - Y_{ij})^2 \right| \\
& \leq K_G \gamma_2(R) \frac{\eta}{\sqrt{d_1 d_2}}.
\end{align*}
However, since $Y$ is a feasible solution to
P\ref{matrix_completion_2}, we have
\[
\gamma_2 (\hat{Y}) \leq \gamma_2 (Y).
\]
Applying \eqref{eq:submultiplicative} and the triangle inequality, 
\[
\gamma_2 (R) \leq \left( \gamma_2(\hat{Y} - Y) \right)^2
\leq \left( \gamma_2(\hat{Y}) + \gamma_2(Y) \right)^2
\leq 4 \gamma_2(Y)^2 .
\]
Using the triangle inequality again gives
\[
\| \mathcal{P}_E (\hat{Y} - Y) \|_F
\leq
\| \mathcal{P}_E (\hat{Y} - Z) \|_F
+
\| \mathcal{P}_E (Z - Y) \|_F
\leq 2 \delta \sqrt{|E|},
\]
taking into account the bound on the observation errors.
Because
\[
\frac{1}{nm} \| \hat{Y} - Y \|_F^2
\leq
\left| 
 \frac{1}{nm} \| \hat{Y} - Y \|_F^2 - 
 \frac{1}{|E|} \| \mathcal{P}_E (\hat{Y} - Y) \|_F^2 
\right| 
+
\frac{1}{|E|} \| \mathcal{P}_E (\hat{Y} - Y) \|_F^2 
\]
we get the final bound
\[
\frac{1}{n m} \| \hat{Y} - Y \|_F^2
\leq
4 K_G \gamma_2(Y)^2 \frac{\eta}{\sqrt{d_1 d_2}} + 4 \delta^2.
\]
\end{proof}

\subsection{Application of the spectral gap}

Theorem~\ref{thm:rect_MSE} provides a bound on the mean squared error
of the approximation $X$. 
Directly applying Theorem~\ref{cor:gap_A}, 
we obtain the following bound on 
the generalization error of the algorithm using a 
random biregular, bipartite graph:
\begin{corollary}
Let $E$ be sampled from a $\mathcal{G}(n,m, d_1, d_2)$ random graph.
For every $Y \in \R^{n \times m}$, if $\hat{Y}$ is the output of the optimization problem 
\eqref{matrix_completion}, then 
\[
	\frac{1}{n m} \| \hat{Y} - Y \|_F^2 \leq 
      c \gamma_2(Y)^2 \frac{\sqrt{d_1 - 1} + \sqrt{d_2 - 1} + \epsilon_n }{\sqrt{d_1 d_2}} ,
\]
where $c = 4 K_G \leq 7.13$ is a universal constant.
\end{corollary}

%%%%%%%%%%%%%%%%%%%%%%%%
%%% Acknowledgements %%%
%%%%%%%%%%%%%%%%%%%%%%%%

\section*{Acknowledgements}

We would like to thank Marina Meil\u{a} for sharing 
Proposition~\ref{thm:spurious_eig_bound} and for suggestions and comments.
Thank you also to Pierre Youssef, Simon Coste, and Subhabrata Sen for 
helpful comments and connections.
We are grateful to our anonymous reviewers for useful comments and suggestions, and at least in one case for pointing out errors 
in an earlier version of this manuscript (which have since been fixed).
K.D.H.\ was supported by the Big Data for Genomics and Neuroscience
NIH training grant, 
Washington Research Foundation postdoctoral fellowship,
as well as NSF grants DMS-1122105 and DMS-1514743.
G.B.\ was partially supported by NSF CAREER award DMS-1552267.
I.D.\ was supported by NSF DMS-1712630 and NSF CAREER award DMS-0847661.

%%%%%%%%%%%%%%%%
%%% Appendix %%%
%%%%%%%%%%%%%%%%

% \newpage
\appendix
\section{List of symbols}
\label{sec:list_of_symbols}

\begin{tabular}{cp{0.8\textwidth}}
$A$ & adjacency matrix\\
$B$ & non-backtracking matrix\\
$X^*$ & transpose of the matrix $X$\\
$\sigma(X)$ & the eigenvalue spectrum of a matrix $X$ \\
$\eta$ & second-largest eigenvalue of the adjacency matrix $A$\\
$\lambda_i (X)$ & the $i$th largest eigenvalue, in absolute value, of a matrix $X$\\ 
$\Gg$ & family of bipartite $d_1,d_2$-regular random graphs on $n,m$ vertices \\
$d$ & the maximum degree: $d = \max \{d_1, d_2\} =d_1$ without loss of generality \\
$V(G)$ & vertex set of graph or subgraph $G$ \\
$E(G)$ & edge set of a graph or subgraph $G$ \\
$\vec{E}(G)$ & oriented edge set of a graph or subgraph $G$\\
$\chi$ & tree excess of a graph or subgraph $G$: 
$\chi(G) = |V(G)| - |E(G)| + 1$ \\
$\gamma$ & a path \\
$\Gamma_{ef}^\ell$ & non-backtracking paths of length $\ell+1$ from oriented edge $e$ to $f$\\
$F_{ef}^\ell$ & 
non-backtracking, tangle-free paths of length $\ell+1$ from oriented edge $e$ to $f$\\
$T_{ef}^{\ell,j}$ & non-backtracking paths of length $\ell+1$, from $e$ to $f$, such that
the overall path is tangled but the first $j$ and last $\ell - j$ form tangle-free subpaths \\
$\| \cdot \|$ & $\ell^2$-norm of a vector or spectral norm of a matrix \\
$\| \cdot \|_F$ & Frobenius norm of a matrix
\end{tabular}

%%%%%%%%%%%%%%%%%%
%%% References %%%
%%%%%%%%%%%%%%%%%%

\bibliographystyle{plainnat}
\bibliography{library}
%\printbibliography

\end{document}